\documentclass[10pt]{article}
	
	\newcommand{\blind}{0}
	\date{}
    \makeatletter
    \renewcommand\section{\@startsection {section}{1}{\z@}%
                                       {-3.5ex \@plus -1ex \@minus -.2ex}%
                                       {2.3ex \@plus.2ex}%
                                       {\normalfont\fontfamily{phv}\fontsize{16}{19}\bfseries}}
    \renewcommand\subsection{\@startsection{subsection}{2}{\z@}%
                                         {-3.25ex\@plus -1ex \@minus -.2ex}%
                                         {1.5ex \@plus .2ex}%
                                         {\normalfont\fontfamily{phv}\fontsize{14}{17}\bfseries}}
    \renewcommand\subsubsection{\@startsection{subsubsection}{3}{\z@}%
                                        {-3.25ex\@plus -1ex \@minus -.2ex}%
                                         {1.5ex \@plus .2ex}%
                                         {\normalfont\normalsize\fontfamily{phv}\fontsize{14}{17}\selectfont}}
    \makeatother
	
	\usepackage{amsfonts}
	\usepackage[leqno]{amsmath}
        
	\usepackage{amssymb}
	\numberwithin{equation}{section}
	\usepackage{graphicx}
	\usepackage{enumerate}
	\usepackage{xcolor}
        \usepackage{bm}
	\usepackage{url} 
        \usepackage{amsthm}
        \usepackage{appendix}
        \newtheorem{proposition}{Proposition}[section]
        \newtheorem{corollary}[proposition]{Corollary}
        \newtheorem{theorem}{Theorem}[section]
        \newtheorem{lemma}[proposition]{Lemma}
        \newtheorem{definition}[proposition]{Definition}
        \newtheorem{remark}{Remark}
        \usepackage[colorlinks=true, linkcolor=blue, citecolor=blue, urlcolor=blue]{hyperref}
	
\begin{document}
		
		\def\spacingset#1{\renewcommand{\baselinestretch}%
			{#1}\small\normalsize} \spacingset{1}
		
		\if0\blind
		{
			\title{Construction of Quasi-periodic solutions with the same Gevrey index as nonlinear terms in Multi-Dimensional NLS}
			\author{Zuhong You $^a$ and Xiaoping Yuan $^a$ \\
			$^a$School of Mathematical Sciences, Fudan University, Shanghai 200433, P. R. China }
			\maketitle
   \renewcommand{\thefootnote}{\fnsymbol{footnote}}
   \footnotetext[1]{The work was supported  by National Natural Science Foundation of China
( Grant NO. 12371189).}
		} \fi
		
		\if1\blind
		{

            \title{\bf \emph{IISE Transactions} \LaTeX \ Template}
			\author{Author information is purposely removed for double-blind review}
			
\bigskip
			\bigskip
			\bigskip
			\begin{center}
				{\LARGE\bf \emph{IISE Transactions} \LaTeX \ Template}
			\end{center}
			\medskip
		} \fi
		\bigskip

	\spacingset{1} 

 \begin{minipage}{\textwidth}
    \begin{abstract}
        We investigate the persistency of quasi-periodic solutions to multi-dimensional nonlinear Schr\"{o}dinger equations (NLS)  involving Gevrey smooth nonlinearity with an arbitrary Gevrey index $\alpha>1$. 
        By applying the Craig-Wayne-Bourgain (CWB) method, we establish the existence of  quasi-periodic solutions that are Gevrey smooth with the same Gevrey index as the nonlinearity.
    \end{abstract}

    \vspace{0.2cm} 

    \noindent \textbf{Keywords:} Quasi-periodic solutions, Gevrey class,  NLS, Craig-Wayne-Bourgain method,  KAM
\end{minipage}

\tableofcontents

\section{Introduction}
In this paper, we consider the persistency of quasi-periodic solutions of the multi-dimensional nonlinear Schr\"{o}dinger equation (NLS):
\begin{equation}
    \frac{1}{i} q_{t}(\tilde{x},t)=-\Delta q(\tilde{x},t)+V(\tilde{x})*q(\tilde{x},t)+\varepsilon \frac{\partial H}{\partial \bar{q}}(q,\bar{q})
    \label{e1}
\end{equation}
under periodic boundary conditions (i.e., $\tilde{x}\in\mathbb{T}^{d}$). The convolution potential $V:\mathbb{T}^{d}\to \mathbb{C}$ has real Fourier coefficients $\hat{V}(n)$, $n\in\mathbb{Z}^{d}$. The Hamiltonian $H(q,\bar{q})=f(q\bar{q})$, where $f:\mathbb{R}\to \mathbb{R}$, belongs to the Gevrey class. In other words, we are intended to prove a Kolmogorov-Arnold-Moser (KAM) theorem for equation (\ref{e1}). 

When $d=1$, the KAM theory of Hamiltonian PDEs is  well established; see, for instance,  \cite{kuksin1987hamiltonian}, \cite{KP96}, \cite{wayne1990periodic}, \cite{Kuk93}, \cite{Pos96} and \cite{craig1993newton}. This equation  is intrinsically interesting, when the spatial dimension $d\ge 2$.   However,  the study of KAM theory for Hamiltonian PDEs with $d\ge 2$ remains in its early stages. For $d=2$, \cite{bourgain1998quasi}
 developed a new method initialed by \cite{craig1993newton} to obtain the Gevrey smooth quasi-periodic solutions of NLS with analytic nonlinearity. His method  relies on  Newton iteration, Fr\"{o}hlich-Spencer techniques (i.e., multi-scale analysis), and the Malgrange preparation theorem. Later, \cite{bourgain2005green} introduced harmonic analysis and semi-algebraic set theory to refine the techniques in \cite{bourgain1998quasi},  proving the persistency of quasi-periodic solutions of NLS and NLW with analytic nonlinearity for arbitrary $d$, and he claimed the obtained solutions  are analytic. However, he only gave a sketch of the proof, and there is a key coupling lemma seems unlikely to hold. We believe the sketch he described there can only obtain the Gevrey smooth solutions. This method is now referred to as  the Craig-Wayne-Bourgain (CWB) method.

 For the case of finitely differentiable nonlinearity, Berti, Bolle, Corsi, and Procesi among others, 
(\cite{BB10waveperiodic}, \cite{BBP10}, \cite{BP11duke}, \cite{BB12}, \cite{berti2013quasi}, \cite{BC15}, and \cite{BB20}, for example),   improved the Nash-Moser iterative scheme and developed a novel multi-scale inductive analysis.
They established the existence of finitely differentiable periodic or quasi-periodic solutions for NLS and NLW with finitely differentiable nonlinearity for arbitrary spatial dimension $d$.

In this paper, we aim to show that for (\ref{e1}) with Gevrey nonlinearity, most quasi-periodic solutions persist and \textbf{the persisted solutions have the same Gevrey index as the nonlinearity}.  To state our results, we begin with the definition of a Gevrey-$(\alpha, L)$ function.
 Let $\alpha>1, L>0$, and let $K$ be a compact set of the form $K'\times \mathbb{T}^{n_2}$, where $K'\subset \mathbb{R}^{n_1}$ is compact, and $n_1+n_2\ge 1$. We define 
 \begin{equation}
     G^{\alpha,L}(K)=\{\varphi\in C^{\infty}(K): \left\lVert \varphi \right\rVert_{\alpha,L}=\sum\limits_{k\in \mathbb{N}^{n_1+n_2}} \left( \frac{L^{|k|_{1}}}{k!}  
    \right)^{\alpha} \left\lVert \partial^{k}\varphi  \right\rVert_{C^{0}(K)}<\infty \},
  \label{fanshuaaa}
 \end{equation}
where $|k|_{1}=\sum\limits_{j=1}^{n_1+n_2} |k_j|$, $k!=\prod\limits_{j=1}^{n_1+n_2}k_{j}!$, and $\left\lVert \cdot \right\rVert_{C^{0}(K)}$ denotes the uniform norm for continuous functions defined on $K$. 
A function  $\varphi\in G^{\alpha,L}(K)$ is called a  Gevrey-$(\alpha, L)$ function. We refer to $\alpha$ as the Gevrey index and $L$ as the Gevrey width of $\varphi$. 
Furthermore, we  define
\begin{equation}
    G^{\alpha}(K)=\mathop{\cup}\limits_{L>0} G^{\alpha,L}(K),
\end{equation}
and  functions in $G^{\alpha}(K)$ are referred  as  Gevrey-$\alpha$ functions.  
We assume $f,f^{\prime}, f''\in G^{\alpha,L_0}([-1,1])$ for some $\alpha>1$ and $ L_0>0$.

One says $\tilde{q}(\tilde{x},t)$ is a quasi-periodic function in time $t$ if there exist $\lambda\in \mathbb{R}^{b}$ and $q(\tilde{x},\theta)\in L^{2}(\mathbb{T}^{d+b})\,$ (here $\tilde{x}\in \mathbb{T}^{d}, \theta\in\mathbb{T}^{b}$) such that 
$$
\tilde{q}(\tilde{x},t)=q(\tilde{x},\lambda t).
$$
We say that $\lambda$ is the frequency and $q(\tilde{x},\theta)$ is the hull of $\tilde{q}(\tilde{x},t)$. See \cite{KP96} for the definitions of quasi-periodic function and its hull, for example.

Select  finite Fourier modes $0\in \{n_1, n_2, ..., n_b\}\subset\mathbb{Z}^{d}$. Denote 
$\mathcal{R}=\{ (n_j,e_j):1\le j\le b \}$, where $e_j$ is the $j$-th unit vector in $\mathbb{Z}^{b}$. Let $\textup{supp } \hat{V}=\{n_1,..., n_b\}$. Then, the spectral of the operator $\mathcal{L}:=-\Delta+V(\tilde{x})*$ are 
\begin{equation}
    \begin{cases}
        \mu_{n_j}=\lambda_{j},\ \ &(1\le j\le b),\\
        \mu_{n}=\left\lvert n \right\rvert^{2} \ \ &n\in\mathbb{Z}^d\setminus\{n_1,...,n_b\},
    \end{cases}
    \label{e2}
\end{equation}
where $|n|^{2}=\sum\limits_{j=1}^{d}|n_j|^2$ and $\lambda:=(\lambda_j=|n_j|^{2}+\hat{V}(n_j):1\le j \le b)$.
Given $a_1,..., a_b\in\mathbb{R}_{+}$, 
\begin{equation}
    \tilde{q}_0(\tilde{x},t)=q_{0}(\tilde{x}, \lambda t)=\sum\limits_{j=1}^{b}a_j e^{i\lambda_{j}t+in_j\cdot \tilde{x}}
    \label{e3}
\end{equation}
yields a solution to the linear equation
\begin{equation}
    \frac{1}{i} q_{t}(\tilde{x},t)=-\Delta q(\tilde{x},t)+V(\tilde{x})*q(\tilde{x},t).
    \label{e4}
\end{equation}
This solution is quasi-periodic and corresponds to a flow on a $b$-torus 
$$\mathbb{T}^{b}=\mathcal{T}_0 \{|a_{j}|\}=\left\{\sum\limits_{j=1}^{b}z_je^{in_{j}\cdot \tilde{x}}:z_j\in\mathbb{C},|z_j|=a_j\right\}.$$
The problem we study is the ``persistency" of solutions (\ref{e3}) for the Gevrey perturbed equation (\ref{e1}). The term ``persistency'' in this paper is defined as follows:

There exists a quasi-periodic solution
\begin{equation}
    \tilde{q}(\tilde{x},t)=q(\tilde{x},\lambda't)
    \notag
\end{equation}
to (\ref{e1}), such that 
\begin{equation}
    |\lambda'-\lambda|<C \varepsilon,
    \notag
\end{equation}
and
\begin{equation}
    \lVert q-q_{0} \rVert_{\alpha, L'}<C \varepsilon,
    \notag
\end{equation}
for some  $L'>0$.
Here,  $C$ is a constant depending only on $b, d, L_0, \alpha$ and the bound of $\lambda$.  

Note that the Gevrey index $\alpha>1$ of the solutions here is the same as the Gevrey index of nonlinearity, and $\alpha>1$ is arbitrary. In previous papers, the persisted solutions only have Gevrey index $\alpha\gg 1$, even if the nonlinearity is polynomial.

In KAM theory, the vector $(\hat{V}(n_j):1\le j \le b)$ is typically considered  as a parameter vector. This is equivalent to  regarding $\lambda=(\lambda_{1},...,\lambda_{b})$ in the same way. Thus, we take $\lambda\in I_{0}=[0,1]^{b}$ as parameters. \\

Here is our main theorem.
\begin{theorem}
For arbitrary $\alpha>1$,
let $H(q,\bar{q})=f(q\bar{q})$, $f,f',f''\in G^{\alpha,L_0}([-1,1])$ $( L_0>0)$.  Fix $a_{1},...,a_{b}\in\mathbb{R}_{+}$ and let $I_0=[0,1]^{b}$. For $\varepsilon$ sufficiently small, there exists a Cantor set $I_{0,\varepsilon}\subset I_0$ such that $mes(I_0\setminus I_{0,\varepsilon}) \to 0$  as $\varepsilon\to 0$.   For $\lambda\in I_{0,\varepsilon}$, there exists a quasi-periodic solution to (\ref{e1}) with frequency $\lambda'$ (depending on $\lambda$), whose hull $q(\tilde{x},\theta)$ is in $G^{\alpha,L_0/2}(\mathbb{T}^{d+b})$.
Moreover, we have
\begin{equation}
    \hat{q}(n_j,e_j)=a_j, \ \ 1\le j\le b,
\end{equation}
 \begin{equation}
      |\lambda'-\lambda|\le C \varepsilon,
 \end{equation}
 and
 \begin{equation}
     \sum\limits_{\substack{(n,k)\in\mathbb{Z}^{d+b}\\(n,k)\notin \mathcal{R}}}|\hat{q}(n,k)|e^{ \frac{\alpha L_0}{2}|(n,k)|_{\alpha}}\le C \varepsilon,
 \end{equation}
 which implies
 \begin{equation}
     \lVert q-q_{0} \rVert_{\alpha,\frac{L_{0}}{2}}<C\varepsilon,
 \end{equation}
where $\hat{q}(n,k)$ is the $(n,k)$-Fourier coefficient of $q(\tilde{x},\theta)$, $q_{0}(\tilde{x},\theta)=\sum\limits_{j=1}^{b}a_{j}e^{i\theta_{j}+i n_{j}\cdot \tilde{x}}$,  and $C$ is a constant depending  on $\alpha, L_0, I_0, b$ and $d$ only. 

Here, the term $|(n,k)|_{\alpha}=\sum\limits_{j=1}^{d}|n_j|^{1/\alpha}+\sum\limits_{j=1}^{d} |k_j|^{1/\alpha}$, and  the weight function $e^{\alpha L' |(n,k)|_{\alpha}}$ corresponds to Gevrey-$(\alpha,L')$ regularity.
 \label{zhudingli}
\end{theorem}

\textbf{The following are some key highlights that we would like to emphasize:}

In this paper, we propose a new approach based on the Craig-Wayne-Bourgain (CWB) method to handle Gevrey smooth nonlinear perturbations.
Unlike earlier works, which 
approximated the Gevrey Hamiltonian by
analytic ones (see \cite{Pop04} for example), we directly employ the Gevrey norm introduced in \cite{marco2003stability} (i.e., (\ref{fanshuaaa})).  \textbf{We find that the Gevrey norm  possesses an almost equivalent norm described by  Fourier coefficients, allowing us to use it  similarly to the sup-norm of real analytic functions in most cases}. 

To apply the semi-algebraic techniques, we explored the relationship between the rate of polynomial approximation of Gevrey functions in Gevrey norm and the domain,  and \textbf{developed an effective strategy for the polynomial approximation of Gevrey funtions}. 

During the Newton iteration, we need to estimate the Gevrey norm of the approximate solution $q(\lambda,\lambda')$.  \textbf{We apply the Markov Inequality to  obtain the  Gevrey norm estimate of  $q(\lambda,\lambda')$ only from the sup-norm of $T_{r-1}^{-1}(\eta, \eta')$ and $F(q_{r})$ on a real domain}.

In addition, as we have mentioned in the beginning, Bourgain only provided an outline of the proof of the persistency of quasi-periodic solutions for NLS and NLW using the CWB method in Chapters 19 and 20 of \cite{bourgain2005green}. There is a key coupling lemma in his argument, but he did not give a proof for this coupling lemma. The coupling lemma he claimed seems unlikely to hold,
 as noted by \cite{wang2016energy}, ``This is because the sizes of the resonant clusters are in general, much larger than their separation."
 \textbf{In our paper, we   exploit  the strict concavity of $g(s)=s^{1/\alpha}$ to improve Lemma 7 of \cite{bourgain1998quasi} and obtain our Lemma 6.1. In Lemma 7 of \cite{bourgain1998quasi}, the loss of the decay rate is $\frac{1}{2}L$, while  the loss of the decay rate  in  Lemma \ref{couplelemma1}  can be arbitrarily small. This improvement helps us to get solutions with arbitrary Gevrey index $\alpha>1$}.

Moreover, when applying the CWB method, the reality of the drifted frequency $\lambda'$ is straightforward when the nonlinearity is polynomial or real analytic. However,  for Gevrey smooth nonlinearity,  the reality of the drifted frequency $\lambda'$ seems not to be obvious. We provide a proposition in Appendix A to clarify this issue.

Here are more remarks.

\begin{remark}
    It seems that the scheme in the present paper could apply to NLW by combining with Lemma 20.14 in \cite{bourgain2005green}. 
\end{remark}

\begin{remark}
    Typically, when $d=1$, the second Melnikov conditions hold, and the solutions obtained from the classical KAM techniques are linearly stable. However, the  linear stability of   solutions obtained from CWB method is unknown.  In 2010, \cite{eliasson2010kam} applied classical KAM techniques to prove the KAM theorem for NLS with analytic nonlinearity for arbitrary $d$, obtaining  KAM tori that  are analytic in time and $C^{\infty}$ in space. For $d\ge 1$,   Yuan\cite{yuan2021kam}  proved the KAM theorem for the generalized Pochhammer-Chree  equations  using the traditional KAM techniques,  and the obtained KAM tori  are also  linearly stable.    
    Recently, He, J. Shi, Y. Shi and Yuan \cite{he2024linear} proved the  the  linear stability of the KAM tori without the second Melnikov conditions by combining the traditional KAM techniques and the CWB method. Thus,  it seems plausible that the solutions here are also  linearly stable. 
\end{remark}

\begin{remark}
If we apply the amplitude-frequency modulation, as in \cite{bourgain1995construction}, \cite{KP96},  \cite{PMPC12}, \cite{Wang2019}, \cite{BB20}, etc, we could replace the convolution potential $V(\tilde{x})*$ by the multiplicative potential $V(\tilde{x})\equiv constant$. However, for the simplicity of the argument, we do not pursue this in the present work.
\end{remark}

\begin{remark}
    For more history and applications of the CWB method to Hamiltonian PDEs, we refer to \cite{craig1993newton}, \cite{bourgain1994construction},\cite{bourgain1995construction}, \cite{bourgain1997melnikov}, \cite{bourgain1998quasi},  \cite{wang2016energy}, \cite{Wang2019}, \cite{Wang2020}.
\end{remark}
  
\textbf{Notations.}
For $a,b\in \mathbb{R}^{d}$, we denote by $a\cdot b$ the standard product in $R^{d}$. The symbol $\left\lVert \cdot \right\rVert$ denotes $l^{2}$ norm, $L^{2}$ norm or the operator norm corresponding to these norms.

Let $K$ be a compact set, and $\varphi\in C(K)$. We denote the uniform norm of $\varphi$ by $\left\lVert \varphi\right\rVert_{C^{0}(K)}$.

Let $C^{k}(K)$ be the set of $k$-times continuously differentiable complex-valued functions.

We set $0^0=1$, $\mathbb{N}=\{0,1,2,...\}$, $\mathbb{R}_{+}=\{s\in\mathbb{R}:s>0\}$. 

Given an integer $n\ge 1$, $\alpha>1$, $k=(k_{1},...,k_n)\in\mathbb{Z}^{n}$ and $\bm{L}=(L_1,L_2,...,L_n)\in\mathbb{R}^n$, we define
$$
|k|=\max\limits_{i=1,...,n} \{|k_{i}|\},\ \ |L|=\max\limits_{i=1,...,n} \{|L_{i}|\}, \ \ |k|_{\alpha}=\sum\limits_{i=1}^{n} |k_{i}|^{\frac{1}{\alpha}},\ \ \bm{L}^{k}=\prod\limits_{i=1}^{n} |L_i|^{k_i}.
$$

Let $\{x_{1}, x_{2},...,x_{k}\}\subset \mathbb{R}^{n}$. We denote by $[x_{1}, x_{2},..., x_{k}]$ the subspace spanned by these vectors.

If we use $k_1$ to denote an element in $\mathbb{N}^n$, we  use $k_1(j)$ to denote the j-th component of $k_1$, and the same convention applies to elements in $\mathbb{R}^n$.
We write $\bm{L}<\bm{L}'$ to indicate that $\bm{L}(j)<\bm{L}'(j)\textup{ for } 1\le j\le n$.

When we say $I\in\mathbb{R}^{n}$ is an interval of size $d$, we mean that $I$ is a closed cube with sides parallel to the axes, and each side has  length  $d$.  Two intervals are said to be disjoint if their interiors are disjoint.

Let $I$ be an interval. We denote by $CI$ the interval with the same center as $I$ but  expanded by the factor $C$.

Whenever we encounter a situation where a specific integer cannot be found, we will choose the closest integer.

For any positive numbers $a, b$ the notation $a\lesssim b$ means $Ca\le b$ for some constant $C>0$ depending only on $\alpha$, $L_0$, $I_0$, $b$ and $d$. By $a\ll b$ we mean that the constant $C$ is very large.  The various constants can be defined by the context in which they arise. Finally, $N^{a-}$ means $N^{a-\epsilon}$ with some small $\epsilon>0$ (the precise meaning of ``small" can again be derived from the context).

\textbf{This paper is organized as follows:} In Section 2, we provide the definitions of Gevrey functions, the Gevrey norm, and the properties of the Gevrey norm. Section 3 and Section 4  describe the Lyapounov-Schimidt decomposition and the Newton iteration scheme. In Section 5, we state  the Iterative Lemma and provide its proof, except for the measure estimates. Notably, we demonstrate that the estimate of the Gevrey norm can be obtained at each step by applying  polynomial truncation and the Markov Inequality. In Section 6, we present the proof of the   measure estimates in the Iterative Lemma.  In Section 7, we give the proof of  Theorem \ref{zhudingli}. Finally,   Appendices A and B present two technical propositions, Appendix C states the matrix-valued Cartan-type theorem,  Appendix D states some facts about semialgebraic sets, Appendix E presents the multi-scale reasoning for the large deviation theorem.

\section{Gevrey functions and their properties}

In this section, we provide the definition of  Gevrey functions and   some technical lemmas. 
In the Introduction, we have introduced the definition of a Gevrey-$(\alpha,L)$ function for $L>0$.
Here, we extend the definition by defining the Gevrey norm with  different Gevrey widths. Let $\bm{L}\in\mathbb{R}_{+}^{n}$ and let $K$ be a compact set in $n$-dimensional space. For $\varphi\in C^{\infty}(K)$, the Gevrey norm is defined as follows:
\begin{equation}
    \left\lVert \varphi \right\rVert_{\alpha,\bm{L}}=\sum_{k\in\mathbb{N}^{n}}\left(\frac{\bm{L}^{k}}{k!}\right)^{\alpha}\left\lVert \partial^{k}\varphi \right\rVert_{C^{0}(K)}.
    \label{AA3.2}
\end{equation}
If $\left\lVert \varphi \right\rVert_{\alpha,\bm{L}}<\infty$,  we say that $\varphi$ is Gevrey-$(\alpha,\bm{L})$ on $K$, denoted by $\varphi\in G^{\alpha,\bm{L}}(K)$. The vector $\bm{L}$ is referred to as the Gevrey width of $\varphi$. We also have
$$
G^{\alpha}(K)=\bigcup\limits_{\bm{L}\in\mathbb{R}_{+}^n} G^{\alpha,\bm{L}}(K).
$$

In this paper, we frequently employ two distinct Gevrey widths. For clarity, we restate the definition in this context.

Let $n_1,n_2\in \mathbb{N}$ and $ n=n_1+n_2$. Let $K_1$ be a compact interval in $\mathbb{R}^{n_1}$  (corresponding to parameters)  and $K_2=\mathbb{T}^{n_2}$. Let $K=K_1\times K_2$ and  $ L_1,L_2>0$. For a function $\varphi(\lambda,\theta)\in C^{\infty}(K)$, where $\lambda\in K_1$ and $\theta\in K_2$,  the Gevrey norm is defined as
\begin{equation}
    \left\lVert \varphi \right\rVert_{\alpha,L_1,L_2}=\sum\limits_{k_1\in \mathbb{N}^{n_1},k_2\in\mathbb{N}^{n_2}} \left(\frac{L_1^{|k_1|_1}L_2^{|k_2|_1}}{k_1! k_2!}\right)^{\alpha}\left\lVert \partial_{\lambda}^{k_1}\partial_{\theta}^{k_2}\varphi\right\rVert_{C^{0}(K)}.
\end{equation}
If $\left\lVert \varphi\right\rVert_{\alpha,L_1,L_2}< \infty$,  we say that $\varphi(\lambda,\theta)$ is Gevrey-$(\alpha,L_1,L_2)$ on $K$, denoted by $\varphi\in G^{\alpha,L_1,L_2}(K)$. The parameters $L_1$ and $L_2$ are referred to the Gevrey widths  with respect to $\lambda$ and $\theta$, respectively.

\begin{proposition}[\textbf{Product}]
Let $\bm{L}\in\mathbb{R}_{+}^{n}$ and  $\varphi,\psi\in G^{\alpha,\bm{L}}(K)$, then we have 

$$\left\lVert \varphi \psi \right\rVert_{\alpha,\bm{L}}\le \left\lVert \varphi \right\rVert_{\alpha,\bm{L}}\left\lVert \psi \right\rVert_{\alpha,\bm{L}}. $$
    In particular, for $L_1,L_2>0$ and $\varphi,\psi\in G^{\alpha,L_1,L_2}(K)$, we have
    $$\left\lVert \varphi \psi \right\rVert_{\alpha,L_1,L_2}\le \left\lVert \varphi \right\rVert_{\alpha,L_1,L_2}\left\lVert \psi \right\rVert_{\alpha,L_1,L_2}. $$
    \label{A1}
\end{proposition}
This proposition shows that $G^{\alpha,\bm{L}}$ is a Banach algebra.

\begin{proposition}[\textbf{Derivative}]
    Suppose $0\le\bm{h}<\bm{L}$ and $\varphi\in G^{\alpha,\bm{L}}(K)$. Then all partial derivatives of $\varphi$ belong to $G^{\alpha,\bm{L}-\bm{h}}(K)$, and we have 
    $$
    \sum\limits_{m\in\mathbb{N}^{n}}\frac{\bm{h}^{m\alpha}}{m!^{\alpha}}\left\lVert\partial^{m}\varphi\right\rVert_{\alpha,\bm{L-h}}\le \left\lVert \varphi \right\rVert_{\alpha,\bm{L}}.
    $$
In particular, for all $m\in\mathbb{N}^{n}$,
$$
\left\lVert \partial^{m}\varphi \right\rVert_{\alpha,\bm{L-h}}\le \left( \frac{m!}{\bm{h}^{m}} \right)^{\alpha}\left\lVert \varphi \right\rVert_{\alpha,\bm{L}}.
$$
\label{Ap3.2}
\end{proposition}
This result is often referred to as the ``Cauchy estimate" for Gevrey functions.

\begin{proposition}[\textbf{Composition}]
Let $L>0$, $\bm{L}\in \mathbb{R}_{+}^{m}$, and let $K$ be an n-dimensional compact set, and $K_1$ be an m-dimensional compact set as previously defined. Let $Y\in G^{\alpha,L}(K)$,  and let $u=(u_1,...,u_n)$ be a mapping whose components belong to $G^{\alpha,\bm{L}}(K_1)$. If $u(K_1)\subset K$ and 
\begin{equation}
    \left\lVert u_j \rVert_{\alpha,\bm{L}}-\lVert u_j \right\rVert_{C^0(K_1)}\le \frac{L^\alpha}{m^{\alpha-1}},\ \ 1\le j\le n,
    \label{co1}
\end{equation}
then $Y\circ u\in G^{\alpha,\bm{L}}(K_1)$ and $\left\lVert Y\circ u \right\rVert_{\alpha,\bm{L}}\le \left\lVert Y \right\rVert_{\alpha,L}$.
\label{composition}
\end{proposition}
The proofs of Proposition \ref{A1}, \ref{Ap3.2} and \ref{composition} follow  the same approaches as in \cite{marco2003stability}, except for the allowance of different Gevrey widths in this context. For detailed proofs, we refer to \cite{marco2003stability}.

We now discuss the decay of Fourier coefficients of Gevrey functions. 
Let $\alpha>1$ and $ K=\mathbb{T}^{d}$.  Lemma 4.8 in \cite{dias2017renormalization} shows that if $f\in G^{\alpha,L}(\mathbb{T}^{d})$, then $|\hat{f}(k)|\le C \left\lVert f \right\rVert_{\alpha,L}e^{-L|k|_{1}^{1/\alpha}}$, where $C$ is a constant independent of $\alpha$ and $L$. However, with a minor modification to the proof, we can deduce a more intrinsic description:

\begin{proposition}[\textbf{Decay of Fourier coefficients}]
    Let $\alpha>1$ and $L>0$. If $f\in G^{\alpha,L}(\mathbb{T}^{d})$, then for any $k\in\mathbb{Z}^{d}$, we have
    \begin{equation}
        |\hat{f}(k)|\le \left(\frac{h(\varepsilon,\alpha)}{2\pi}\right)^{d}\left\lVert f \right\rVert_{\alpha,L} e^{-\alpha(1-\varepsilon)L|k|_{\alpha}},
    \end{equation}
where $h(\varepsilon,\alpha)=\left(\frac{1}{1-(1-\varepsilon)^{\frac{\alpha}{\alpha-1}}}\right)^{\alpha-1}$.
\label{pfc}
\end{proposition}

\begin{proof}
 Since for any $\beta\in\mathbb{N}^{d}$, we have $|(\partial^{\beta}f)^{\land}(k)|\le \frac{1}{(2\pi)^{d}}\left\lVert \partial^{\beta} f \right\rVert_{C^{0}(K)}$, \  it follows that
    \begin{equation}(2\pi)^{d}\sum\limits_{\beta\in\mathbb{N}^{d}}\frac{L^{|\beta|_{1}\alpha}}{\beta!^{\alpha}}|(\partial^{\beta }f)^{\land}(k)|\le
       \sum\limits_{\beta\in\mathbb{N}^{d}}\frac{L^{|\beta|_{1}\alpha}}{\beta!^{\alpha}}\left\lVert \partial^{\beta} f \right\rVert_{C^{0}(K)}=
        \left\lVert f \right\rVert_{\alpha,L}.
    \end{equation}
Thus, since $(\partial^{\beta}f)^{\land}(k)=(ik)^{\beta}\hat{f}(k)$, we obtain 
$$(2\pi)^{d}|\hat{f}(k)|\sum\limits_{\beta\in\mathbb{N}^{d}}\frac{L^{|\beta|_{1}\alpha}}{\beta!^{\alpha}}k^{\beta}\le \left\lVert f \right\rVert_{\alpha,L}.$$
Therefore, it suffices to estimate $$\sum\limits_{\beta\in\mathbb{N}^{d}}\frac{L^{|\beta|_{1}\alpha}}{\beta!^{\alpha}}k^{\beta}.$$
Notice that $$\sum\limits_{\beta\in\mathbb{N}^{d}}\frac{L^{|\beta|_{1}\alpha}}{\beta!^{\alpha}}k^{\beta}=\prod\limits_{j=1}^{d}\sum\limits_{n=0}^{\infty}\frac{L^{n\alpha}}{n!^{\alpha}}|k_j|^n.$$ 
We now consider $$\sum\limits_{n=0}^{\infty}\frac{L^{n\alpha}}{n!^{\alpha}}|k_j|^n=\sum\limits_{n=0}^{\infty}\left( \frac{(L|k_j|^{1/\alpha})^n}{n!} \right)^{\alpha}.$$
By H\"{o}lder inequality, we have 
    \begin{equation}
        \left( \sum\limits_{n=0}^{\infty} x_n y_n \right)^{\alpha}\le \left( \sum\limits_{n=0}^{\infty} y_n^t \right)^{\frac{\alpha}{t}}\sum\limits_{n=0}^{\infty}x_n^{\alpha},
        \nonumber
    \end{equation}
    where $\frac{1}{\alpha}+\frac{1}{t}=1$. Taking $x_n=\frac{(L|k_j|^{1/\alpha})^n}{n!}$ and $ y_n=(1-\varepsilon)^n$, we have
    \begin{equation}
        \left(\sum\limits_{n=0}^{\infty} \frac{((1-\varepsilon)L|k_j|^{1/\alpha})^n}{n!} \right)^{\alpha}
        \le \left( \sum\limits_{n=0}^{\infty} (1-\varepsilon)^{tn} \right)^{\alpha/t} 
        \sum\limits_{n=0}^{\infty}\frac{L^{n\alpha}}{n!^{\alpha}}|k_j|^n.
        \label{fc1}
    \end{equation}
    Denote $h(\varepsilon,\alpha)=\left( \sum\limits_{n=0}^{\infty} (1-\varepsilon)^{tn} \right)^{\alpha/t} =\left(\frac{1}{1-(1-\varepsilon)^{\frac{\alpha}{\alpha-1}}}\right)^{\alpha-1}$, then the inequality (\ref{fc1}) can be written as 
    \begin{equation}
        \exp(\alpha (1-\varepsilon)L|k_j|^{1/\alpha})\le h(\varepsilon,\alpha)\sum\limits_{n=0}^{\infty}\frac{L^{n\alpha}}{n!^{\alpha}}|k_j|^n.
        \label{fc2}
    \end{equation}
    Thus, we can obtain
    \begin{equation}
        \begin{split}
            |\hat{f}(k)|
            &\le \frac{1}{(2\pi)^{d}}\left\lVert f \right\rVert_{\alpha,L} \frac{1}{\sum\limits_{\beta\in\mathbb{N}^{d}}\frac{L^{|\beta|_{1}\alpha}}{\beta!^{\alpha}}}\\
            &=\frac{1}{(2\pi)^{d}}\left\lVert f \right\rVert_{\alpha,L}\frac{1}{\prod\limits_{j=1}^{d}\sum\limits_{n=0}^{\infty}\frac{L^{n\alpha}}{n!^{\alpha}}|k_j|^n}\\
            &\le \left(\frac{h(\varepsilon,\alpha)}{2\pi}\right)^{d}\left\lVert f \right\rVert_{\alpha,L} \exp\left(-\alpha(1-\varepsilon)L\sum\limits_{j=1}^{d}|k_j|^{1/\alpha}\right). \qedhere
        \end{split}
    \end{equation}
    \end{proof}

 Proposition \ref{pfc} tells us that the $k$-Fourier coefficients of Gevrey functions exhibit sub-exponential decay with rate $\alpha (L-\varepsilon)|k|_{\alpha}$, corresponding to the Gevrey width. On the other hand, if  the Fourier coefficients of a function $f$ exhibit sub-exponential decay, then $f$ is Gevrey function. In the following, we provide a precise description.

We first consider the Gevrey norm of $e^{ik\cdot x}$.
\begin{lemma}
   Let $\alpha>1$ and $L>0$, then we have
    $$\left\lVert e^{ik\cdot x} \right\rVert_{\alpha,L}\le e^{\alpha L|k|_{\alpha}}$$.
    \label{leg}
\end{lemma}
\begin{proof}
By the definition of $\lVert \cdot \rVert_{\alpha,L}$, we have
\begin{equation}
    \left\lVert e^{ik\cdot x} \right\rVert_{\alpha,L} 
            =\sum\limits_{\beta\in\mathbb{N}^{d}} \frac{L^{|\beta|_{1}\alpha}}{\beta!^{\alpha}} \prod\limits_{j=1}^{d}|k_j|^{\beta_j}=\sum\limits_{\beta\in\mathbb{N}^{d}}\prod\limits_{j=1}^{d}|k_j|^{\beta_{j}}\frac{L^{\beta_j \alpha}}{\beta_j!^{\alpha}}.
            \notag
\end{equation}
This simplifies to
\begin{equation}
    \begin{split}
    \sum\limits_{\beta\in\mathbb{N}^{d}}\prod\limits_{j=1}^{d}|k_j|^{\beta_{j}}\frac{L^{\beta_j \alpha}}{\beta_j!^{\alpha}}
            &=\prod\limits_{j=1}^{d}\sum\limits_{n=0}^{\infty}\left( \frac{|k_j|^{n/\alpha}}{n!}L^n \right)^{\alpha}\\
            &\le \left( \prod\limits_{j=1}^{d}\sum\limits_{n=0}^{\infty}\frac{|k_j|^{n/\alpha}}{n!}L^n \right)^{\alpha}\\
            &=\exp(\alpha L\sum\limits_{j=1}^{d}|k_j|^{1/\alpha}).\qedhere
    \end{split}
    \notag
\end{equation}
\end{proof}
 
Therefore, we can naturally define a norm for $f$ defined on $\mathbb{T}^{d}$ as follows:
\begin{equation}
    \left\lVert f \right\rVert_{\mathcal{F}_{\alpha,L}}:=\sum\limits_{k\in \mathbb{Z}^{d}} e^{\alpha L|k|_{\alpha}}|\hat{f}(k)|.
\end{equation}
We denote by $\mathcal{F}_{\alpha,L}$ the set of smooth functions $f\in C^{\infty}(\mathbb{T}^{d})$ such that $\left\lVert f \right\rVert_{\mathcal{F}_{\alpha,L}} < \infty$. The two norms we have defined are related as follows:

\begin{proposition}
    Let $f\in C^{\infty}(\mathbb{T}^{d})$. Then:
    \begin{itemize}
        \item[\textup{(1)}] $\left\lVert f \right\rVert_{\alpha,L}  \le  \left\lVert f \right\rVert_{\mathcal{F}_{\alpha,L}}$.
        \item[\textup{(2)}] $\left\lVert f \right\rVert_{\mathcal{F}_{\alpha,(1-\varepsilon)L-\delta}} \le \left( \frac{h(\varepsilon,\alpha)}{2\pi} \right)^{d} C_{\delta} \left\lVert f \right\rVert_{\alpha,L}$, where $C_{\delta}=\sum\limits_{k\in\mathbb{Z}^d} e^{-\alpha \delta|k|_{\alpha}} \lesssim (\alpha\delta)^{-\alpha}$, and $h(\varepsilon,\alpha )$ is as defined above.
    \end{itemize}
    \label{pe}
\end{proposition}

\begin{proof}
    Part \textup{(1)} is straightforward by Lemma \ref{leg}. 
    For part \textup{(2)}, we have 
    \begin{equation}
    \begin{split}
        \left\lVert f \right\rVert_{\mathcal{F}_{\alpha,(1-\varepsilon)L-\delta}}
        &=\sum\limits_{k\in \mathbb{Z}^{d}}\exp(\alpha ((1-\varepsilon)L-\delta) |k|_{\alpha})|\hat{f}(k)|  \\
        &\le 
        \left(\frac{h(\varepsilon,\alpha)}{2\pi}\right)^{d}\left\lVert f \right\rVert_{\alpha,L} 
        \sum\limits_{k\in \mathbb{Z}^{d}} e^{-\alpha\delta|k|_{\alpha}}\\
        &=\left( \frac{h(\varepsilon,\alpha)}{2\pi} \right)^{d} C_{\delta} \left\lVert f \right\rVert_{\alpha,L}. \qedhere
    \end{split}
    \end{equation}
\end{proof}

 We can express part (2) of Proposition $\ref{pe}$ in a more concise form as follows:

 \begin{corollary}
     There exists a constant $C_{\alpha,d}$ that  depends only  on $\alpha$ and $d$ such that 
     \begin{equation}
         \left\lVert f \right\rVert_{\mathcal{F}_{\alpha,(1-\varepsilon)L-\delta}} \le C_{\alpha,d} \varepsilon^{-(\alpha-1)d}\delta^{-\alpha d} \left\lVert f \right\rVert_{\alpha,L}.
     \end{equation}
     \label{coroA.7}
 \end{corollary}
 \begin{proof}
     It suffices to notice that
      \begin{equation}
        \begin{split}
            h(\varepsilon,\alpha)
            &=\left(\frac{1}{1-(1-\epsilon)^{\frac{\alpha}{\alpha-1}}}\right)^{\alpha-1}\\
            &=\left(\frac{1}{1-(1-\frac{\alpha}{\alpha-1}\varepsilon+O(\varepsilon^2))}\right)^{\alpha-1}\\
            &=\left( \frac{1}{\frac{\alpha}{\alpha-1}\varepsilon+O(\varepsilon^2)} \right)^{\alpha-1}.  \qedhere 
        \end{split}
    \end{equation}
 \end{proof}

Next, we discuss truncation techniques in Gevrey space. 

\begin{definition}
    Let $I\subset \mathbb{N}^n$ be a closed interval, and let $\lambda_{0}\in I$. Let $V$ be a Banach algebra with norm  denoted by $\left\lVert \cdot \right\rVert_{V}$, and let $\varphi: I\to V$ be a $V$-valued map. We say that $\varphi$ admits a Gevrey-$(\alpha,L)$ asymptotic expansion at $\lambda_{0}$ if there exists a formal series 
    \begin{equation}
        \widetilde{\varphi}(\lambda)=\sum\limits_{k\in\mathbb{N}^{n}}(\lambda-\lambda_0)^{k}\varphi_{k},
        \nonumber
    \end{equation}
    where $\varphi_k\in V$ for each $k\in\mathbb{N}^{n}$, and a constant $C$ such that, for all $\lambda\in I$ and for all $p>0$,
    \begin{equation}
        \left\lVert \varphi(\lambda)-\sum_{|k|<p}  (\lambda-\lambda_0)^{k}\varphi_k\right\rVert_{V} \le C p!^{\alpha-1}L^{-p\alpha}|\lambda-\lambda_0|_{1}^{p}.
        \nonumber
    \end{equation}
\end{definition}

This definition differs slightly from the traditional definition of Gevrey asymptotic expansions, as it involves multiple variables. In fact, this can be viewed as a definition of a Gevrey-type $V$-valued function. 
For a more usual framework of Gevrey asymptotic expansions, the reader is referred to \cite{balser1994divergent},\cite{ramis1996gevrey}, and \cite{tougeron1990introduction}.

\begin{lemma}
    Let $K_1\subset{\mathbb{R}^n}$ be a closed interval, and let $K_2$ be an m-dimensional compact set. Let $\alpha>1$ and $ L_1,L_2>0$. If $\varphi(\lambda,x)\in G^{\alpha,L_1,L_2}(K_1\times K_2)$, then we have 
    \begin{equation}
        \forall k_1\in\mathbb{N}^{n},\ \ \sup\limits_{\lambda\in K_1} \left\lVert \partial_{\lambda}^{k_1}\varphi(\lambda,\cdot) \right\rVert_{\alpha,L_2} \le \left(\frac{k_1!}{L_1^{|k_1|_1}}\right)^{\alpha}\left\lVert \varphi \right\rVert_{\alpha,L_1,L_2}.
    \end{equation}
    \label{supg}
\end{lemma}
\begin{proof}
    Fix $\lambda\in K_1$, we have
    \begin{equation}
        \begin{split}
            \left\lVert \partial_{\lambda}^{k_1}\varphi(\lambda,\cdot) \right\rVert_{\alpha,L_2}
            &=\sum\limits_{k_2\in\mathbb{N}^{m}}\left(\frac{L_2^{|k_2|_1}}{k_2!}\right)^{\alpha}\left\lVert \partial_{x}^{k_2}\partial_{\lambda}^{k_1} \varphi\right\rVert_{C^{0}(K_2)}\\
            &=\left(\frac{k_1!}{L_1^{|k_1|_1}}\right)^{\alpha}\sum\limits_{k_2\in\mathbb{N}^{m}}\left(\frac{L_1^{|k_1|_1}L_2^{|k_2|_1}}{k_1!k_2!}\right)^{\alpha}\left\lVert \partial_{\lambda}^{k_1}\partial_{x}^{k_2}\varphi \right\rVert_{C^{0}(K_2)}\\
            &\le \left(\frac{k_1!}{L_1^{|k_1|_1}}\right)^{\alpha}\left\lVert \varphi \right\rVert_{\alpha,L_1,L_2}.    \qedhere
        \end{split}
    \end{equation}
\end{proof}

\begin{proposition}
    Let $K=K_1\times K_2$, where $K_1\subset \mathbb{R}^{n}$ is a closed interval, and $K_2$ is an $m$-dimensional compact set. Let $\lambda_0\in K_1$. For any $\varphi(\lambda,x)\in G^{\alpha,L_1,L_2}(K)$, we can regard $\varphi$ as a map from $K_1$ to $G^{\alpha,L_2}(K_2)$, and the Taylor series 
    \begin{equation}
        \widetilde{\varphi}(\lambda_0+\Delta\lambda,x)=\sum\limits_{k=0}^{\infty}\frac{1}{k!}\left( \Delta\lambda_1\partial_{\lambda_1}+\cdots+ \Delta\lambda_n\partial_{\lambda_n} \right)^{k}\varphi(\lambda_0,x)
    \end{equation}
is the Gevrey-$(\alpha,L_2)$ asymptotic expansion of $\varphi$ at $\lambda_{0}$.
\label{bijinjino}
\end{proposition}
\begin{proof}
    Denote 
    $$\widetilde{\varphi}_{p}(\lambda_0+\Delta\lambda,x)=\sum\limits_{k=0}^{p}\frac{1}{k!}\left( \Delta\lambda_1\partial_{\lambda_1}+\cdots+ \Delta\lambda_n\partial_{\lambda_n} \right)^{k}\varphi(\lambda_0,x).$$ 
    By Taylor's formula, we have
    \begin{equation}
    \begin{split}
        \varphi(\lambda,x)-\widetilde{\varphi}_{p}(\lambda,x)
        =\frac{1}{p!}\int_{0}^{1}(1-t)^{p}\left( \Delta\lambda_1\partial_{\lambda_1}+\cdots+ \Delta\lambda_n\partial_{\lambda_n} \right)^{p+1}\varphi(\lambda_{0}+t\Delta\lambda,x)dt,
    \end{split}
    \notag
    \end{equation}
    where $\lambda=\lambda_0+\Delta\lambda$. Thus, we have
\begin{equation}
    \begin{split}
        &\left\lVert \varphi(\lambda,x)-\widetilde{\varphi}_{p}(\lambda,x) \right\rVert_{\alpha,L_2}\\
        \le& \frac{1}{p!}
        \int_{0}^{1}|(1-t)|^{p}\left\lVert\left( \Delta\lambda_1\partial_{\lambda_1}+\cdots+ \Delta\lambda_n\partial_{\lambda_n} \right)^{p+1}\varphi(\lambda_{0}+t\Delta\lambda)\right\rVert_{\alpha,L_2}dt\\
        \le&\frac{1}{p!}\int_{0}^{1}|(1-t)|^{p}\left\lVert \sum\limits_{k\in\mathbb{N}^{n}\atop |k|=p+1}\frac{(p+1)!}{k!}\Delta\lambda^{k}\partial_{\lambda}^{k}\varphi(\lambda_0+t\Delta\lambda,x) \right\rVert_{\alpha,L_2}dt\\
        \le&\frac{1}{p!}\int_{0}^{1}(1-t)^{p}dt
        \sum\limits_{k\in\mathbb{N}^{n}\atop |k|=p+1}\frac{(p+1)!}{k!}|\Delta\lambda^{k}|\sup\limits_{0\le t\le 1}\left\lVert \partial_{\lambda}^{k}\varphi(\lambda_0+t\Delta\lambda,x) \right\rVert_{\alpha,L_2}.
    \end{split}
    \notag
\end{equation}
Using Lemma \ref{supg}, we have
\begin{equation}
    \left\lVert \partial_{\lambda}^{k}\varphi(\lambda_0+t\Delta\lambda,x) \right\rVert_{\alpha,L_2}\le \left(\frac{k!}{L_{1}^{|k|_{1}}}\right)^{\alpha}\lVert \varphi \rVert_{\alpha,L_{1}, L_{2}}.
    \notag
\end{equation}
Thus, we obtain
\begin{equation}
    \left\lVert \varphi(\lambda,x)-\widetilde{\varphi}_{p}(\lambda,x) \right\rVert_{\alpha,L_2}\le \frac{1}{p!}\int_{0}^{1}(1-t)^{p}dt
        \sum\limits_{k\in\mathbb{N}^{n}\atop |k|_{1}=p+1}\frac{(p+1)!}{k!}|\Delta\lambda^{k}| \left(\frac{k!}{L_{1}^{|k|_{1}}}\right)^{\alpha}\lVert \varphi \rVert_{\alpha,L_{1}, L_{2}}.
        \notag
\end{equation}
This simplifies to
\begin{equation}
\begin{split}
    &\left\lVert \varphi(\lambda,x)-\widetilde{\varphi}_{p}(\lambda,x) \right\rVert_{\alpha,L_2}\\
    \le&\frac{1}{(p+1)!}\frac{\left\lVert \varphi \right\rVert_{\alpha,L_1,L_2}}{L_1^{(p+1)\alpha}}
        \sum\limits_{k\in\mathbb{N}^{n}\atop |k|_{1}=p+1} \frac{(p+1)!}{k!}|\Delta\lambda^k|\cdot k!^{\alpha}\\
        \le&\frac{(p+1)!^{\alpha-1}}{L_1^{(p+1)\alpha}}\left\lVert \varphi \right\rVert_{\alpha,L_1,L_2}
        \sum\limits_{k\in\mathbb{N}^{n}\atop |k|_1=p+1}\frac{(p+1)!}{k!}|\Delta\lambda^{k}|.
\end{split}
\notag
\end{equation}
Finally, using the multi-dimensional binomial formula, we conclude that
\begin{equation}
    \left\lVert \varphi(\lambda,x)-\widetilde{\varphi}_{p}(\lambda,x) \right\rVert_{\alpha,L_2}\le \left\lVert \varphi \right\rVert_{\alpha,L_1,L_2}(p+1)!^{\alpha-1}L_1^{-(p+1)\alpha}|\Delta\lambda|_1^{p+1}.
\end{equation}
Thus, the proof is complete.    
\end{proof}

We now present a frequently used lemma.

\begin{lemma}
    Let $F$ be a Gevrey function, and suppose its derivative $F'\in G^{\alpha,L}([-1,1])$. Let $q$ and $ \Delta q$ be functions defined on $K$, and assume $ q+t\Delta q\in G^{\alpha,\bm{L}}(K)$ for all $t\in[0,1]$, where $L,\bm{L}>0$ and $K$ is an m-dimensional compact set. If $\left\lVert q+t\Delta q \right\rVert_{\alpha,\bm{L}}\le\frac{L^{\alpha}}{m^{\alpha-1}}$ for all $t\in [0,1]$, $\left\lVert q\right\rVert_{C^{0}(K)}\ll 1$, $\left\lVert \Delta q\right\rVert_{C^{0}(K)}\ll 1$, and $\left\lVert q+\Delta q\right\rVert_{C^{0}(K)}\ll 1$,
    then we have
    \begin{equation}
        \left\lVert F(q+\Delta q) -F(q)\right\rVert_{\alpha,\bm{L}}\le \left\lVert F' \right\rVert_{\alpha,L} \left\lVert \Delta q\right\rVert_{\alpha,\bm{L}}.
        \nonumber
    \end{equation}
\label{chajv}
\end{lemma}

\begin{proof}
    Observe that 
    $$F(q+\Delta q)-F(q)=\int_{0}^{1} F'(q+t\Delta q)\cdot \Delta q dt.$$ 
    Thus, we have
    \begin{equation}
    \begin{split}
        \left\lVert F(q+\Delta q)-F(q) \right\rVert_{\alpha,\bm{L}}
        &=\left\lVert \int_{0}^{1} F'(q+t\Delta q)\cdot \Delta q dt \right\rVert_{\alpha,\bm{L}}        \\
        &\le \int_{0}^{1}\left\lVert F'(q+t\Delta q) \right\rVert_{\alpha,\bm{L}} \left\lVert \Delta q \right\rVert_{\alpha,\bm{L}} dt \\
        &\le\left \lVert F' \right\rVert_{\alpha,L'}  \left\lVert \Delta q \right\rVert_{\alpha,\bm{L}} .\qedhere
    \end{split}
    \end{equation}
\end{proof}


\section{Lyapounov-Schmidt decomposition}

In this section, we describe the Lyapunov-Schmidt decomposition. We begin by defining quasi-periodic functions, along with their associated frequencies and hulls.

\begin{definition}
    A function  $\tilde{q}(t,\tilde{x})$ is called  quasi-periodic in time $t$ if there exist $\lambda\in\mathbb{R}^{b}$ and  $q(\theta,\tilde{x})\in L^{2}(\mathbb{T}^{b+d}) $, with $\theta\in \mathbb{T}^{b}$ and $\tilde{x}\in\mathbb{T}^{d}$, such that
    \begin{equation}
        \tilde{q}(t,\tilde{x})=q(\lambda t,\tilde{x}).
    \end{equation}
Here, $\lambda$ is referred to as the frequency, and $q(\theta,\tilde{x})$ as the hull of $\tilde{q}(t,\tilde{x})$. 
  When it is expressed as a Fourier series,  we obtain
    \begin{equation}
        \tilde{q}(t,\tilde{x})=\sum\limits_{k\in\mathbb{Z}^{b},n\in\mathbb{Z}^{d}} \hat{q}(k,n)e^{i\lambda\cdot k t+in\cdot \tilde{x}},
    \end{equation}
where $\hat{q}(k,n)$ denotes the Fourier coefficients of $q(\theta,\tilde{x})$.
\end{definition}
To find quasi-periodic solutions with frequency $\lambda'$ to equation (\ref{e1}) is equivalent to solve
\begin{equation}
    \frac{1}{i}\lambda'\cdot\partial_{\theta}q-\mathcal{L}q-\varepsilon\frac{\partial H}{\partial\bar{q}}(q,\bar{q})=0,
    \label{e233}
\end{equation}
where $\lambda'\cdot \partial_{\theta}:=\sum\limits_{j=1}^{b}\lambda_{j}'\partial_{\theta_{j}}$. 
Taking the conjugate of (\ref{e233}), we obtain
\begin{equation}
    -\frac{1}{i}\lambda'\cdot\partial_{\theta}\bar{q}-\mathcal{L}'\bar{q}-\varepsilon\frac{\partial H}{\partial q}(q,\bar{q})=0,
    \label{e233.1}
\end{equation}
where $\mathcal{L}'$ is the operator with the multiplier $\{\mu_{n}'=\mu_{-n}\}_{n\in \mathbb{Z}^{d}}$.

The lattice formulation of (\ref{e233}) is as follows:
\begin{equation}
    (k\cdot\lambda'-\mu_n)\hat{q}(n,k)-\varepsilon\widehat{\frac{\partial H}{\partial\bar{q}}}(q,\bar{q})(n,k)=0,\ \ \ \ (n\in\mathbb{Z}^d,k\in\mathbb{Z}^b),
    \label{e2.2}
\end{equation}
where $\widehat{\frac{\partial H}{\partial\bar{q}}}(q,\bar{q})(n,k)$ represents the $(n,k)$-Fourier coefficient of $\frac{\partial H}{\partial\bar{q}}(q,\bar{q})$. If we can find $\lambda'\in \mathbb{R}^b$ and a sequence $\{\hat{q}(n,k)\}_{n\in\mathbb{Z}^d,k\in\mathbb{Z}^b}$ with a certain degree of decay that satisfies  (\ref{e2.2}), then the function
\begin{equation}
    \tilde{q}(t,\tilde{x})=\sum\limits_{k\in\mathbb{Z}^{b},n\in\mathbb{Z}^{d}} \hat{q}(k,n)e^{i\lambda'\cdot k t+in\cdot \tilde{x}}.
    \label{e2.1.1.1}
\end{equation}
with the found $\lambda'$ and $\hat{q}(n,k)$  is a quasi-periodic solution to equation (\ref{e1}) with the corresponding regularity.
Denote by $\left\lVert\cdot\right\rVert$ the norm of $L^{2}(\mathbb{T}^{b+d})$. If there exists $q(\theta,\tilde{x})\in L^{2}(\mathbb{T}^{b+d})$ such that 
\begin{equation}
    \left\lVert\frac{1}{i}\lambda'\cdot\partial_{\theta}q-\mathcal{L}q-\varepsilon\frac{\partial H}{\partial\bar{q}}(q,\bar{q})\right\rVert<\delta
\end{equation}
for $\lambda'\in\mathbb{R}^{b}$, then the quasi-periodic function $\tilde{q}(t,x)$, with the hull $q(\theta,\tilde{x})$ and frequency $\lambda'$, is referred to as an $O(\delta)$ approximate solution to (\ref{e1}). 

Let $\mathcal{R}=\{(n_j,e_j)\vert 1\le j\le b\}$, which we refer to as the resonant set. We specify the Fourier coefficients as
\begin{equation}
    \hat{q}(n_j,e_j)=a_j, \ \   1\le j\le b.
    \label{fixfour}
\end{equation}
To describe the Lyapounov-Schmidt decomposition, we introduce the following definition.

\begin{definition}
    Let $g(\theta,\tilde{x})\in L^{2}(\mathbb{T}^{b+d})$. We define the projection operators $\bm{\Gamma}_{Q}$, $\bm{\Gamma}_{P}$ and $\bm{\Gamma}_{P'}$   as follows:
    \begin{equation}
    \begin{split}
        &\bm{\Gamma}_{Q}g=\sum\limits_{(n,k)\in \mathcal{R}} \hat{g}(n,k)e^{in\cdot \tilde{x}+ik\cdot\theta} ,\\
        &\bm{\Gamma}_{P}g=\sum\limits_{(n,k)\in \mathbb{Z}^{b+d}\setminus\mathcal{R}}\hat{g}(n,k)e^{in\cdot \tilde{x}+ik\cdot\theta},\\
        &\bm{\Gamma}_{P'}g=\sum\limits_{(n,k)\in \mathbb{Z}^{b+d}\setminus(-\mathcal{R})}\hat{g}(n,k)e^{in\cdot \tilde{x}+ik\cdot\theta}.
    \end{split}
    \notag
    \end{equation}
\end{definition}
Denote the left side of (\ref{e233}) by $F(q)$, i.e., 
\begin{equation}
    F(q)=\frac{1}{i}\lambda'\cdot \partial_{\theta}q-\mathcal{L}q-\varepsilon \frac{\partial H}{\partial \bar{q}}(q,\bar{q})
\end{equation}
 then (\ref{e233}) is equivalent to
\begin{equation}
   \bm{\Gamma}_{Q}F(q)=0,
   \label{Qeq}
\end{equation}
\begin{equation}    
      \bm{\Gamma}_{P}F(q)=0.
    \label{Peq}
\end{equation}
Passing equations (\ref{Qeq}) and (\ref{Peq}) to Fourier coefficients,  we obtain the following  lattice equations:
\begin{equation}
    (\lambda_j'-\lambda_j)a_j-\varepsilon\widehat{\frac{\partial H}{\partial\bar{q}}}(q,\bar{q})(n_j,e_j)=0,\ \ 1\le j\le b,
    \label{e2.3}
\end{equation}
\begin{equation}
    (k\cdot\lambda'-\mu_n)\hat{q}(n,k)-\varepsilon\widehat{\frac{\partial H}{\partial\bar{q}}}(q,\bar{q})(n,k)=0, \ \ (n,k)\notin \mathcal{R}.
    \label{e2.4}
\end{equation}
we refer to the finite equations (\ref{e2.3}) as the \textbf{Q-equations}
and the infinite equations (\ref{e2.4}) as the  \textbf{P-equations}.

However, the P-equations cannot be addressed using a standard implicit function theorem due to the presence of small divisors. Consequently, we will solve P-equations using Newton's algorithm. To apply Newton's algorithm, it is necessary to duplicate  equations 
 (\ref{e2.4}).
Note that $\hat{\bar{q}}(-n,-k)=\bar{\hat{q}}(n,k)$. Taking the conjugate of equations (\ref{e2.4}),  we can duplicate them as follows:

\begin{equation}
    \begin{cases}
        (k\cdot\lambda'-\mu_n)\hat{q}(n,k)-\varepsilon\widehat{\frac{\partial H}{\partial\bar{q}}}(q,\bar{q})(n,k)=0 \ \ &(n,k)\notin \mathcal{R},\\
        (-k\cdot\lambda'-\mu_{-n})\hat{\bar{q}}(n,k)-\varepsilon\widehat{\frac{\partial H}{\partial{q}}}(q,\bar{q})(n,k)=0 \ \ &(n,k)\notin -\mathcal{R}.
    \end{cases}
    \label{e2.5}
\end{equation}
We will see that the relation $\hat{\bar{q}}(-n,-k)=\bar{\hat{q}}(n,k)$ is preserved for the approximate solutions in  Newton's algorithm. In fact, the second equality in (\ref{e2.5}) is simply $\bm{\Gamma}_{P'}\bar{F}(q)=0$ modulo a Fourier transform. Therefore,  up to a Fourier transform, (\ref{e2.5}) can be written as follows:
\begin{equation}
    \begin{cases}
        \bm{\Gamma}_{P}F(q)=0,\\
        \bm{\Gamma}_{P'}\bar{F}(q)=0.
    \end{cases}
    \label{phh}
\end{equation}
We also refer to both  (\ref{e2.5}) and (\ref{phh}) as the $P$-equations.

Thus, the general procedure is to determine $\hat{q}(n,k)\vert_{(n,k)\notin \mathcal{R}}$(which depends on $\lambda'$ and $\lambda$)   from the $P$-equations,  and then substitute it  into $Q$-equations  to obtain the new frequencies $\lambda'=(\lambda_1',...,\lambda_b')$. Note that the frequencies $\lambda_1',...,\lambda_b'$ need to be real,  otherwise the function (\ref{e2.1.1.1}) cannot represent a quasi-periodic solution to (\ref{e1}).  We will demonstrate the reality of $\lambda'$ through our construction by applying Proposition \ref{p1} in the Appendix.

\section{The homological equations and the estimate of the remians}

Now we describe the Newton iteration scheme to solve $P$-equations. Recall that $H(q,\bar{q})=f(q\bar{q})$. Let $q\in L^{2}(\mathbb{T}^{b+d})$, $\Delta q\in L^{2}(\mathbb{T}^{b+d})$ and $\textup{supp } \widehat{\Delta q}\subset \mathbb{Z}^{b+d}\setminus\mathcal{R}$, then we have 
\begin{equation}
\begin{split}
    F(q+\Delta q)=&F(q)+\frac{1}{i}\lambda'\cdot\partial_{\theta}\Delta q-\mathcal{L}\Delta q-\varepsilon \frac{\partial^{2}H}{\partial q\partial\bar{q}}(q,\bar{q})\Delta q-\varepsilon\frac{\partial^{2}H}{\partial\bar{q}^2}(q,\bar{q})\overline{\Delta q}   \\
    &-\varepsilon \int_{0}^{1}(1-s)(\Delta q \frac{\partial}{\partial q}+\overline{\Delta q}\frac{\partial}{\partial\bar{q}})^{2}\frac{\partial H}{\partial \bar{q}}(q+s\Delta q, \bar{q}+s\overline{\Delta q})ds.
\end{split}
\label{n.1}
\end{equation}
Taking the conjugate on both sides of (\ref{n.1}),  we have 
\begin{equation}
    \begin{split}
        \bar{F}(q+\Delta q)=\bar{F}(q)-\frac{1}{i}\lambda'\cdot\partial_{\theta}\overline{\Delta q}-\mathcal{L}'\overline{\Delta q}-\varepsilon\frac{\partial^{2}H}{\partial q^{2}}(q,\bar{q})\Delta q-\varepsilon\frac{\partial^2 H}{\partial q\partial\bar{q}}(q,\bar{q})\overline{\Delta q}\\
        -\varepsilon \int_{0}^{1}(1-s)(\Delta q \frac{\partial}{\partial q}+\overline{\Delta q}\frac{\partial}{\partial\bar{q}})^{2}\frac{\partial H}{\partial q}(q+s\Delta q, \bar{q}+s\overline{\Delta q})ds.
    \end{split}
    \label{n.2}
\end{equation}
Summarizing equations (\ref{n.1}) and (\ref{n.2}), we conclude that
\begin{equation}
    \begin{split}
        \begin{pmatrix}
            F(q+\Delta q)\\
            \bar{F}(q+\Delta q)
        \end{pmatrix}
        =&\begin{pmatrix}
            F(q)\\
            \bar{F}(q)
            \end{pmatrix}
            +\widetilde{T}_{q}
            \begin{pmatrix}
                \Delta q\\
                \overline{\Delta q}
            \end{pmatrix}\\    
        &+
            \begin{pmatrix}
                -\varepsilon \int_{0}^{1}(1-s)(\Delta q \frac{\partial}{\partial q}+\overline{\Delta q}\frac{\partial}{\partial\bar{q}})^{2}\frac{\partial H}{\partial \bar{q}}(q+s\Delta q, \bar{q}+s\overline{\Delta q})ds\\
                -\varepsilon \int_{0}^{1}(1-s)(\Delta q \frac{\partial}{\partial q}+\overline{\Delta q}\frac{\partial}{\partial\bar{q}})^{2}\frac{\partial H}{\partial q}(q+s\Delta q, \bar{q}+s\overline{\Delta q})ds
            \end{pmatrix},
    \end{split}
\end{equation}
where
\begin{equation}
    \widetilde{T}_{q}=\left[
    \begin{pmatrix}
        \frac{1}{i}\lambda'\cdot\partial_{\theta}-\mathcal{L}& 0\\
        0 & -\frac{1}{i}\lambda'\cdot\partial_{\theta}-\mathcal{L}'
    \end{pmatrix}
    -\varepsilon
    \begin{pmatrix}
        \frac{\partial^2 H}{\partial q \partial \bar{q}}(q,\bar{q}) &\frac{\partial^2 H}{\partial \bar{q}^2}(q,\bar{q})\\
        \frac{\partial^2 H}{\partial q^2}(q,\bar{q}) & \frac{\partial^2 H}{\partial q\partial \bar{q}}(q,\bar{q})
    \end{pmatrix}
    \right].
    \label{huanmianshangxianxinghuasuanzi}
\end{equation}
Let
\begin{equation}
\widetilde{T}_{q,P}
    \begin{pmatrix}
        \Delta q\\
        \overline{\Delta q}
    \end{pmatrix}
     =-
        \begin{pmatrix}
            \bm{\Gamma}_{P}F(q)\\
            \bm{\Gamma}_{P'}\bar{F}(q)
        \end{pmatrix},  
        \label{n.3}
\end{equation}
where 
\begin{equation}
\widetilde{T}_{q,P}=
    \begin{pmatrix}
    \bm{\Gamma}_{P}&0\\
    0&\bm{\Gamma}_{P'}
\end{pmatrix}
\widetilde{T}_{q}
\begin{pmatrix}
    \bm{\Gamma}_{P}&0\\
    0&\bm{\Gamma}_{P'}
\end{pmatrix}.
\end{equation}

We denote $T_{q}=D-\varepsilon S_{q}$, where
\begin{equation}
    D=
    \begin{pmatrix}
        k\cdot\lambda'-\mu_n&0\\
        0&  -k\cdot\lambda'-\mu_{-n}
    \end{pmatrix},
    \ \ \ \ 
    S_{q}=
    \begin{pmatrix}
        S_{\frac{\partial^2 H}{\partial q \partial \bar{q}}(q,\bar{q})}& S_{\frac{\partial^2 H}{\partial \bar{q}^2}(q,\bar{q})}\\
        S_{\frac{\partial^2 H}{\partial q^2}(q,\bar{q})} & S_{\frac{\partial^2 H}{\partial q \partial \bar{q}}(q,\bar{q})}
    \end{pmatrix},
\end{equation}
and $S_{\phi}$ represents  the Toeplitz operator corresponding to $\phi$-multiplication. The matrix index is $(\pm,n,k)$, where $n\in\mathbb{Z}^{d}$, $k\in\mathbb{Z}^{b}$, and the index of $\begin{pmatrix}
    \widehat{\Delta q}\\
    \widehat{\overline{\Delta q}}
\end{pmatrix}$ is also $(\pm,n,k)$. More precisely, we have
\begin{equation}
        S_{q}((+,n,k),(+,n',k'))=S_{\frac{\partial^2 H}{\partial q \partial \bar{q}}(q,\bar{q})}((n,k),(n',k')),
        \notag
\end{equation}
\begin{equation}
    S_{q}((+,n,k),
        (-,n',k'))=S_{\frac{\partial^2 H}{\partial \bar{q}^2}(q,\bar{q})}((n,k),(n',k')),
        \notag
\end{equation}

\begin{equation}
          S_{q}((-,n,k),(+,n',k'))=S_{\frac{\partial^2 H}{\partial q^2}(q,\bar{q})}((n,k),(n',k')),
          \notag
\end{equation}

\begin{equation}
    S_{q}((-,n,k),(-,n',k'))=S_{\frac{\partial^2 H}{\partial q \partial \bar{q}}(q,\bar{q})}((n,k),(n',k')),
    \notag
\end{equation}

$$
\begin{pmatrix}
    \widehat{\Delta q}\\
    \widehat{\overline{\Delta q}}
\end{pmatrix}(+,n,k)=\widehat{\Delta q}(n,k)
$$
and 
$$
\begin{pmatrix}
    \widehat{\Delta q}\\
    \widehat{\overline{\Delta q}}
\end{pmatrix}(-,n,k)=\widehat{\overline{\Delta q}}(n,k).
$$

Passing equation (\ref{n.3})  to Fourier coefficients,  we obtain the homological equation of (\ref{e2.5}): 
\begin{equation}
    T_{q,P}\begin{pmatrix}
        \widehat{\Delta q}\\
        \widehat{\overline{\Delta q}}
    \end{pmatrix}=-
    \begin{pmatrix}
            R_{P}\widehat{F(q)}\\
            R_{P'}\widehat{\bar{F}(q)}
        \end{pmatrix},
        \label{linear equation}
\end{equation}
where the linearized operator is given by 
\begin{equation}
    T_{q,P}=R_{P\cup P'}(D-\varepsilon S_{q})R_{P\cup P'},
    \label{n.4}
\end{equation}
and $R_{P}$ refers to coordinate restriction to $(\mathbb{Z}^{b+d}\setminus \mathcal{R})\subset\mathbb{Z}^{b+d}$, $R_{P'}$ refers to coordinate restriction to $(\mathbb{Z}^{b+d}\setminus (-\mathcal{R}))\subset\mathbb{Z}^{b+d}$ and $R_{P\cup P'}$ refers to coordinate restriction to $\{(+,n,k)|(n,k)\notin \mathcal{R}\}\cup\{(-,n,k)|(n,k)\notin -\mathcal{R}\}$.

However, we cannot solve (\ref{linear equation}) directly. Instead, we aim to solve an approximation of (\ref{linear equation}) as follows:
\begin{equation}
    T_{q,P,N} \begin{pmatrix}
        \widehat{\Delta q}\\
        \widehat{\overline{\Delta q}}
    \end{pmatrix}
    =\begin{pmatrix}
        -R_{P,N} \widehat{F(q)}\\
        -R_{P',N} \widehat{\bar{F}(q)}
    \end{pmatrix},
    \label{sol}
\end{equation}
i.e.,
\begin{equation}
    \begin{pmatrix}
        \widehat{\Delta q}\\
        \widehat{\overline{\Delta q}}
    \end{pmatrix}
    =(T_{q,P,N})^{-1}\begin{pmatrix}
        -R_{P,N} \widehat{F(q)}\\
        -R_{P',N} \widehat{\bar{F}(q)}
    \end{pmatrix},
    \label{solution}
\end{equation}
where  $T_{q,P,N}=R_{P\cup P',N}T_{q}R_{P\cup P',N}$, and $R_{P,N}$ refers to coordinate restriction to 
$$\left((\mathbb{Z}^{b+d}\setminus \mathcal{R})\cap \{|n|<N,|k|<N\}\right)\subset\mathbb{Z}^{b+d},$$
$R_{P',N}$ refers to coordinate restriction to 
$$\left((\mathbb{Z}^{b+d}\setminus (-\mathcal{R}))\cap \{|n|<N,|k|<N\}\right)\subset\mathbb{Z}^{b+d},$$
and $R_{P\cup P', N}$ refers to coordinate restriction to
$$
\left(\{(+,n,k)|(n,k)\notin \mathcal{R}\}\cup\{(-,n,k)|(n,k)\notin -\mathcal{R}\}\right)\cap\{|n|<N,|k|<N\}.
$$
  The solution $\Delta q$ and $\overline{\Delta q}$ obtained by solving equations (\ref{sol}) are conjugates of each other as the relation $\widehat{\overline{\Delta q}}(-n,-k)=\overline{\widehat{\Delta q}(n,k)}$ holds.  Up to a Fourier transform, equation (\ref{sol}) can be rewritten as:
\begin{equation}
    \widetilde{T}_{q,P,N}
    \begin{pmatrix}
        \Delta q\\
        \overline{\Delta q}
    \end{pmatrix}=
    \begin{pmatrix}
        -\bm{\Gamma}_{P,N}F(q)\\
        -\bm{\Gamma}_{P',N}\bar{F}(q)
    \end{pmatrix},
    \label{solfou}
\end{equation}
where
\begin{equation}
    \widetilde{T}_{q,P,N}
    =\begin{pmatrix}
    \bm{\Gamma}_{P,N}& 0\\
    0 & \bm{\Gamma}_{P',N}
\end{pmatrix}
\widetilde{T}_{q}
\begin{pmatrix}
    \bm{\Gamma}_{P,N}& 0\\
    0 & \bm{\Gamma}_{P',N}
\end{pmatrix}
,
\end{equation}
and $$\bm{\Gamma}_{P,N} h(\theta,\tilde{x})=\sum\limits_{\substack{(n,k)\in\mathbb{Z}^{b+d}\setminus\mathcal{R}\\ |n|<N,|k|<N}}\hat{h}(n,k)e^{i n\cdot \tilde{x}+ik\cdot \theta},$$ 
$$\bm{\Gamma}_{P',N} h(\theta,\tilde{x})=\sum\limits_{\substack{(n,k)\in\mathbb{Z}^{b+d}\setminus(-\mathcal{R})\\ |n|<N,|k|<N}}\hat{h}(n,k)e^{i n\cdot \tilde{x}+ik\cdot \theta}$$ 
for $h\in L^{2}(\mathbb{T}^{b+d})$. 
For $\Delta q,\overline{\Delta q}$ satisfying (\ref{sol}) (i.e.,(\ref{solfou}))  we have
$\textup{supp } \widehat{\Delta q}\subset (\mathbb{Z}^{b+d}\setminus \mathcal{R})\cap\{|n|<N,|k|<N\}$ and 
\begin{align}
    \begin{pmatrix}
        \bm{\Gamma}_{P}F(q+\Delta q)\\
        \bm{\Gamma}_{P'}\bar{F}(q+\Delta q)
    \end{pmatrix}
    =&
    \begin{pmatrix}
        \bm{\Gamma}_{P} F(q)\\
        \bm{\Gamma}_{P'}\bar{F}(q)
    \end{pmatrix}
    +\widetilde{T}_{q,P}
    \begin{pmatrix}
        \Delta q\\
        \overline{\Delta q}
    \end{pmatrix}
    +O(\Delta q^{2})\notag\\
    =&\begin{pmatrix}
        \bm{\Gamma}_{P,N} F(q)\\
        \bm{\Gamma}_{P',N}\bar{F}(q)
    \end{pmatrix}+
    \begin{pmatrix}
        (\bm{\Gamma}_{P}-\bm{\Gamma}_{P,N}) F(q)\\
        (\bm{\Gamma}_{P'}-\bm{\Gamma}_{P',N})\bar{F}(q)
    \end{pmatrix}\notag\\
    & +\widetilde{T}_{q,P,N}
    \begin{pmatrix}
        \Delta q\\
        \overline{\Delta q}
    \end{pmatrix}
    +(\widetilde{T}_{q,P}-\widetilde{T}_{q,P,N})
    \begin{pmatrix}
        \Delta q\\
        \overline{\Delta q}
    \end{pmatrix}
    +O(\Delta q^{2})\notag\\
    =&\begin{pmatrix}
        (\bm{\Gamma}_{P}-\bm{\Gamma}_{P,N}) F(q)\\
        (\bm{\Gamma}_{P'}-\bm{\Gamma}_{P',N})\bar{F}(q)
    \end{pmatrix}
    +(\widetilde{T}_{q,P}-\widetilde{T}_{q,P,N})
    \begin{pmatrix}
        \Delta q\\
        \overline{\Delta q}
    \end{pmatrix}
    +O(\Delta q^{2}),\label{n.7}
\end{align}
where
\begin{equation}
    O(\Delta q^{2}):=
    \begin{pmatrix}
        -\varepsilon\bm{\Gamma}_{P} \int_{0}^{1}(1-s)(\Delta q \frac{\partial}{\partial q}+\overline{\Delta q}\frac{\partial}{\partial\bar{q}})^{2}\frac{\partial H}{\partial \bar{q}}(q+s\Delta q, \bar{q}+s\overline{\Delta q})ds\\
        -\varepsilon \bm{\Gamma}_{P'} \int_{0}^{1}(1-s)(\Delta q \frac{\partial}{\partial q}+\overline{\Delta q}\frac{\partial}{\partial q})^{2}\frac{\partial H}{\partial q}(q+s\Delta q, \bar{q}+s\overline{\Delta q})ds
    \end{pmatrix},
\end{equation}
and we denote the first and the second components of $O(\Delta q^{2})$  by $O(\Delta q^{2})_{1}$ and $O(\Delta q^{2})_{2}$ respectively.
Assume one succeeds in ensuring that
\begin{equation}
    |T_{q,P,N}^{-1}((\nu,x),(\nu',x'))|<B(r) e^{-\alpha L^{(1)}|x-x'|_{\alpha}},
    \label{n.8.1}
\end{equation}
where $\nu,\nu'=\pm$, $x,x'\in\mathbb{Z}^{b+d}$, $B(r)$ and $N$ is to be specified,  and
\begin{equation}
\left\{
    \begin{split}
        &\left\lVert\bm{\Gamma}_{P} F(q) \right\rVert_{\mathcal{F}_{\alpha, L^{(0)}}}<\epsilon,\\
        &\left\lVert \bm{\Gamma}_{P'} \bar{F}(q) \right\rVert_{\mathcal{F}_{\alpha, L^{(0)}}}<\epsilon,\\
        &\left\lVert q \right\rVert_{\alpha,L^{(0)}}\ll 1,
    \end{split}\right.
    \label{n.8}
\end{equation}
where $\epsilon$ is small, $L^{(0)},L^{(1)}>0$ and $ L^{(0)}>L^{(1)}$. 

First, we estimate $\Delta q$. According to equation (\ref{solution}),  for any $x\in\mathbb{Z}^{b+d}\setminus\mathcal{R}$ and $|x|<N$, we have
\begin{equation}
    \widehat{\Delta q}(x)=\sum\limits_{\substack{|y|<N\\ y\in\mathbb{Z}^{b+d}\setminus\mathcal{R}}}T_{q,P,N}^{-1}((+,x),(+,y))\widehat{F(q)}(y)+\sum\limits_{\substack{|y|<N\\ y\in\mathbb{Z}^{b+d}\setminus-\mathcal{R}}}T_{q,P,N}^{-1}((+,x),(-,y))\widehat{\bar{F}(q)}(y).
    \notag
\end{equation}
Since $\left\lVert \bm{\Gamma}_{P}F(q) \right\rVert_{\mathcal{F}_{\alpha, L^{(0)}}}<\epsilon$ and $\left\lVert \bm{\Gamma}_{P'}\bar{F}(q) \right\rVert_{\mathcal{F}_{\alpha, L^{(0)}}}<\epsilon$, we have
$$|\widehat{F(q)}(x)|<\epsilon e^{-\alpha L^{(0)} |x|_{\alpha}}, \textup{ for } x\in\mathbb{Z}^{b+d}\setminus \mathcal{R},$$   and 
$$|\widehat{\bar{F}(q)}(x)|<\epsilon e^{-\alpha L^{(0)} |x|_{\alpha}}, \textup{ for } x\in\mathbb{Z}^{b+d}\setminus-\mathcal{R}.$$ 
Thus,
\begin{align}
    |\widehat{\Delta q}(x)|
        \le& \sum\limits_{\substack{|y|<N\\ y\in\mathbb{Z}^{b+d}\setminus\mathcal{R}}}|T_{q,P,N}^{-1}((+,x),(+,y))||\widehat{F(q)}(y)| \notag \\
        &+\sum\limits_{\substack{|y|<N\\ y\in\mathbb{Z}^{b+d}\setminus-\mathcal{R}}}|T_{q,P,N}^{-1}((+,x),(-,y))||\widehat{\bar{F}(q)}(y)| \notag \\
        \le&  \sum\limits_{\substack{|y|<N\\ y\in\mathbb{Z}^{b+d}\setminus\mathcal{R}}} \epsilon B(r)e^{-\alpha L^{(1)}|x-y|_{\alpha}}e^{-\alpha L^{(0)} |y|_{\alpha}}+
        \sum\limits_{\substack{|y|<N\\ y\in\mathbb{Z}^{b+d}\setminus-\mathcal{R}}}\epsilon B(r)e^{-\alpha L^{(1)}|x-y|_{\alpha}}e^{-\alpha L^{(0)} |y|_{\alpha}} \notag \\
        \le& \epsilon B(r) C_{\alpha} \delta^{-\alpha(b+d)}e^{-L^{(2)}|x|_{\alpha}},\label{estimateofdq}
\end{align}
where $0<\delta<L^{(1)}$, $L^{(2)}=L^{(1)}-\delta$ and $C_{\alpha}=\alpha^{-\alpha(b+d)}(\int_{0}^{+\infty}e^{-x^{1/\alpha}}dx)^{b+d}$. 

In the following part of this section, we  make the following  assumptions:
\begin{itemize}
    \item $\epsilon=\varepsilon^{a^{r}} \textup{for some } 1<a<2$,
    \item $\log B(r)< Cr^{C} \log\log\frac{1}{\varepsilon}$,
    \item $\delta\sim r^{-C}$.
\end{itemize}
Based on these assumptions, we derive that $$|\widehat{\Delta q}(x)|\le \epsilon^{1-} e^{-\alpha L^{(2)} |x|_{\alpha}} \textup{ for } x\in \mathbb{Z}^{b+d}\setminus\mathcal{R}.$$ 
 Additionally, the norm $\left\lVert \Delta q \right\rVert_{\mathcal{F}_{\alpha,L^{(2)}-\delta}}$ can be bounded as follows:
\begin{align}
    \left\lVert \Delta q \right\rVert_{\mathcal{F}_{\alpha,L^{(2)}-\delta}}    
    &=\sum\limits_{x\in \mathbb{Z}^{b+d}}|\widehat{\Delta q}(x)|\cdot e^{\alpha(L^{(2)}-\delta)|x|_{\alpha}}\notag\\
    &\le\epsilon^{1-} \sum\limits_{x\in \mathbb{Z}^{b+d}}e^{-\alpha \delta |x|_{\alpha}}\notag\\
    &\le \epsilon^{1-}C_{\alpha}\delta^{-\alpha (b+d)}\notag\\
    &\le \epsilon^{1-}.
\end{align}

Now we estimate the three terms in (\ref{n.7}).

Firstly, we consider $O(\Delta q^{2})$.  Let $L^{(3)}=L^{(2)}-\delta$, then $\left\lVert \Delta q \right\rVert_{\mathcal{F}_{\alpha,L^{(3)}}}\le \epsilon^{1-}$.  Therefore, we have
$$\left\lVert O(\Delta q^{2})_{1}\right\rVert_{\mathcal{F}_{\alpha,L^{(4)}}}<\epsilon^{2-} \textup{ and }\left\lVert O(\Delta q^{2})_{2}\right\rVert_{\mathcal{F}_{\alpha,L^{(4)}}}<\epsilon^{2-},$$ 
where $L^{(4)}=L^{(3)}-\delta$.

Recalling the definition of $T_{q}$, we get
\begin{equation}
    \begin{cases}
        |T_{q}((\nu,x),(\nu',x'))|<\varepsilon,\ \ \nu,\nu'=\pm, x\ne x', |x-x'|<\delta^{-C},\\
        |T_{q}((\nu,x),(\nu',x'))|< e^{-\alpha (L^{(0)}-\delta)|x-x'|_{\alpha}},\ \ |x-x'|\ge \delta^{-C},
    \end{cases}
\end{equation}
since $\left\lVert q \right\rVert_{\alpha,L^{(0)}}\ll 1$. 

Secondly, we consider the term: 
$$\begin{pmatrix}
    (\bm{\Gamma}_{P}-\bm{\Gamma}_{P,N})F(q)\\  (\bm{\Gamma}_{P'}-\bm{\Gamma}_{P',N})\bar{F}(q)
\end{pmatrix}.$$ 
Taking $(\bm{\Gamma}_{P}-\bm{\Gamma}_{P,N})F(q)$ for example, since $\left\lVert \bm{\Gamma}_{P} F(q) \right\rVert_{\mathcal{F}_{\alpha,L^{(0)}}}<\epsilon$, we have
\begin{align}
    \left\lVert (\bm{\Gamma}_{P}-\bm{\Gamma}_{P,N})F(q) \right\rVert_{\mathcal{F}_{\alpha,L^{(2)}}}
        &=\sum\limits_{|x|\ge N,x\notin\mathcal{R}}|\widehat{F(q)}(x)|e^{\alpha L^{(2)}|x|_{\alpha}}\notag\\
        &\le \sum\limits_{|x|\ge N,x\notin\mathcal{R}} \epsilon e^{-\alpha L^{(0)} |x|_{\alpha}+\alpha L^{(2)}|x|_{\alpha}}\notag\\
        &\le \epsilon \sum\limits_{|x|\ge N,x\notin\mathcal{R}} e^{-\alpha \delta |x|_{\alpha}}\notag\\
        &\lesssim \epsilon (\alpha \delta)^{-\alpha (b+d)}e^{-\frac{1}{2}\alpha \delta N^{1/\alpha}}\label{Gamm}.
\end{align}
If we choose $N=A^{\alpha r}$ for large $A$ (for example $A=(\log \frac{1}{\varepsilon})^{\alpha}$),  then $\epsilon r^{C}e^{-\frac{1}{2}r^{-C}A^{r}}<\epsilon^{2}$. 

Thirdly, we consider 
$$(\widetilde{T}_{q,P}-\widetilde{T}_{q,P,N})
    \begin{pmatrix}
        \Delta q\\
        \overline{\Delta q}
    \end{pmatrix}.$$
For $\nu=\pm$ and $|x|<N$, we have
\begin{equation}
\begin{split}
    &\left((\widetilde{T}_{q,P}-\widetilde{T}_{q,P,N})
    \begin{pmatrix}
        \Delta q\\
        \overline{\Delta q}
    \end{pmatrix}\right)^{\land}(\nu,x)\\
    =&\sum\limits_{|y|< N}[T_{q,P}((\nu,x),(+,y))-T_{q,P,N}((\nu,x),(+,y))]\widehat{\Delta q}(y)+\\
    &\sum\limits_{|y|< N}[T_{q,P}((\nu,x),(-,y))-T_{q,P,N}((\nu,x),(-,y))]\widehat{\overline{\Delta q}}(y)\\
    =&0.
\end{split}
\notag
\end{equation}
For $\nu=\pm$ and  $|x|\ge N$, we have
\begin{align}
    &\left((\widetilde{T}_{q,P}-\widetilde{T}_{q,P,N})
    \begin{pmatrix}
        \Delta q\\
        \overline{\Delta q}
    \end{pmatrix}\right)^{\land}(\nu,x)\notag\\
        =&\sum\limits_{|y|<N}|T_{q,P}((\nu,x),(+,y))\widehat{\Delta q}(y)|+\sum\limits_{|y|<N}|T_{q,P}((\nu,x),(-,y))\widehat{\overline{\Delta q}}(y)|\notag\\
        =&\sum\limits_{|x-y|<\delta^{-C}}N^{2}(|\widehat{\Delta q}(y)|+|\widehat{\overline{\Delta q}}(y)|)+\sum\limits_{|x-y|\ge \delta^{-C}} e^{-\alpha(L^{(0)}-\delta)|x-y|_{\alpha}}(|\widehat{\Delta q}(y)|+|\widehat{\overline{\Delta q}}(y)|)\notag\\
        \le& \epsilon^{1-}\left(\sum\limits_{|x-y|<\delta^{-C}}e^{-\alpha L^{(3)} |y|_{\alpha}}+\sum\limits_{|x-y|\ge \delta}e^{-\alpha L^{(3)}|x-y|_{\alpha}}e^{-\alpha L^{(3)} |y|_{\alpha}}\right)\notag\\
        \le& \epsilon^{1-}\left( \delta^{-C}e^{\alpha L^{(3)} \delta^{-C}}e^{-\alpha L^{(3)}|x|_{\alpha}}+C_{\alpha}\delta^{-\alpha (b+d)}e^{-\alpha L^{(4)}|x|_{\alpha}} \right)\notag\\
        \le& \epsilon^{1-}e^{-\alpha L^{(4)} |x|_{\alpha}}.\label{ttn}
\end{align}
Denote by 
$$\left((\widetilde{T}_{q,P}-\widetilde{T}_{q,P,N})
    \begin{pmatrix}
        \Delta q\\
        \overline{\Delta q}
    \end{pmatrix}\right)_{1}$$
    and 
    $$\left((\widetilde{T}_{q,P}-\widetilde{T}_{q,P,N})
    \begin{pmatrix}
        \Delta q\\
        \overline{\Delta q}
    \end{pmatrix}\right)_{2}$$
    the first and the second components of $(\widetilde{T}_{q,P}-\widetilde{T}_{q,P,N})
    \begin{pmatrix}
        \Delta q\\
        \overline{\Delta q}
    \end{pmatrix}$ respectively.
 Let $L^{(5)}=L^{(4)}-\delta$,  as discussion in (\ref{Gamm}).  We then have 
\begin{equation}
\left\{
\begin{split}
    \left\lVert \left((\widetilde{T}_{q,P}-\widetilde{T}_{q,P,N})
    \begin{pmatrix}
        \Delta q\\
        \overline{\Delta q}
    \end{pmatrix}\right)_{1} \right\rVert_{\mathcal{F}_{\alpha,L^{(5)}}}\le \epsilon^{1-} (\alpha\delta)^{-\alpha(b+d)}e^{-\frac{1}{2}\alpha\delta N^{1/\alpha}}<\epsilon^{2},\\
    \left\lVert \left((\widetilde{T}_{q,P}-\widetilde{T}_{q,P,N})
    \begin{pmatrix}
        \Delta q\\
        \overline{\Delta q}
    \end{pmatrix}\right)_{2} \right\rVert_{\mathcal{F}_{\alpha,L^{(5)}}}\le \epsilon^{1-} (\alpha\delta)^{-\alpha(b+d)}e^{-\frac{1}{2}\alpha\delta N^{1/\alpha}}<\epsilon^{2}.
\end{split}\right.
\end{equation}
Thus, we have $\left\lVert \bm{\Gamma}_{P} F(q+\Delta q) \right\rVert_{\mathcal{F}_{\alpha,L^{(5)}}}<\epsilon^{2-}$  and $\left\lVert \bm{\Gamma}_{P'} \bar{F}(q+\Delta q) \right\rVert_{\mathcal{F}_{\alpha,L^{(5)}}}<\epsilon^{2-}$.

Thus under the assumptions (\ref{n.8.1}) and (\ref{n.8}), we  can find $\Delta q$ such that
\begin{equation}
\left\{
    \begin{split}
        &\left\lVert \bm{\Gamma}_{P} F(q+\Delta q) \right\rVert_{\mathcal{F}_{\alpha,L^{(5)}}}<\epsilon^{2-}<\varepsilon^{a^{r+1}},\\
        &\left\lVert \bm{\Gamma}_{P'} \bar{F}(q+\Delta q) \right\rVert_{\mathcal{F}_{\alpha,L^{(5)}}}<\epsilon^{2-}<\varepsilon^{a^{r+1}},\\
        &\left\lVert q+\Delta q \right\rVert_{\alpha, L^{(3)}}\ll 1.
    \end{split}\right.
    \label{n.8.2}
\end{equation}
Moreover, we have
\begin{equation}
\left\{
    \begin{split}
        &\textup{supp } \widehat{\Delta q}\subset \{x\in\mathbb{Z}^{b+d}| |x|<N\}\cap\mathcal{R},\\
        &\left\lVert \Delta q \right\rVert_{\mathcal{F}_{\alpha,L^{(3)}}}<\epsilon^{1-}.
    \end{split}\right.
\end{equation}
Thus our main task is to obtain  (\ref{n.8.1}) at each step, i.e., the estimate of the inverse of the linearized operator $T_{q,P,N}^{-1}$. In the sequel, we refer to $T_{q,P,N}^{-1}$ 
 as the Green function.

Now we consider the Q-equations:
\begin{equation}
    (\lambda'_j-\lambda_j)a_j-\varepsilon\widehat{\frac{\partial H}{\partial \bar{q}}}(q,\bar{q})(n_j,e_j)=0,\ \ 1\le j \le b,
\end{equation}
i.e.,
\begin{equation}
    \lambda'_j=\lambda_j+\varepsilon\frac{1}{a_j}\widehat{\frac{\partial H}{\partial \bar{q}}}(q,\bar{q})(n_j,e_j).
    \label{Q-equation}
\end{equation}

In fact, to maintain the bound of the Green function, some conditions on $(\lambda,\lambda',a)$ need to be satisfied. However, we assume  that $\hat{q}|_{(n,k)\notin \mathcal{R}}$ are defined as  $C^1$ functions over the entire $(\lambda,\lambda')$-parameter space $I_0 \times I_0\subset \mathbb{R}^{2b}$ (we will extend q to the entire parameter space). Moreover,  assume  $\left\lVert \partial q \right\rVert_{C(I_0\times I_0)}<C$ (where $\partial$ refers to derivative with respect to $\lambda$ or $\lambda'$), then we have $|\partial\widehat{\frac{\partial H}{\partial \bar{q}}}(n_j,e_j)|<C$. 

Let 
\begin{equation}
G(\lambda,\lambda')=\left(\lambda_j'-\lambda_j-\varepsilon\frac{1}{a_j}\widehat{\frac{\partial H}{\partial \bar{q}}}(q,\bar{q})(n_j,e_j)\right)_{j=1,...,b},
\end{equation}
 then $$\frac{\partial G(\lambda,\lambda')}{\partial \lambda'}=Id+O(\varepsilon),$$ 
 and
 $$\det\left(\frac{\partial G(\lambda,\lambda')}{\partial \lambda'}\right)\ne 0 \textup{  for all  }(\lambda,\lambda')\in I_0\times I_0.$$
Thus, we can solve $\lambda'$ from (\ref{Q-equation}) to get  a solution such that:
\begin{equation}
\lambda'=\lambda+\varepsilon \varphi(\lambda), \ \ 
        \left\lVert \varphi \right\rVert_{C^{1}(I_0)}<C.
\end{equation}
If there are two functions $q^{(1)}, q^{(2)}$ satisfying
$$\left\lVert  q^{(1)} \right\rVert_{C^{1}(I_0)}<C,\ \ \left\lVert  q^{(2)} \right\rVert_{C^{1}(I_0)}<C,$$
and we denote their respective solutions as $\lambda'=\lambda+\varepsilon \varphi^{(1)}(\lambda)$ and $\lambda'=\lambda+\varepsilon \varphi^{(2)}(\lambda)$, 
then we have the following inequality:
\begin{equation}
    \left\lVert\varphi^{(1)}-\varphi^{(2)}\right\rVert_{C^{0}(I_0)}\lesssim\left\lVert q^{(1)}-q^{(2)}\right\rVert_{C^{0}(I_0)}.
\end{equation}

\section{The Iterative Lemma }
We define the following iterative constants and domains:
\begin{itemize}
    \item $\alpha>1$, $L_0>0$ corresponding to the Gevrey function $f$.
    \item $r\ge 0$, the number of Newton iteration steps.
    \item $q_r(\lambda,\lambda')(\theta,\tilde{x})\in L^{2}(\mathbb{T}^{b+d})$, the approximate solution at the $r$-th step. 
    \item $q_0(\theta,\tilde{x})=\sum\limits_{1\le j\le b}a_j e^{i n_j\cdot \tilde{x}+i e_j\cdot \theta}$.
    \item $\lambda'=\lambda+\varepsilon\varphi_{r}(\lambda)$ is the solution of \begin{equation}
    (\lambda'_j-\lambda_j)a_j-\varepsilon\widehat{\frac{\partial H}{\partial \bar{q}}}(q_{r},\bar{q}_{r})(n_j,e_j)=0,\ \ 1\le j \le b.
    \notag
\end{equation}
    \item $A=(\log \frac{1}{\varepsilon})^{\alpha}$, $N_r=A^{r}$.
    \item $\epsilon_r=\varepsilon^{(\frac{4}{3})^{r}}$.
    \item $\delta_r=\frac{3 L_0}{10\pi^{2}r^2}$.
    \item $L_r=L_0-10\sum\limits_{j=1}^{r}\delta_j$, and for $1\le j\le 10$, $L_{r}^{(j)}=L_{r-1}-j\delta_{r}$, $L_{r}^{(0.j)}=L_{r-1}-\frac{j\delta_{r}}{10}$.
    \item $B(0,N):=\{x\in \mathbb{Z}^{b+d}:|x|<N\}$.
    \item $\Gamma_{r}$, the graph of $\lambda'=\lambda+\varepsilon\varphi_{r}(\lambda)$. 
\end{itemize}
Choose constants $C_1,C_2,C_3$ such that $C_1>C_2+2$ and $C_2>C_3>10$. 
We now state the following Iterative Lemma:

\begin{lemma}[\textbf{Iterative Lemma}]
 Let $r\ge 1$, and for each step $r$, there  exists an  approximate solution $q_r(\lambda,\lambda')(\theta,\tilde{x})\in L^{2}(\mathbb{T}^{b+d})$  defined on $I_0\times I_0\subset \mathbb{R}^{2b}$, and these solutions are $C^{1}$ with respect to $(\lambda,\lambda')$. They satisfy the following statements:
\begin{itemize}
    \item[\textbf{(r.1)}] $\textup{supp }  \hat{q}_r|_{\mathbb{Z}^{b+d}\setminus\mathcal{R}}\subset B(0,N_r)$, and $\hat{q}_{r}(x)\in\mathbb{R}$ for $x\in\mathbb{Z}^{b+d}$.
    \item[\textbf{(r.2)}] Let $\Delta q_r=q_r-q_{r-1}$. Then $\textup{supp } \widehat{\Delta q}_{r}\subset B(0,N_{r})\setminus\mathcal{R}$ and 
$\left\lVert \Delta q_r \right\rVert_{C^{1}(I_0\times I_0)}<\epsilon_{r-1}^{1-}$.
    \item[\textbf{(r.3)}] There exists a collection $\Lambda_r$ of intervals $I$ in $\mathbb{R}^{2b}$ of size $A^{-\alpha (r+2)^{C_1}}$,   such that:
    \begin{itemize}
        \item[\textbf{(r.3.a)}] Let $T=T_{q_{r-1}}$. For $(\lambda,\lambda')\in \mathop{\cup}\limits_{I\in \Lambda_{r}} I$,  we have
            \begin{equation}
                |T_{P,N_r}^{-1}((\nu,x),(\nu',x'))|<2B(r) e^{-\alpha L_r^{(1)} |x-x'|_{\alpha}}\ \  \text{for}\ \ \nu, \nu'=\pm,\    x,x'\in \mathbb{Z}^{b+d}, 
                \label{rbujvzhen}
            \end{equation}
        where $B(r)=A^{r^{C_3}}$.
    \item[\textbf{(r.3.b)}] Let  $I\in \Lambda_r$. Then  $q_{r}(\lambda,\lambda')$ is given by a rational function of $(\lambda,\lambda')$ of degree at most $A^{r^{3}}$ on $I$, and we have
        \begin{equation}
        \left\{
            \begin{split}
            &\left\lVert \Delta q_r\right\rVert_{\alpha,  L_{r}}<\epsilon_{r-1}^{1-},\\
            &\left\lVert \bm{\Gamma}_{P} F(q_r) \right\rVert_{\mathcal{F}_{\alpha,L_r}}\lesssim\epsilon_{r},\\
            &\left\lVert \bm{\Gamma}_{P'} \bar{F}(q_r) \right\rVert_{\mathcal{F}_{\alpha,L_r}}\lesssim\epsilon_{r}.
            \end{split}\right.
        \label{rbushuliang}
    \end{equation}
    \item[\textbf{(r.3.c)}] For $I\in \Lambda_{r}$, regarding $q_r$ as a function defined on $I\times \mathbb{T}^{b+d}$, we have $q_r,\Delta q_r\in G^{\alpha,A^{-(r+2)^{C_1}},L_r}(I\times \mathbb{T}^{b+d})$, and  
    $$\left\lVert \Delta q_{r} \right\rVert_{\alpha,A^{-(r+2)^{C_1}},L_r}<\epsilon_{r-1}^{1-}.$$
    Moreover,   $\left\lVert q_r \right\rVert_{\alpha,A^{- (r+2)^{C_1}},L_r}\ll 1$.
    \item[\textbf{(r.3.d)}] Each $I\in \Lambda_r$ is contained in an interval $I'\in \Lambda_{r-1}$. 
    For $r\gtrsim \log\log\frac{1}{\varepsilon}$, we have
        \begin{equation}
            \text{mes}_{b}\left( \Gamma_r \cap\left( \mathop{\cup}\limits_{I'\in \Lambda_{r-1}} I'\setminus\mathop{\cup}\limits_{I\in\Lambda_{r}}I \right) \right)<[\exp\exp(\log r)^{1/4}]^{-1}.
        \label{measureestimate}
        \end{equation}
    For $r\lesssim \log\log\frac{1}{\varepsilon}$, we have
        \begin{equation}
            \text{mes}_{b}\left( \Gamma_r \cap\left( I_0\times I_0 \setminus\mathop{\cup}\limits_{I\in\Lambda_{r}}I \right) \right)<(\log\log\frac{1}{\varepsilon})^{-\frac{1}{2}}.
        \label{measureestimateyouxianbu}
        \end{equation}
    \end{itemize}
    \item[\textbf{(r.4)}] On $\Gamma_r \cap (\mathop{\cup}\limits_{I\in \Lambda_r} I)$, we have
        \begin{equation}
        \left\lVert\frac{1}{i}\lambda'\cdot \partial_{\theta}q_{r}-\mathcal{L}q_r-\varepsilon\frac{\partial H}{\partial \bar{q}}(q_r,\bar{q_r})\right\rVert_{\mathcal{F}_{\alpha,L_r}}\lesssim\epsilon_{r}.
        \end{equation}
    i.e.,  $q_{r}(\lambda't,x)$ is an $O(\epsilon_{r})$ approximate solution of (\ref{e1}) for $(\lambda,\lambda')\in \Gamma_{r}$.
\end{itemize}
\label{irlnls}
\end{lemma}

When $r=0$, since $q_0$ is independent of $(\lambda,\lambda')$, it is easy to check that $\left\lVert q_0\right\rVert_{\alpha,L,L_0}\ll 1$ if we take $\{a_j\}$ sufficiently small. It is  also easy to check that $$\left\lVert\bm{\Gamma}_{P}F(q_0)\right\rVert_{\mathcal{F}_{\alpha,L_0}}\lesssim\varepsilon=\epsilon_{0}$$
and 
$$\left\lVert\bm{\Gamma}_{P'}\bar{F}(q_0)\right\rVert_{\mathcal{F}_{\alpha,L_0}}\lesssim\varepsilon=\epsilon_{0}.$$ In the sequel, when we say the Iterative Lemma holds at step 0, we mean that:
$$\left\lVert q_0\right\rVert_{\alpha,L,L_0}\ll 1, \ \ \left\lVert\bm{\Gamma}_{P}F(q_0)\right\rVert_{\mathcal{F}_{\alpha,L_0}}\lesssim\varepsilon=\epsilon_{0},\ \ \left\lVert\bm{\Gamma}_{P'}\bar{F}(q_0)\right\rVert_{\mathcal{F}_{\alpha,L_0}}\lesssim\varepsilon=\epsilon_{0}.$$    Moreover,  we denote $\Lambda_{0}=\{I_0\times I_0\}$.

\subsection{Estimate of the Green function in first finite steps}
     As mentioned in the previous section, the key point is to obtain the bound for the Green function. 
    Throughout the estimation of the Green function, one can easily observe that the index $\nu=\pm$ plays a minimal role in establishing the bound. Therefore,  for simplicity, we omit the index $\nu=\pm$ and write $T(x,x')=T((\nu,x),(\nu',x'))$ in the sequel. 
    
    In this subsection, we will prove \textbf{(r.3.d)} for $r\le C \log\log\frac{1}{\varepsilon}$ ($C>2\alpha$ can be large).
    
    Recall that $T_{q,P,N}=D_{P,N}-\varepsilon S_{q,P,N}$, and we have the following expressions:
\begin{equation}
    \begin{split}
        &\frac{\partial^{2} H}{\partial q\partial\bar{q}}(q,\bar{q})=f''(|q|^2)|q|^{2}+f'(|q|^2)q,\\
        &\frac{\partial^2 H}{\partial \bar{q}^2}(q,\bar{q})=f''(|q|^2)q^2,\\
        &\frac{\partial^2 H}{\partial q^2}(q,\bar{q})=f''(|q|^2)\bar{q}^2.
    \end{split}
    \nonumber
\end{equation}
Assuming the Iterative Lemma holds up to step $r-1$, then for $(\lambda,\lambda')\in \mathop{\cup}\limits_{I\in\Lambda_{r-1}} I$, we have 
$$\left\lVert \frac{\partial^2 H}{\partial q \partial \bar{q}}(q_{r-1}, \bar{q}_{r-1}) \right\rVert_{\alpha,L_{r-1}}<1,\left\lVert \frac{\partial^{2}H}{\partial \bar{q}^2} (q_{r-1}, \bar{q}_{r-1})\right\rVert_{\alpha,L_{r-1}}<1,\left\lVert \frac{\partial^2 H}{\partial q^2} (q_{r-1}, \bar{q}_{r-1})\right\rVert_{\alpha,L_{r-1}}<1.$$ 
Thus we have
\begin{equation}
    |T(x,x')|<\varepsilon \delta_{r}^{-C_{\alpha}}e^{-\alpha L_{r}^{(0.1)}|x-x'|_{\alpha}},
    \label{decay3}
\end{equation}
where $C_{\alpha}$ depends on $\alpha$. 

In fact,  we can also derive 
\begin{equation}
    |T(x,x')|\le
       e^{-\alpha L_r^{(0.2)}|x-x'|_{\alpha}}, \ \  \text{for}\ \   x\ne x',
    \label{decay1}
\end{equation}
by  applying  Lemma \ref{chajv}  and (l.3.b)$ (l\le r-1)$ inductively. 

We impose the following Diophantine condition on the vector $\lambda$:
\begin{equation}
    |k\cdot \lambda-\mu_{n}|>\gamma(1+|k|)^{-C},\ \ |k|<N_r,\  (n,k)\notin\mathcal{R},
    \label{diophantine}
\end{equation}
where $C>b+d$. We denote this condition by $DC_{\gamma,C}$, and we can easily derive that 
$$\text{mes}\{I_0\setminus DC_{\gamma,C}\}\lesssim \gamma.$$ Also, note that we have $\lambda'=\lambda+O(\varepsilon)$ by solving Q-equations at each step. It follows that 
\begin{align}
    |k\cdot \lambda'-\mu_n|&>|k\cdot \lambda-\mu_n|-|k\cdot O(\varepsilon)|\notag\\
        &>\frac{1}{2}\gamma(1+|k|)^{-C}
\end{align}
for $|k|<\gamma^{\frac{1}{2C}}\varepsilon^{-\frac{1}{2C}}$.
Let $\gamma=(\log\log\frac{1}{\varepsilon})^{-\frac{1}{2}}$, and since $N_r=A^{r}<e^{C(\log\log\frac{1}{\varepsilon})^{2}}$,  we have 
$$\left\lVert D_{P,N}^{-1} \right\rVert<2\gamma^{-1}N_{r}^{C}<e^{C(\log\log \frac{1}{\varepsilon})^{2}}.$$  
For simplicity, we omit the subscripts $q$ and $P$ when there is no confusion in the sequel.
Now we use the Neumann argument to estimate the bound of the Green function:

  We have
\begin{align}
    T_N^{-1}&=(D_N-\varepsilon S_N)^{-1}\notag\\
        &=D_N^{-1}(I-\varepsilon S_N D_N^{-1})^{-1}\notag\\
        &=D_N^{-1}+\mathop{\sum}\limits_{j=1}^{\infty}\varepsilon^{j} D_{N}^{-1}(S_N D_N^{-1})^{j}.
\end{align}
Notice that 
$$\left\lVert (I-\varepsilon S_N D_N^{-1})^{-1} \right\rVert<\frac{1}{1-\left\lVert 
\varepsilon S_N D_N^{-1} \right\rVert}<\frac{1}{1-\varepsilon^{1-}}<2,$$
thus we have $$\left\lVert T_N^{-1} \right\rVert<2\left\lVert D_N^{-1} \right\rVert<4\gamma^{-1}N^C<2e^{C(\log\log \frac{1}{\varepsilon})^{2}}.$$ 
Now we consider the off-diagonal decay of $T_N^{-1}$. We have
\begin{equation}
    |\varepsilon S_N D_N^{-1}(x,x')|<\varepsilon^{1-}|S_N(x,x')|<\varepsilon^{1-} \delta_{r}^{-C_{\alpha}} e^{-\alpha L_{r}^{(0.1)}|x-x'|_{\alpha}}.
\end{equation}
Thus we have
\begin{equation}
    \begin{split}
        |(\varepsilon S_N D_N^{-1})^{j}(x,x')|
        &=\left|\sum\limits_{k_1,...,k_{j-1}\in \mathbb{Z}^{b+d}}(\varepsilon S_N D_N^{-1})(x,k_1)\cdots (\varepsilon S_N D_N^{-1})(k_{j-1},x')\right|\\
        &\le \sum\limits_{k_1,...,k_{j-1}\in\mathbb{Z}^{b+d}}\varepsilon^{\frac{1}{4}j}(\delta_{r}^{-C_{\alpha}})^{j}e^{-\alpha L_r^{(0.1)}|x-x'|_{\alpha}}\\
        &\le (2N)^{(b+d)j} \varepsilon^{\frac{1}{4}j} (\delta_{r}^{-C_{\alpha}})^{j}e^{-\alpha L_{r}^{(0.1)}|x-x'|_{\alpha}}\\
        &\le ((2N)^{b+d}\delta_{r}^{-C_{\alpha}}\varepsilon^{\frac{1}{4}})^{j} e^{-\alpha L_{r}^{(0.1)}|x-x'|_{\alpha}}.
    \end{split}
    \notag
\end{equation}
Since $N_r<e^{C(\log\log\frac{1}{\varepsilon})^{2}}$,  we have 
\begin{equation}
    (2N_{r})^{b+d}\delta_{r}^{-C_{\alpha}}\varepsilon^{1/4}<\varepsilon^{\frac{1}{5}},
\end{equation}
thus 
$$|(\varepsilon S_N D_{N}^{-1})^j(x,x')|<(\varepsilon^{\frac{1}{5}})^j e^{-\alpha L_r^{(0.1)}|x-x'|_{\alpha}}.$$
So for $x\ne x'$, we have
\begin{equation}
    |T_N^{-1}(x,x')|\le \varepsilon^{\frac{1}{5}}\cdot2\gamma^{-1} N^{C} e^{-\alpha L_{r}^{(0.1)}|x-x'|_{\alpha}}<e^{-\alpha L_{r}^{(0.1)}|x-x'|_{\alpha}}.
\end{equation}
Thus, we have the bound (\ref{rbujvzhen}), and the measure estimate (\ref{measureestimateyouxianbu}) is straightforward to verify.

\subsection{Polynomials replacement in each step}
To estimate the Green functions in the following steps, we must impose further conditions on $(\lambda,\lambda')$. To make our techniques work,  we need to replace $T_{q}$, $\widehat{F(q)}$ and $\widehat{\bar{F}(q)}$ by polynomials of $(\lambda,\lambda')$ at each step. First, the approximate solution $q_{0}$ is a rational function of $(\lambda,\lambda')$. Actually, it does not depend on $(\lambda,\lambda')$.

\begin{lemma}
    Let $r\ge 1$, and suppose the Iterative Lemma holds up to  steps $r-1$. Then, there exists a collection $\Lambda_{r-1}^{(1)}$ of intervals $I\subset R^{2b}$ of size $A^{-\alpha ((r+2)^{C_1}-(r+2)^{C_2})}$. For each $I\in \Lambda_{r-1}^{(1)}$, there exist $T_{r-1}'$ and $F_{r-1}'$, whose elements are given by polynomials of degree at most $2^{r+1}\log\frac{1}{\varepsilon}$ in $(\lambda,\lambda')$. Furthermore, they satisfy the following inequalities on each $I$:     
    \begin{equation}
    \left\{
        \begin{split}
            &|(T_{r-1}-T_{r-1}')(x,x')|<\epsilon_{r}^2 e^{-\alpha L_{r}^{(0.1)}|x-x'|_{\alpha}},\\
            &|((F(q_{r-1})-F_{r-1}'))^{\land}(x)|<\epsilon_{r}^{2} e^{-\alpha L_{r}^{(0.1)}|x|_{\alpha}},\\
            &|\partial(T_{r-1}-T_{r-1}')(x,x')|<\epsilon_{r}^2 e^{-\alpha L_{r}^{(0.1)}|x-x'|_{\alpha}},\\
            &|\partial((F(q_{r-1})-F_{r-1}'))^{\land}(x)|<\epsilon_{r}^{2} e^{-\alpha L_{r}^{(0.1)}|x|_{\alpha}},
        \end{split}\right.
        \label{yaoyanzheng}
    \end{equation}
where $T_{r-1}:=T_{q_{r-1}}$, and $\partial$ refers to $\partial_\lambda$ or  $\partial_{\lambda'}$. Here, $C_1>C_2+2$.   
Moreover, each $I\in \Lambda_{r-1}^{(1)}$ is contained in an interval $I'\in \Lambda_{r-1}$ and $$\mathop{\cup}\limits_{I\in \Lambda_{r-1}^{(1)}} I=\mathop{\cup}\limits_{I'\in\Lambda_{r-1}} I'.$$
    \label{di}
\end{lemma}
\begin{proof}
    Take $I'\in \Lambda_{r-1}$. By our assumptions, we have 
    \begin{equation}
        \left\lVert q_{r-1} \right\rVert_{\alpha, A^{-(r+1)^{C_1}},L_{r-1}}\ll 1\ \  \text{for}\  I'\times \mathbb{T}^{b+d}.
        \label{control}
    \end{equation}
Mesh $I'$ into intervals of size $A^{-\alpha((r+2)^{C_1}-(r+2)^{C_2})}$. Take an interval $I$ from this mesh and assume $(\lambda_0,\lambda_0')$ is the center of  $I$. 
Regard $q_{r-1}$ as a function defined on $I\times \mathbb{T}^{b+d}$.
From (\ref{control}), we have 
    \begin{equation}
    \left\{
        \begin{split}
            &\left\lVert \frac{\partial^2 H}{\partial q \partial \bar{q}}(q_{r-1},\bar{q}_{r-1}) \right\rVert_{\alpha,A^{-(r+1)^{C_1}},L_{r-1}}<C,\\
            &\left\lVert \frac{\partial^2 H}{\partial q^2} (q_{r-1},\bar{q}_{r-1})\right\rVert_{\alpha,A^{-(r+1)^{C_1}},L_{r-1}}<C,\\
            &\left\lVert \frac{\partial^2 H}{\partial \bar{q}^2} (q_{r-1},\bar{q}_{r-1})\right\rVert_{\alpha,A^{-(r+1)^{C_1}},L_{r-1}}<C,\\
            &\left\lVert \frac{\partial H}{\partial \bar{q}}(q_{r-1},\bar{q}_{r-1}) \right\rVert_{\alpha,A^{-(r+1)^{C_1}},L_{r-1}}<C.
        \end{split}\right.
    \end{equation}
Take $\frac{\partial^2 H}{\partial q\partial \bar{q}}$ as an example. 
Let $\psi_{p}\left(\frac{\partial^2 H}{\partial q\partial \bar{q}}\right)(\lambda,\lambda')$ be the Taylor polynomial of degree $p$ for  $\frac{\partial^2 H}{\partial q\partial\bar{q}}(q_{r-1},\bar{q}_{r-1})$  centered at $(\lambda_{0},\lambda_{0}')$.  By Proposition \ref{bijinjino}, we have
\begin{align}
    &\left\lVert\frac{\partial^2 H}{\partial q \partial\bar{q}}(\lambda,\lambda')-\psi_{p}\left(\frac{\partial^2 H}{\partial q\partial \bar{q}}\right)(\lambda,\lambda')\right\rVert_{\alpha,L_{r-1}}\notag\\
    \le& C(p+1)!^{\alpha-1} (A^{(r+1)^{C_1}})^{(p+1)\alpha} (A^{-\alpha ((r+2)^{C_1}-(r+2)^{C_2})})^{p+1}
\end{align}
for $(\lambda,\lambda')\in I$. Taking $p+1= 2^{r+1}\log\frac{1}{\varepsilon}$, we obtain
\begin{equation}
    \begin{split}
        \left\lVert\frac{\partial^2 H}{\partial q \partial\bar{q}}-\psi_{p}\left(\frac{\partial^2 H}{\partial q\partial \bar{q}}\right)\right\rVert_{\alpha,L_{r-1}}
        &\le C((p+1)^{\alpha-1} A^{-\alpha (C_1-2)(r+1)^{C_1-1}})^{p+1}\\
        &\le A^{-\alpha (r+1)^{C_1-1} (p+1)}\\
        &\le \varepsilon^{\frac{\log A}{\log\frac{1}{\varepsilon}}\alpha (r+1)^{C_1-1}(p+1)}\\
        &\le \varepsilon^{\frac{\log A}{\log\frac{1}{\varepsilon}}\alpha (r+1)^{C_1-1}2^{r+1} \log\frac{1}{\varepsilon}}\\
        &\le \varepsilon^{2^{r+1}}.
    \end{split}
    \notag
\end{equation}
Thus, we have
\begin{equation}
    \left|\left(\frac{\partial^2 H}{\partial q \partial\bar{q}}-\psi_{p}\left(\frac{\partial^2 H}{\partial q\partial \bar{q}}\right)\right)^{\land}(x)\right|<\epsilon_{r}^{2} e^{-\alpha L_{r}^{(0.1)}|x|_{\alpha}}.
\end{equation}
Moreover, by Proposition \ref{Ap3.2}, we have 
\begin{equation}
\left\lVert \partial\frac{\partial^2 H}{\partial q \partial \bar{q}}(q_{r-1},\bar{q}_{r-1}) \right\rVert_{\alpha,\frac{1}{2}A^{-(r+1)^{C_1}},L_{r-1}}<2^{\alpha}C A^{\alpha(r+1)^{C_1}}.   
\end{equation}
 Similarly, we obtain
\begin{equation}
    \begin{split}
        \left\lVert\partial\frac{\partial^2 H}{\partial q \partial\bar{q}}-\partial\psi_{p}\left(\frac{\partial^2 H}{\partial q\partial \bar{q}}\right)\right\rVert_{\alpha,L_{r-1}}
        &\le 2C A^{(r+1)^{C_1}}
       A^{-\alpha (C_1-3)(r+1)^{C_1-1} (p+1)} \\
        &\le 2C A^{ (r+1)^{C_1}}
       A^{-\alpha (C_1-3)(r+1)^{C_1-1} 2^{r+1} \log \frac{1}{\varepsilon}}\\
       &\le A^{- 2^{r+1}\log \frac{1}{\varepsilon}}\\
       &\le \varepsilon^{2^{r+1}}.
    \end{split}
    \notag
\end{equation}
Thus, we have
\begin{equation}
    \left|\partial \left(\frac{\partial^2 H}{\partial q \partial\bar{q}}-\psi_{p}(\frac{\partial^2 H}{\partial q\partial \bar{q}})\right)^{\land}(x)\right|<\epsilon_{r}^{2} e^{-\alpha L_{r}^{(0.1)}|x|_{\alpha}}.
\end{equation}
The remainder of the argument proceeds similarly. Denote the Taylor polynomial of degree $p$ for a function $h$  centered at $(\lambda_{0},\lambda_{0}')$
by $\psi_{p}(h)$.
Let 
\begin{equation}
\begin{split}
    &\widetilde{T}_{r-1}'\\
    =&\left[
    \begin{pmatrix}
        \frac{1}{i}\lambda'\cdot\partial_{\theta}-\mathcal{L}& 0\\
        0 & -\frac{1}{i}\lambda'\cdot\partial_{\theta}-\mathcal{L}'
    \end{pmatrix}
    -\varepsilon
    \begin{pmatrix}
        \psi_{p}(\frac{\partial^2 H}{\partial q \partial \bar{q}}(q_{r-1},\bar{q}_{r-1})) &\psi_{p}(\frac{\partial^2 H}{\partial \bar{q}^2}(q_{r-1},\bar{q}_{r-1}))\\
        \psi_{p}(\frac{\partial^2 H}{\partial q^2}(q_{r-1},\bar{q}_{r-1})) & \psi_{p}(\frac{\partial^2 H}{\partial q\partial \bar{q}}(q_{r-1},\bar{q}_{r-1}))
    \end{pmatrix}
    \right],
\end{split}
\notag
\end{equation}

\begin{equation}
    T_{r-1}'=D-\varepsilon
    \begin{pmatrix}
        S_{\psi_{p}(\frac{\partial^2 H}{\partial q \partial \bar{q}}(q_{r-1},\bar{q}_{r-1}))}&S_{\psi_{p}(\frac{\partial^2 H}{\partial \bar{q}^2} (q_{r-1},\bar{q}_{r-1}))}\\
        S_{\psi_{p}(\frac{\partial^2 H}{\partial q^2} (q_{r-1},\bar{q}_{r-1}))}&S_{\psi_{p}(\frac{\partial^2 H}{\partial q \partial \bar{q}} (q_{r-1},\bar{q}_{r-1}))}
    \end{pmatrix}
\end{equation}
and 
\begin{equation}
    F_{r-1}'=\frac{1}{i}\lambda'\cdot \partial_{\theta}q_{r-1}-\mathcal{L}q_{r-1}-\varepsilon \psi_{p}(\frac{\partial H}{\partial \bar{q}}(q_{r-1},\bar{q}_{r-1})).
\end{equation}
Then they satisfy (\ref{yaoyanzheng}).
\end{proof}
From now on, we assume that the Iterative Lemma holds up to  step $r-1$.
We denote 
\begin{equation}
    S_{r-1}=\begin{pmatrix}
        S_{\psi(\frac{\partial^2 H}{\partial q \partial \bar{q}}(q_{r-1},\bar{q}_{r-1}))}&S_{\psi(\frac{\partial^2 H}{\partial \bar{q}^2} (q_{r-1},\bar{q}_{r-1}))}\\
        S_{\psi(\frac{\partial^2 H}{\partial q^2} (q_{r-1},\bar{q}_{r-1}))}&S_{\psi(\frac{\partial^2 H}{\partial q \partial \bar{q}} (q_{r-1},\bar{q}_{r-1}))}
    \end{pmatrix}
\end{equation}
for ease of reference later. 

 We will  replace $T_{q_{r-1}}$ and $F(q_{r-1})$ by $T_{r-1}'$ and $F_{r-1}'$ in (\ref{sol}), respectively.
 
First, we show that the following two inequalities:
\begin{equation}
    |T_{r-1,N_{r}}^{-1}(x,x')|<B(r)e^{-\alpha L_{r}^{(1)}|x-x'|_{\alpha}} 
\end{equation}
and
\begin{equation}
        |T_{r-1,N_{r}}^{\prime-1}(x,x')|<B(r)e^{-\alpha L_{r}^{(1)}|x-x'|_{\alpha}} 
\end{equation}
are essentially equivalent. This can be easily verified by applying Lemma \ref{di} and the following Lemma \ref{equi} whose proof is a standard application of Neumann series.

\begin{lemma}
    Assume that
    \begin{equation}
    \left\{
 \begin{split}
     &\left\lVert T_{N}^{-1} \right\rVert<B,\\
     &|T_{N}^{-1}(x,x')|<e^{-\alpha L|x-x'|_{\alpha}}\ \text{for} \ |x-x'|> D,
 \end{split}   
 \right.
 \label{youjvli}
\end{equation}
and there is another matrix $T'$ satisfying
\begin{equation}
    |T'(x,x')-T(x,x')|<\eta e^{-\alpha L|x-x'|_{\alpha}}.
    \label{diff}
\end{equation}
Furthermore, if $\eta N^{C_5}B^{2}e^{D}\ll 1$,
then 
\begin{equation}
\left\{
    \begin{split}
     &\left\lVert T_{N}^{\prime-1} \right\rVert<2B,\\
     &|T_{N}^{\prime-1}(x,x')|<2e^{-\alpha L|x-x'|_{\alpha}}\ \text{for} \ |x-x'|> D,
 \end{split}\right.  
\end{equation}
where $C_5=2(b+d)+1$ is a constant.

Additionally, if condition (\ref{youjvli}) is replaced by
\begin{equation}
     |T_{N}^{-1}(x,x')|<B e^{-\alpha L|x-x'|_{\alpha}},
\end{equation}
and if $\eta N^{C_5}B^2\ll 1$, then
\begin{equation}
     |T_{N}^{\prime -1}(x,x')|<2 B e^{-\alpha L|x-x'|_{\alpha}}.
\end{equation}

\label{equi}
\end{lemma}
Since we  replace $T_{q_{r-1}}, F(q_{r-1})$ by $T_{r-1}^{\prime}, F_{r-1}^{\prime}$ in (\ref{sol}),   we will solve the following equation at each step:
\begin{equation}
    T_{r-1,P,N}' \begin{pmatrix}
        \widehat{\Delta q}\\
        \widehat{\overline{\Delta q}}
    \end{pmatrix}
    =\begin{pmatrix}
        -R_{P,N} \widehat{F_{r-1}'}\\
        -R_{P',N} \widehat{\overline{F_{r-1}'}}
    \end{pmatrix}.
    \label{sol'}
\end{equation}
Let 
\begin{equation}
    \begin{pmatrix}
        \widehat{\Delta q_{r}}\\
        \widehat{\overline{\Delta q_{r}}}
    \end{pmatrix}
    =
    \left(  T_{r-1,P,N}' \right)^{-1}
    \begin{pmatrix}
        -R_{P,N} \widehat{F_{r-1}'}\\
        -R_{P',N} \widehat{\overline{F_{r-1}'}}
    \end{pmatrix}.
    \label{solution'}
\end{equation}
 By Proposition \ref{p1}, we know that the components of $T_{r-1}'$ and $F_{r-1}'$ are real. 
  Therefore, we have $\widehat{\Delta q_{r}}(x)\in \mathbb{R}$ for $x\in\mathbb{Z}^{b+d}$.
Thus (r.1) is proved.
 
 Besides, there will be more remainders at each step compared with the scheme we described in Section 4. In fact, we have 

 \begin{align}
     &\begin{pmatrix}
        \bm{\Gamma}_{P}F(q_{r-1}+\Delta q_{r})\\
        \bm{\Gamma}_{P'}\bar{F}(q_{r-1}+\Delta q_{r})
    \end{pmatrix}\notag\\
    =&
    \begin{pmatrix}
        \bm{\Gamma}_{P} F(q_{r-1})\\
        \bm{\Gamma}_{P'}\bar{F}(q_{r-1})
    \end{pmatrix}
    +\widetilde{T}_{q_{r-1},P}
    \begin{pmatrix}
        \Delta q_{r}\\
        \overline{\Delta q_{r}}
    \end{pmatrix}
    +O(\Delta q_{r}^{2})\notag\\
    =&\begin{pmatrix}
        (\bm{\Gamma}_{P}-\bm{\Gamma}_{P,N}) F(q_{r-1})\\
        (\bm{\Gamma}_{P'}-\bm{\Gamma}_{P',N})\overline{F(q_{r-1})}
    \end{pmatrix}
    +(\widetilde{T}_{q_{r-1},P}-\widetilde{T}_{q_{r-1},P,N})
    \begin{pmatrix}
        \Delta q_{r}\\
        \overline{\Delta q_{r}}
    \end{pmatrix}
    +O(\Delta q_{r}^{2})\notag\\
    &+\begin{pmatrix}
        \bm{\Gamma}_{P,N} \left(F(q_{r-1})-F_{r-1}'\right)\\
        \bm{\Gamma}_{P',N}\left(\bar{F}(q_{r-1})-\overline{F_{r-1}'}\right)
    \end{pmatrix}+
    \left(\widetilde{T}_{q_{r-1},P,N}-\widetilde{T}_{r-1,P,N}'\right)
    \begin{pmatrix}
        \Delta q_{r}\\
        \overline{\Delta q_{r}}
    \end{pmatrix}.\label{n.7'}
 \end{align}

Note that the first three terms were already discussed in Section 4, so we only need to estimate the last two terms. By  applying Lemma \ref{di}, we get 
\begin{equation}
    \left\lVert \begin{pmatrix}
        \bm{\Gamma}_{P,N} \left(F(q_{r-1})-F_{r-1}'\right)\\
        \bm{\Gamma}_{P',N}\left(\bar{F}(q_{r-1})-\overline{F_{r-1}'}\right)
    \end{pmatrix} \right\rVert_{\mathcal{F}_{\alpha,L_{r}}} <\epsilon_{r}^{2},
\end{equation}
and
\begin{equation}
    \left\lVert 
    \left(\widetilde{T}_{q_{r-1},P,N}-\widetilde{T}_{r-1,P,N}'\right)
    \begin{pmatrix}
        \Delta q_{r}\\
        \overline{\Delta q_{r}}
    \end{pmatrix}
    \right\rVert_{\mathcal{F}_{\alpha,L_{r}}} <\epsilon_{r}^{2},
\end{equation}
where $$\left\lVert \begin{pmatrix}
    h_1\\
    h_2
\end{pmatrix} \right\rVert_{\mathcal{F}_{\alpha,L}}:=\left\lVert 
h_2 \right\rVert_{\mathcal{F}_{\alpha,L}}+
\left\lVert 
h_2 \right\rVert_{\mathcal{F}_{\alpha,L}}.$$
Thus, after the replacement, we can  obtain (r.3.b) at each step.

\subsection{Construction of $q_{r}$}
In this subsection, we construct $q_{r}$ and $\Lambda_{r}$ satisfying (r.1), (r.2), (r.3.a), (r.3.b), (r.3.c) and (r.4).
To obtain (r.3.c), we need the following theorem:
\begin{theorem}[\textbf{Markov Inequality}]
Suppose $I\subset \mathbb{R}^{n}$ is an interval, and let $p$ be a polynomial of total degree k or less in n variables. Then we have 
\begin{equation}
    \left\lVert \nabla p\right\rVert_{C^{0}(I)}\le \frac{4 k^2}{\omega_{I}} \left\lVert p \right\rVert_{C^{0}(I)},
\end{equation}
where $$\omega_{I}=\min\limits_{\lVert u \rVert=1}\textup{dist }(\mathcal{H}_{u},\mathcal{H}_{-u}),$$
and $\mathcal{H}_{u}$  is the support hyperplane of  $I$  with outer normal vector $u$.

(see \cite{wilhelmsen1974markov} for a proof).
\end{theorem}
We can easily derive from the Markov Inequality that 
\begin{equation}
    \left\lVert \partial^{l} p \right\rVert_{C^{0}(I)}\le \frac{(4 k^2)^{|l|_{1}}}{(\omega_{I})^{|l|_{1}}} \left\lVert p \right\rVert_{C^{0}(I)}.
\end{equation}
\begin{lemma}
   Suppose  the Iterative Lemma holds up to  step $r-1$. Let   $T=T_{r-1}^{\prime}$, and
   let   $\widetilde{\Lambda_{r}}$ be the collection of real intervals $I$ of size $A^{-\alpha (r+2)^{C_1}}$, where each $I$ contains $(\lambda_0,\lambda_0')$  satisfying
 \begin{equation}
    |T_{P,N}^{\prime -1}(\lambda_0,\lambda_0')(x,x')|<B(r) e^{-\alpha L_{r}^{(1)}|x-x'|_{\alpha}}. 
    \label{chufadian}
\end{equation}
 Moreover, we assume that $I\subset I'$ for some $I'\in\Lambda_{r-1}^{(1)}$.
   Then for $I\in \widetilde{\Lambda_{r}}$  we have:
   \begin{itemize}
       \item[(1)] For $(\lambda,\lambda')\in I$,
            \begin{equation}
                |T_{P,N}^{\prime -1}|<2B(r) e^{-\alpha L_{r}^{(1)}|x-x'|_{\alpha}}, 
            \end{equation}
        where  $N=N_{r}$ and $B(r)$ is given in the Iterative lemma.
        \item[(2)] Let $\left\lVert \cdot \right\rVert_{\alpha,A^{- (r+2)^{C_1}},L_r}$ be the Gevrey norm for functions defined on $I\times \mathbb{T}^{b+d}$, we have
           \begin{align}
               &\left\lVert \Delta q_r \right\rVert_{\alpha,A^{- (r+2)^{C_1}},L_r}<\epsilon_{r-1}^{1-}, \\
                    & \left\lVert  q_r \right\rVert_{\alpha,A^{- (r+2)^{C_1}},L_r}\ll 1. 
           \end{align}
   \end{itemize}

Moreover,  if we replace $I$ by $\frac{3}{2}I$, (1) and (2) still hold.
\label{qvyvraodong}

\end{lemma}
\begin{proof}
Take $I\in\widetilde{\Lambda_{r}}$, and suppose $(\lambda_{0},\lambda_{0}')$ satisfies (\ref{chufadian}).
    Since the elements of $T_{r-1}^{\prime}$ and $(F_{r-1}^{\prime})^{\land}$ are polynomials, we can complexize $\lambda$ and $\lambda'$. 
Let $(\lambda,\lambda')$ be in the complex neighborhood of $(\lambda_0,\lambda_{0}')$ of size $10\sqrt{2b} A^{-\alpha (r+2)^{C_1}}$.   Let p be an element of $T_{r-1,P}'$ or $(\bm{\Gamma}_{P}F_{r-1}')^{\land}$.  Since the degree of the polynomial is at most $ 2^{r+1}\log \frac{1}{\varepsilon}$, and  $\omega_{I'}=A^{-\alpha ((r+2)^{C_1}-(r+2)^{C_2})}$, by the Markov Inequality,  we have
\begin{equation}
    |p^{(l)}(\lambda_0,\lambda_0')|\le \left(\frac{4\cdot 4^{r+1}(\log \frac{1}{\varepsilon})^{2}}{A^{-\alpha ((r+2)^{C_1}-(r+2)^{C_2})}}\right)^{|l|_{1}}\left\lVert p \right\rVert_{C(I')}.
\end{equation}
Let $(\Delta \lambda,\Delta\lambda')=(\lambda,\lambda')-(\lambda_0,\lambda_0')$, since 
\begin{equation}
    p(\lambda,\lambda')=\sum\limits_{l=0}^{\textup{degree}}(\Delta \lambda \cdot \partial_{\lambda}+\Delta \lambda'\cdot\partial_{\lambda'})^{l} p(\lambda_0,\lambda_0'),
\end{equation}
we have
\begin{equation}
    \begin{split}
        |p(\lambda,\lambda')-p(\lambda_0,\lambda_0')|
        &\le \sum\limits_{l=1}^{\textup{degree}} \left(2b \frac{4\cdot 4^{r}(\log \frac{1}{\varepsilon})^{2}}{A^{-\alpha ((r+2)^{C_1}-(r+2)^{C_2})}}\cdot 10\sqrt{2b} A^{-\alpha (r+2)^{C_1}}\right)^{l} \left\lVert p \right\rVert_{C(I')}\\
        &\le \sum\limits_{l=1}^{\textup{degree}} (A^{-\alpha (r+2)^{C_2} +r})^{l} \left\lVert p \right\rVert_{C(I')}\\
        &\le A^{-\alpha (r+2)^{C_2} +r} \left\lVert p \right\rVert_{C(I')}.
    \end{split}
    \notag
\end{equation}
Thus, for $x\notin\mathcal{R}$, we have 
\begin{equation}
\left\{
    \begin{split}
        &|T_{r-1}'(\lambda,\lambda')(x,x')-T_{r-1}'(\lambda_0,\lambda_0')(x,x')|<A^{-\alpha (r+2)^{C_2} +2r} e^{-\alpha L_{r}^{(0.2)}|x-x'|_{\alpha}},\\
        &|\left(F_{r-1}'(\lambda,\lambda')\right)^{\land}(x)-\left(F_{r-1}'(\lambda_0,\lambda_0')\right)^{\land}(x)|<\epsilon_{r-1} A^{-\alpha (r+2)^{C_2} +2r} e^{-\alpha L_{r}^{(0.2)}|x|_{\alpha}}.
    \end{split}\right.
\end{equation}
This implies 
\begin{equation}
    |\left(F_{r-1}'(\lambda,\lambda')\right)^{\land}(x)|<2\epsilon_{r-1} e^{-\alpha L_{r}^{(0.2)}|x|_{\alpha}},\ x\notin\mathcal{R}.
\end{equation}
Since $A^{-\alpha (r+2)^{C_2}+2r} N_r^{C_5} B(r)^2 \ll 1$ (since $C_2>C_3$), by Lemma \ref{equi}, we have
\begin{equation}
        |T_{P,N}^{\prime -1}(x,x')|<2B(r) e^{-\alpha L_{r}^{(1)}|x-x'|_{\alpha}} 
\end{equation}
at $(\lambda,\lambda')$.  Moreover by the same  argument as in (\ref{estimateofdq}),   we have
\begin{equation}
    \begin{split}
        |\widehat{\Delta q_{r}}(x)(\lambda,\lambda')|
        \le \epsilon_{r-1}^{1-} e^{-\alpha L_r^{(2)} |x|_{\alpha}}.
    \end{split}
\end{equation}
Let $\left\lVert \cdot \right\rVert_{\alpha,A^{-(r+2)^{C_1}},L_r}$ be the Gevrey norm for function defined on $I\times\mathbb{T}^{b+d}$. Fix $x\in\mathbb{Z}^{b+d}$. Since $\Delta q_{r}(x)(\lambda,\lambda')$ is a rational function of $(\lambda,\lambda')$, by Cauchy estimate, we have
\begin{equation}
    \left\lVert\partial^l \widehat{\Delta q_{r}}(x)\right\rVert_{C(I)}\le\frac{l!}{(8 A^{-\alpha (r+2)^{C_1}})^{l}}\epsilon_{r-1}^{1-}e^{-\alpha L_r^{(2)}|x|_{\alpha}}.
\end{equation}
Thus
\begin{equation}
    \left\lVert \widehat{\Delta q_r}(x)(\lambda,\lambda') \right\rVert_{\alpha,A^{-(r+2)^{C_1}},L_r}=\left\lVert \widehat{\Delta q_r}(x)(\lambda,\lambda') \right\rVert_{\alpha,A^{-(r+2)^{C_1}}}<\epsilon_{r-1}^{1-}e^{-\alpha L_{r}^{(2)}|x|_{\alpha}},
    \label{5.57}
\end{equation}
which implies that
\begin{equation}
\begin{split}
    \left\lVert \Delta q_r \right\rVert_{\alpha,A^{-(r+2)^{C_1}},L_r}
    &= \left\lVert\sum\limits_{x\in\mathbb{Z}^{b+d}} \widehat{\Delta q_r}(x)e^{ix\cdot\theta} \right\rVert_{\alpha,A^{-(r+2)^{C_1}},L_r}\\
    &\le\sum\limits_{x\in\mathbb{Z}^{b+d}} e^{\alpha L_r |x|_{\alpha}} \left\lVert \widehat{\Delta q_r}(x)(\lambda,\lambda') \right\rVert_{\alpha,A^{-(r+2)^{C_1}},L_r} \\
    &<\epsilon_{r-1}^{1-}.
\end{split}
\notag
\end{equation}
Thus,  we also have 
$$\left\lVert q_r\right\rVert_{\alpha,A^{- (r+2)^{C_1}},L_r }\ll 1.$$ 
Moreover, if $I$ is replaced by $\frac{3}{2}I$, the above discussion still holds.
\end{proof}
Thus, under the assumption that the Iterative Lemma holds up to $r-1$ step, if we take $\Lambda_{r}$ to be a subset of $\widetilde{\Lambda_{r}}$ and let 
\begin{equation}
    \begin{pmatrix}
        \widehat{\Delta q_{r}}\\
        \widehat{\overline{\Delta q_{r}}}
    \end{pmatrix}
    =
    \left(  T_{r-1,P,N}' \right)^{-1}
    \begin{pmatrix}
        -R_{P,N} \widehat{F_{r-1}'}\\
        -R_{P',N} \widehat{\overline{F_{r-1}'}}
    \end{pmatrix},
\notag
\end{equation}
then we have verified conditions (r.3.a), (r.3.b) and (r.3.c) of the Iterative Lemma.

Notice that, up to this point,  $\Delta q_{r}$ is only defined on $(\lambda,\lambda')\in \mathop{\cup}\limits_{I\in \Lambda_{r}} I$. We can apply the Whitney Extension Theorem to extend $\Delta q_{r}$ to the entire $(\lambda,\lambda')$ parameter space (see \cite{poschel1982integrability}\cite{stein1970singular} for the Whitney Extension Theorem). Suppose the distance of two distinct intervals $I\in \Lambda_{r}$ is greater than $[\exp\exp\exp(\log r)^{1/2}]^{-1}$, which only causes a negligible loss in measure compared with the estimate we have made (i.e., $[\exp\exp(\log r)^{1/4}]^{-1}$).     

Given (r.3.c), it follows that  
$$\left\lVert \widehat{\Delta q_r}(x)\right\rVert_{C^{1}(I)}<\epsilon_{r-1}^{1-}$$
for $I\in \Lambda_{r}$ and $x\in\mathbb{Z}^{b+d}\setminus\mathcal{R}$.
Thus we can use the Whitney Extension Theorem to extend $\widehat{\Delta q_r}(x)$ to $I_0\times I_0$, yielding $\widehat{\Delta q_r}(x)\in \mathbb{R}$ and $\left\lVert \widehat{\Delta q_r}(x)\right\rVert_{C^{1}(I_0\times I_0)}<\epsilon_{r-1}^{1-}$. Since $\textup{supp } \widehat{\Delta q_{r}}\subset B(0,N_r)$, we have $\left\lVert \Delta q_{r} \right\rVert_{C^{1}(I_0\times I_0)}<\epsilon_{r-1}^{1-}$.

At this point, under the assumption that the Iterative Lemma holds up to step $r-1$, we have established conditions (r.1) and (r.2). Furthermore, condition (r.4) is easy to verify.

\subsection{The reality of $\lambda'$ and the approximation of $Q$-equations}

Since $\hat{q}_{r}(x)\in\mathbb{R}$ for all $x\in\mathbb{Z}^{b+d}$, we have 
$\widehat{\frac{\partial H}{\partial \bar{q}}}(q_{r},\bar{q}_{r})(n_j,e_j)(\lambda,\lambda')\in\mathbb{R}$ by Proposition \ref{p1}.
Thus, the solution  $\lambda'=\lambda+\varepsilon\varphi_{r}(\lambda)$   of \begin{equation}
    (\lambda'_j-\lambda_j)a_j-\varepsilon\widehat{\frac{\partial H}{\partial \bar{q}}}(q_{r},\bar{q}_{r})(n_j,e_j)=0,\ \ 1\le j \le b,
    \notag
\end{equation}
is real.

For measure estimate, we will consider an approximation of $Q$-equations to obtain a semi-algebraic description. Since $\left\lVert q_{r} \right\rVert_{\alpha,A^{-(r+2)^{C_1}},L_r}\ll 1$ for $q_{r}$ on $I\times \mathbb{T}^{b+d}$, where $I\in \Lambda_{r}$, we have that $\widehat{\frac{\partial H}{\partial \bar{q}}}(q_{r},\bar{q}_{r})(n_j,e_j)(\lambda,\lambda')$ is a Gevrey function in $(\lambda,\lambda')$, and 
\begin{equation}
    \left\lVert \widehat{\frac{\partial H}{\partial\bar{q}}}(q_{r},\bar{q}_{r})(n_j,e_j) \right\rVert_{\alpha,A^{-(r+2)^{C_1}}}\lesssim 1.
\end{equation}
Let $I'\in \Lambda_{r}^{(1)}$, then the size of $I'$ is $A^{-\alpha((r+3)^{C_1}-(r+3)^{C_2})}$. Let $(\lambda_0,\lambda'_0)$ be the center of $I'$ and let $\psi_{p}(\lambda,\lambda')$ be the Taylor polynomial of   $\widehat{\frac{\partial H}{\partial\bar{q}}}(q_{r},\bar{q}_{r})(n_j,e_j)(\lambda,\lambda')$ of degree $p$ centered at $(\lambda_0,\lambda'_0)$.  Then for $(\lambda,\lambda')\in I'$ we have 
\begin{equation}
\begin{split}
      &\left|\widehat{\frac{\partial H}{\partial\bar{q}}}(q_{r},\bar{q}_{r})(n_j,e_j)(\lambda,\lambda')-\psi_p(\lambda,\lambda')\right|\\
      \le& C(p+1)!^{\alpha-1}(A^{(r+2)^{C_1}})^{\alpha(p+1)} (A^{-\alpha((r+3)^{C_1}-(r+3)^{C_2})})^{p+1}\\
      \le& C (p+1)^{(\alpha-1)(p+1)} A^{-\alpha((r+3)^{C_1}-(r+2)^{C_1}-(r+3)^{C_2})(p+1)}\\
      \le& C((p+1)^{\alpha-1}\cdot A^{-\alpha(C_1-2)(r+2)^{C_1-1}})^{p+1}.
\end{split}
\notag
\end{equation}
Take $p+1=2^{r+1} \log\frac{1}{\varepsilon}$, then we obtain
\begin{equation}
    \begin{split}
       &\left|\widehat{\frac{\partial H}{\partial\bar{q}}}(n_j,e_j)(\lambda,\lambda')-\psi_p(\lambda,\lambda')\right|\\
      \le& C A^{-\alpha (r+2)^{C_1-1}(p+1)}\\
      \le & \varepsilon^{\frac{\log A}{\log1/\varepsilon}\alpha (r+2)^{C_1-1}(p+1)}\\
      \le& \varepsilon^{2^{r+1}}.
    \end{split}
    \notag
\end{equation}
Thus, at step $r$, we can replace $Q$-equations by 
\begin{equation}
    (\lambda'_j-\lambda_j)a-\varepsilon \psi_p(\lambda,\lambda')=0,\ \ 1\le j\le b,
    \label{tidaiQ}
\end{equation}
on $I'$ with a perturbation of $\varepsilon^{2^{r+1}}$. Moreover, equations (\ref{tidaiQ}) have a semi-algebraic description of degree less than $A^{r}$. In the sequel, we will   denote  by $\Gamma_r'$ the graph of (\ref{tidaiQ}) defined on $I'\in \Lambda_{r}^{(1)}$. We will denote  the solution of (\ref{tidaiQ}) as $\lambda'=\lambda+\varepsilon\varphi_{r}'(\lambda)$.  Now we are only left with proving (r.3.d) for $r\gtrsim \log\log\frac{1}{\varepsilon}$.

\section{Measure Estimates}
In this section, we will prove (r.3.d) for $r\gtrsim\log\log\frac{1}{\varepsilon}$. For simplicity, we omit the subscript $P$.
\subsection{Coupling Lemma}
In this subsection, we present two coupling lemmas. Here, we  denote $c=\frac{1}{\alpha}$, where $\alpha$ is the Gevrey index. Let $T_{\Lambda}$ be the restriction of the matrix $T$ to $\Lambda\times\Lambda$, where $\Lambda$ is a subset of the index set, and we denote $G_{\Lambda}:=T_{\Lambda}^{-1}$. Recall that if 
 \begin{equation}
     \Lambda=\Lambda_1\cup \Lambda_{2},\ \  \Lambda_1\cap\Lambda_2=\varnothing,
 \end{equation}
where $\Lambda$, $\Lambda_1$ and $\Lambda_2$ are finite subsets, then we have the resolvent identity:
\begin{equation}
    G_{\Lambda}=(G_{\Lambda_1}+G_{\Lambda_2})-(G_{\Lambda_1}+G_{\Lambda_2})(T_{\Lambda}-T_{\Lambda_1}-T_{\Lambda_2})G_{\Lambda}.
\end{equation}
\textbf{The following coupling lemma is the key improvement that enable us to obtain the quasi-periodic solution with arbitrary Gevrey index $\alpha>1$.}
\begin{lemma}
    Assume that  $T$ satisfies the off-diagonal estimate 
    \begin{equation}
        |T(x,x')|<e^{-L |x-x'|_{\alpha}},\ \ x\ne x'.
    \end{equation}
Let $M_0,M_1\in \mathbb{Z}_{+}$, and suppose $ (\log M_0)^{\iota'}  \le \log \log M_1\le \iota\log M_0 $, where $\iota<c$ and $\iota'$ is a  small constant. Let  $Q_0=[-M_0,M_0]^d$ and  $\Omega=[-M_1,M_1]^b \times Q$, where $Q$ is an $M_1$ interval in  $\mathbb{Z}^{d}$. Assume that there is a decomposition $\Omega=\Omega_0\cup\Omega_1$ satisfying the following conditions: 
\begin{itemize}
    \item For $x=(k,n)\in \Omega_0$, define $V_{x}=(k+[-M_0,M_0]^{b})\times(n+Q_0)$, and we have
\begin{equation}
\left\{
\begin{split}
    &\left\lVert G_{V_{x}} \right\rVert\le e^{M_0^{\kappa}},\\
    &|G_{V_{x}}(\eta,\eta')|<e^{-L|\eta-\eta'|_{\alpha}},\ \ |\eta-\eta'|>M_0^{\theta},
\end{split}\right.
\end{equation}
where $\kappa<c$ and $0<\theta<1$.
\item $\Omega_1=\mathop{\cup}\limits_{\beta}\Omega_{1,\beta}$, with $\text{diam }\Omega_{1,\beta}<M_1^{\theta_1}$ and $\text{dist }(\Omega_{1,\beta},\Omega_{1,\beta'})>M_1^{\theta_2}$ for $\beta\ne\beta'$.
\item  The $M_1^{\theta_3}$-neighborhood $\widetilde{\Omega}_{1,\beta}$ of $\Omega_{1,\beta}$ satisfies 
\begin{equation}
    \left\lVert G_{\widetilde{\Omega}_{1,\beta}} \right\rVert<e^{M_1^{\theta_4}},
\end{equation}
\label{couplelemma1}

\end{itemize}
where $0<\theta_4<\theta_3<\theta_2<\theta_1<\theta<1$.  Furthermore, assume  $c>\frac{\theta_4}{\theta_3}$ and $c<\frac{\theta-\theta_1-\tilde{\delta}}{1-\theta_3}$ for 
 some small constant $\tilde{\delta}>0$. Then, we have 
\begin{equation}
\left\{
    \begin{split}
        &\left\lVert G_{\Omega} \right\rVert<e^{M_1^{\theta_4+}},\\
        &|G_{\Omega}(x,y)|<e^{-(L-2 M_{0}^{\max\{ \iota,\kappa,\theta c \}-c})|x-y|_{\alpha}},\ \ |x-y|>M_1^{\theta}.
    \end{split}\right.
\end{equation}
\label{6.1}
   
\end{lemma}
\begin{proof}
    For  $x\in \Omega_{0}$ and $\text{dist}(x,\Omega_1)>M_1^{\theta_3}$, let $W_{x}$ be the $2M_1^{\theta_3}$-interval centered at $x$. Let $\Lambda=W_x$, $\Lambda_1=V_{x'}$, $\Lambda_2=W_x\setminus V_{x'}$. Then, we have 
    \begin{equation}
        \begin{split}
            |G_{\Lambda}(x',y)|
            \le & |G_{\Lambda_1}(x',y)|+\sum\limits_{\substack{z\in\Lambda_1\\ w\in\Lambda_2}}|G_{\Lambda_1}(x',z) T(z,w) G_{\Lambda}(w,y)|\\
            \le & |G_{\Lambda_1}(x',y)|+\sum\limits_{\substack{z\in \Lambda_1\\ w\in \Lambda_2}} e^{M_0^{\kappa}}\cdot e^{L(b+d)M_0^{\theta c}} e^{-L|x'-z|_{\alpha}} e^{-L|z-w|_{\alpha}}|G_{\Lambda}(w,y)|\\
            \le & |G_{\Lambda_1}(x',y)|+M_1^{2\theta_3 (b+d)}e^{M_0^{\kappa}+L(b+d)M_0^{\theta c}}\max\limits_{w\in \Lambda_2} e^{-L|x'-w|_{\alpha}}|G_{\Lambda}(w,y)|\\
            \le & |G_{\Lambda_1}(x',y)|+M_1^{2\theta_3 (b+d)}e^{M_0^{\kappa}+L(b+d)M_0^{\theta c}} e^{-\delta M_{0}^c} \max\limits_{w\in \Lambda_2}  e^{-(L-\delta)|x'-w|_{\alpha}}|G_{\Lambda}(w,y)|\\
            \le & |G_{\Lambda_1}(x',y)|+\max\limits_{w\in \Lambda_2} e^{-(L-\delta)|x'-w|_{\alpha}}|G_{\Lambda}(w,y)|,
        \end{split}
        \notag
    \end{equation}
where $\delta=M_{0}^{\max\{ \iota,\kappa,\theta c \}-c}$.  Thus, we need $\iota<c$ and $\kappa<c$ to obtain an acceptable $\delta$, i.e., $\delta=M_0^{-\gamma}$ for some $\gamma>0$.  

The inequality we just derived tells us 
\begin{equation}
    \max |G_{\Lambda}(x',y)|\le \max |G_{\Lambda_{1}}(x',y)|+\frac{1}{2} \max |G_{\Lambda}(x',y)|,
\end{equation}
thus we  have
\begin{equation}
    \max |G_{\Lambda}(x,y)|<2e^{M_0^{\kappa}}.
\end{equation}
Let $|\eta-\eta'|>M_0^{\frac{10}{c}}$, then we have
\begin{equation}
\begin{split}
    |G_{\Lambda}(\eta,\eta')|&\le e^{-(L-\delta)|\eta-\eta'|_{\alpha}}e^{L(b+d)M_0^c+M_0^{\kappa}}\\
    &\le e^{-(L-\delta-\delta')|\eta-\eta'|_{\alpha}},
\end{split}
\end{equation}
where $\delta'=M_0^{-9}$. For simplicity, we make some symbol abuse in the following, i.e.,  we say
\begin{equation}
    |G_{W_x}(\eta,\eta')|<e^{-L|\eta-\eta'|_c}, \ \ |\eta-\eta'|>M_0^{\frac{10}{c}},
\end{equation}
but we need to remember that here $L$ is actually $L-\delta-\delta'$.

Now we consider $\Omega$, and let $x,y \in \Omega$.

    \textbf{Case 1}:  $\text{dist }(x,\Omega_1)>M_1^{\theta_3}$.

Let $W_x$ be the $2 M_1^{\theta_3}$-interval centered at $x$, and let $\Lambda_1=W_{x}$ and $  \Lambda_2=\Omega\setminus\Lambda_{1}$, then we have 
    \begin{equation}
    \begin{split}
        |G_{\Omega}(x,y)|
        &\le |G_{\Lambda_1}(x,y)|+\sum\limits_{\substack{z\in \Lambda_1\\ w\in\Lambda_{2}}} e^{M_0^{\kappa}}e^{M_0^{C_{\alpha}}}e^{-L|x-z|_{\alpha}}e^{-L|z-w|_{\alpha}}|G_{\Omega}(w,y)|\\
        &\le |G_{\Lambda_1}(x,y)|+M_1^{b+d} e^{M_0^{C_{\alpha}}}\max\limits_{w\in{\Lambda_2}}e^{-L|x-w|_{\alpha}}|G_{\Omega}(w,y)|\\
        &\le |G_{\Lambda_1}(x,y)|+ \max\limits_{w\in\Lambda_2} 
e^{-(L-\delta_1)|x-w|_{\alpha}} |G_{\Omega}(w,y)|,
    \end{split}
    \notag
    \end{equation}
where $\delta_1=M_1^{-\theta_3 c-}$. When $|x-y|>M_1^{\theta_3}$, the first term vanish, and we get 
  \begin{equation}
      |G_{\Omega}(x,y)|<\max\limits_{w\in\Omega\setminus W_x} 
e^{-(L-\delta)|x-w|_{\alpha}} |G_{\Omega}(w,y)|.
  \end{equation}
  Clearly, $|x-w|>M_1^{\theta_3}$.

  \textbf{Case 2}:  $\text{dist }(x,\Omega_1)<M_1^{\theta_3}$.

  If this case happens, $x$ must be in some $\widetilde{\Omega}_{1,\beta}$. Let $\Lambda_1=\widetilde{\Omega}_{1,\beta}$ and $ \Lambda_2=\Omega\setminus\Lambda_1$, we have
  \begin{equation}
      \begin{split}
          |G_{\Omega}(x,y)|
        &\le |G_{\Lambda_1}(x,y)|+\sum\limits_{\substack{z\in \Lambda_1,  w\in\Lambda_{2}}} |G_{\Lambda_1}(x,z)||T(z,w)| |G_{\Omega}(w,y)|\\
        &\le |G_{\Lambda_1}(x,y)|+\sum\limits_{\substack{z\in \Lambda_1,  w\in\Lambda_{2}}} e^{M_1^{\theta_4}}e^{-L|z-w|_{\alpha}}|G_{\Omega}(w,y)|\\
        &\le |G_{\Lambda_1}(x,y)|+M_1^{b+d} e^{M_1^{\theta_4}} \max\limits_{\substack{z\in \Lambda_1,  w\in\Lambda_{2}}} e^{-L|z-w|_{\alpha}}|G_{\Omega}(w,y)|\\
        &\le |G_{\Lambda_1}(x,y)|+e^{2M_1^{\theta_4}} \max\limits_{\substack{z\in \Lambda_1,  w\in\Lambda_{2}}} e^{-L|z-w|_{\alpha}}|G_{\Omega}(w,y)|.
      \end{split}
      \notag
  \end{equation}
  
    \textbf{Subcase 2.1}.  $\text{dist }(w,\Omega_{1,\beta})>2 M_1^{\theta_3}$.

    In this case, since $\textup{dist}(w,\widetilde{\Omega}_{1,\beta})>M_1^{\theta_3}$, it follows that $|z-w|>M_1^{\theta_{3}}$. Therefore, we have  
    $$e^{2 M_1^{\theta_4}}e^{-\delta_2 |z-w|_{\alpha}}<1,$$ 
    where $\delta_2=M_1^{\theta_4-\theta_3 c}$, given that $\theta_4<\theta_3 c$. Thus we obtain
    \begin{equation}
        |G_{\Omega}(x,y)|\le |G_{\Lambda_1}(x,y)|+\max\limits_{\substack{R=|z-w|>M_1^{\theta_3}\\ z\in\widetilde{\Omega}_{1,\beta}, w\in \Omega\setminus\widetilde{\Omega}_{1,\beta} }} e^{-(L-\delta_2)|z-w|_{\alpha}} |G_{\Omega}(w,y)|.
    \end{equation}
    When $|x-y|>2 M_1^{\theta_1}$, the first term vanishes.

    \textbf{Subcase 2.2}. $\text{dist }(w,\Omega_{1,\beta})\le 2M_{1}^{\theta_3}$.

In this case, since $\text{dist }(\Omega_{1,\beta},\Omega_{1,\beta'})>M_1^{\theta_2}$ for $\beta\ne\beta'$, we have $\text{dist }(w,\Omega_1)>M_1^{\theta_3}$, so $G_{\Omega}(w,y)$ can be estimated as in Case 1. We have
\begin{equation}
    \begin{split}
        |G_{\Omega}(x,y)|
        \le &|G_{\Lambda_1}(x,y)|+e^{2M_1^{\theta_4}} \max\limits_{\substack{z\in \Lambda_1, w\in\Lambda_2\\ \text{dist}(w,\Omega_{1,\beta})\le 2M_1^{\theta_3}}} e^{-L|z-w|_{\alpha}}|G_{\Omega}(w,y)|\\
        \le &|G_{\Lambda_1}(x,y)|+e^{2 M_1^{\theta_4}} \max\limits_{\substack{z\in \Lambda_1,  w\in\Lambda_2\\ \text{dist}(w,\Omega_{1,\beta})\le 2M_1^{\theta_3}}} e^{-L|z-w|_{\alpha}}|G_{W_{w}}(w,y)|\\
        &+e^{2M_1^{\theta_4}} \max\limits_{\substack{z\in \Lambda_1, w\in\Lambda_2\\ \text{dist}(w,\Omega_{1,\beta})\le 2M_1^{\theta_3}}} e^{-L|z-w|_{\alpha}}\max\limits_{\substack{w'\in \Omega\setminus W_{w}} } e^{-(L-\delta_1)|w-w'|_{\alpha}} |G_{\Omega}(w',y)|\\
        \le &|G_{\Lambda_1}(x,y)|+e^{2 M_1^{\theta_4}} |G_{W_w}(w,y)|+  e^{2M_1^{\theta_4}}\max\limits_{\substack{\text{dist }(w,\widetilde{\Omega}_{1,\beta})<M_1^{\theta_3}\\ |w'-w|>M_1^{\theta_3}}} e^{-(L-\delta_1)|w-w'|_{\alpha}}|G_{\Omega}(w',y)|\\
        \le & |G_{\Lambda_1}(x,y)|+e^{2 M_1^{\theta_4}} |G_{W_w}(w,y)|+      \max\limits_{\substack{\text{dist }(w,\widetilde{\Omega}_{1,\beta})<M_1^{\theta_3}\\ |w'-w|>M_1^{\theta_3}}} e^{-(L-\delta_1-\delta_2)|w-w'|_{\alpha}}|G_{\Omega}(w',y)|.
\end{split}
\notag
\end{equation}
When $|x-y|>M_1^{\theta_1}+3 M_1^{\theta_3}$,  the first two terms vanish. In general, when $|x-y|>2 M_1^{\theta_1}$, the first two terms vanish, i.e.,
\begin{equation}
    |G_{\Omega}(x,y)|\le \max\limits_{\substack{z\in\widetilde{\Omega}_{1,\beta}\\ |w'-z|>M_1^{\theta_3}}} e^{-(L-\delta_1-\delta_2)|z-w'|_{\alpha}}|G_{\Omega}(w',y)|.
\end{equation}

Note that when $|x-y|>2 M_1^{\theta_1}$, the first one or two terms vanish in all cases. In fact, we can apply Case 2.2 to replace all cases, so we have
\begin{equation}
    |G_{\Omega}(x,y)|\le 2 e^{3 M_1^{\theta_4}}+
    \max\limits_{\substack{|z-x|<2 M_1^{\theta_1}\\ |w-z|>M_1^{\theta_3}}} e^{-(L-\delta_1-\delta_2)|z-w|_{\alpha}}|G_{\Omega}(w,y)|.
    \label{zongeq}
\end{equation}
When $|x-y|>2M_1^{\theta_1}$, the first term vanishes:
\begin{equation}
    |G_{\Omega}(x,y)|\le 
    \max\limits_{\substack{|z-x|<2 M_1^{\theta_1}\\ |w-z|>M_1^{\theta_3}}} e^{-(L-\delta_1-\delta_2)|z-w|_{\alpha}}|G_{\Omega}(w,y)|.
    \label{jincheng}
\end{equation}
From (\ref{zongeq}), we obtain
\begin{equation}
    |G_{\Omega}(x,y)|<2 e^{3 M_1^{\theta_4}}+\frac{1}{2}\max |G_{\Omega}(x,y)|,
\end{equation}
thus 
$$\max |G_{\Omega}(x,y)|<4e^{3 M_1^{\theta_4}}.$$ 

Now we consider the off-diagonal decay. Assume that $|x-y|>M_1^{\theta}$. Our goal is to achieve the desired off-diagonal decay by iteratively applying inequality (\ref{jincheng}).
The iteration process continues until the distance between $w$ and $y$ becomes less than $2 M_{1}^{\theta_{1}}$, which serves as the termination condition.

In each iteration, the distance between $w$ and $y$ decreases relative to the distance between $x$ and $y$. Specifically, from (\ref{jincheng}), we observe that when the distance between $w$ and $y$ in all dimensions decreases by $2M_1^{\theta_1}$, we obtain a decay factor of at least $e^{-M_1^{\theta_3 c}}$.
We will analyze the decay in each dimension separately.  

Suppose that, in (\ref{jincheng}), for some $j$, the distance $|z_j-w_j|<M_{1}^{\theta_3}$ is repeated $M_{1}^{\theta-\theta_1-\tilde{\delta}}$ times for a small $\tilde{\delta}$. After $M_{1}^{\theta-\theta_1-\tilde{\delta}}$ iterations, the repeated application of (\ref{jincheng}) leads to a decay of at least $e^{- M_1^{\theta-\theta_1-\tilde{\delta}}\cdot M_1^{\theta_3 c}}$, which simplifies to $e^{-(b+d)LM_1^{c}}$ under the condition $\theta-\theta_1-\tilde{\delta}+\theta_3 c>c$.

Next, consider the case where $|z_j-w_j|<M_{1}^{\theta_3}$ occurs fewer than $M_{1}^{\theta-\theta_1-\tilde{\delta}}$ times. In this case, for every dimension, there is at most a loss of distance of $M_1^{\theta-\tilde{\delta}}$.  Without loss of generality, we  assume that the distance between $x$ and $y$ in all dimension (i.e., $|x_j-y_j|$) is greater than $M_1^{\theta-\tilde{\delta}/2}$. Otherwise, the dimension with the greatest distance will provide the decay for those dimensions where the distance is less than $M_1^{\theta-\tilde{\delta}/2}$, with a small  reduction in $L$ by $O(M_1^{-c\tilde{\delta}})$.

Let us now focus on each individual dimension.  Take $x_1$, $y_1$ for example.
Let $R$ denote the distance that contributes to the decay, then we have 
$$R+M_1^{\theta-\tilde{\delta}}\ge|x_1-y_1|.$$ 
We refer to the distance that contributes to the decay as the decay distance.
For $1\le j\le R/M_1^{\theta}$,  denote by $l_j R$  the sum of the decay distances provided by  (\ref{jincheng}) where $j M_1^{\theta_1} \le|z_1-w_1|<(j+1) M_{1}^{\theta_1}$. Denote by $l_0 R$  the sum of the decay distances provided by  (\ref{jincheng}) where $M_1^{\theta_3} \le|z_1-w_1|<M_{1}^{\theta_1}$.
 Clearly, we have 
 $$\sum\limits_{j=0}^{R/M_1^{\theta_1}} l_j \ge 1.$$ 
 Thus, the total decay across iteration is greater than
\begin{equation}
    \begin{split}
        \frac{l_0 R} {3M_1^{\theta_1}} M_1^{\theta_3 c}+\sum\limits_{j=1}^{R/M_1^{\theta_1}}\frac{l_j R}{(j+3)M_1^{\theta_1}} j^c M_1^{\theta_1 c}
        &\ge \sum\limits_{j=0}^{R/M_1^{\theta_1}}l_j (1-O(M_1^{-\theta+\frac{3\tilde{\delta}}{4}+\theta_1}))R^c\\
        &\ge (1-O(M_1^{-\theta+\frac{3\tilde{\delta}}{4}+\theta_1}))R^c\\
        &\ge (1-O(M_1^{-\theta+\frac{3\tilde{\delta}}{4}+\theta_1})) \left(1-\frac{M_1^{c(\theta_1-\tilde{\delta})}}{|x_1-y_1|^c}\right)|x_1-y_1|^c\\
        &\ge (1-O(M_1^{-c\tilde{\delta}/2}))|x_1-y_1|^c.
    \end{split}
    \notag
\end{equation}
The inequality above relies on the facts that
\begin{equation}
    \frac{R}{3 M_{1}^{\theta_1}} M_{1}^{\theta_3 c}\ge R^{c},
\end{equation}
and 
\begin{equation}
    \frac{j^{c} R}{(j+3)M_{1}^{\theta_1}}M_{1}^{\theta_1 c}\ge (1-O(M_{1}^{-\theta+\frac{3\tilde{\delta}}{4}+\theta_1}))R^{c},
    \label{yigbudengshiwoxianzaizhengzaicheck}
\end{equation}
under the assumption $$c<\frac{\theta-\frac{\tilde{\delta}}{2}-\theta_1}{\theta-\frac{\tilde{\delta}}{2}-\theta_3}.$$
The inequality (\ref{yigbudengshiwoxianzaizhengzaicheck}) follows from
\begin{equation}
    \frac{j}{j+3}\left(\frac{R}{j M_{1}^{\theta_{1}}}\right)^{1-c}\ge 1-O(M_{1}^{-\theta+\frac{3\tilde{\delta}}{4}+\theta_1}),
\end{equation}
which holds for all $j$. 
Actually, when $j\ge \frac{R}{M_{1}^{\theta_{1}+\frac{\tilde{\delta}}{4}}}$, we have
\begin{equation}
    \frac{j}{j+3}\left(\frac{R}{j M_{1}^{\theta_{1}}}\right)^{1-c}\ge 1-O(M_{1}^{-\theta+\frac{3\tilde{\delta}}{4}+\theta_1}),
    \notag
\end{equation}
and when $j< \frac{R}{M_{1}^{\theta_{1}+\frac{\tilde{\delta}}{4}}}$, we have
\begin{equation}
    \frac{j}{j+3}\left(\frac{R}{j M_{1}^{\theta_{1}}}\right)^{1-c}>1.
\end{equation}
We conclude that
\begin{equation}
    |G_{\Omega}(x,y)|<e^{-(L-O(M_1^{-c\tilde{\delta}/2}))|x-y|_{\alpha}}, \textup{ for } |x-y|>M_1^{\theta}.
\end{equation}
Finally, note that we have refreshed the parameter $L$ earlier. Hence, we obtain the final off-diagonal decay result:
\begin{equation}
     |G_{\Omega}(x,y)|<e^{-(L-2 M_{0}^{\max\{ \iota,\kappa,\theta c \}-c})|x-y|_{\alpha}}, \textup{ for } |x-y|>M_1^{\theta}.
\end{equation}
This complete the proof. 
\end{proof}

Now we give another coupling lemma.

\begin{lemma}
    Assume that $T$ satisfies the off-diagonal estimate 
    \begin{equation}
        |T(x,x')|<e^{-L |x-x'|_{\alpha}},\ \ x\ne x'.
    \end{equation}
Let $N=A^{r+1}$ (where $r\gtrsim\log\log\frac{1}{\varepsilon}$), $M=(\log N)^{C_6}$, and $0<\kappa,\theta<1$. Denote by $T_{A^r}$ the restriction of $T$ to $\{|x|<A^{r}\}$. The constants below satisfy the following conditions:
$$\kappa<c,\ cC_6>4,\ C_6(c-\kappa)>3,\ C_6 c(1-\theta)>3,\ 2 c C_6<C_3-1,\ C_3>10.$$ 
Additionally, assume that: 
\begin{itemize}
    \item For $|x|,|x'|<A^{r}$, we have
    \begin{equation}
            |G_{A^{r}}(x,x')|<A^{r^{C_3}} e^{- L|x-x'|_{\alpha}}.
    \end{equation}
    \item For $|x|\ge \frac{1}{2}A^r$, let $W_x$ be the $2M$-interval centered at $x$ (if the $2M$-interval centered at $x$ is not contained within $[-N,N]^{b+d}$, let $W_{x}$ be the $2M$-interval whose  boundary is as far as possible from $x$). We assume the following holds:
\begin{equation}
\left\{
    \begin{split}
        &\left\lVert G_{W_x}\right\rVert<e^{M^{\kappa}},\\
        &|G_{W_x}(\eta,\eta')|<e^{-L|\eta-\eta'|_{\alpha}},\ \ |\eta-\eta'|>M^{\theta}.
    \end{split}\right.
\end{equation}
\end{itemize}

Then we have 
\begin{equation}
            |G_{N}(x,x')|<A^{(r+1)^{C_3}} e^{- (L-O(r^{-3}))|x-x'|_{\alpha}}.
\end{equation}

\label{nianjie}
\end{lemma}

\begin{proof}
   Denote $\Lambda=[-N,N]^{b+d}$ and let $x,y\in\Lambda$.
    
    \textbf{Case 1}: $|x|\le \frac{1}{2} A^r$.

    Let $\Lambda_1=[-A^r, A^r]^{b+d}$ and $\Lambda_2=\Lambda\setminus\Lambda_1$. We have
\begin{equation}
    \begin{split}
        |G_{\Lambda}(x,y)|
        \le & |G_{\Lambda_1}(x,y)|+\sum\limits_{\substack{z\in \Lambda_1\\ w\in\Lambda_2}}
        |G_{\Lambda_1}(x,z)||T(z,w)||G_{\Lambda}(w,y)|\\
        \le & |G_{\Lambda}(x,y)|+\sum\limits_{\substack{z\in \Lambda_1\\ w\in\Lambda_2}} A^{r^{C_3}} e^{-L|x-z|_{\alpha}}e^{-L|z-w|_{\alpha}}|G_{\Lambda}(w,y)|\\
        \le & |G_{\Lambda_1}(x,y)|+N^{b+d} A^{r^{C_3}} \max\limits_{w\in\Lambda_2}e^{-L|x-w|_{\alpha}}|G_{\Lambda}(w,y)|.
    \end{split}
    \notag
\end{equation}
Obviously, in this case, $|w-x|>\frac{1}{2} A^{r}$, thus we have
\begin{equation}
    |G_{\Lambda}(x,y)|\le |G_{\Lambda_1}(x,y)|+\max_{w\in \Lambda_2}e^{-(L-O(A^{-r\cdot c-}))|x-w|_{\alpha}}|G_{\Lambda}(w,y)|.
\end{equation}

\textbf{Case 2}:    $|x|>\frac{1}{2}A^{r}$.

Let $\Lambda_1=W_x$ and $\Lambda_2=\Lambda\setminus W_{x}$. We have
\begin{equation}
    \begin{split}
    |G_{\Lambda}(x,y)|
    \le& |G_{\Lambda_1}(x,y)|+\sum\limits_{\substack{z\in \Lambda_1\\ w\in\Lambda_2}}|G_{\Lambda_1}(x,z)||T(z,w)||G_{\Lambda}(w,y)|\\
    \le & |G_{\Lambda_1}(x,y)|+\sum\limits_{\substack{z\in \Lambda_1\\ w\in\Lambda_2}} e^{M^{\kappa}+M^{\theta c}}e^{-L|x-z|_{\alpha}}e^{-L|z-w|_{\alpha}}|G_{\Lambda}(w,y)|\\
    \le & |G_{\Lambda_1}(x,y)|+N^{b+d}e^{M^{\kappa}+M^{\theta c}}\max\limits_{w\in \Lambda_2}e^{-L|x-w|_{\alpha}}|G_{\Lambda}(w,y)|\\
    \le & |G_{\Lambda_1}(x,y)|+e^{ (b+d)\log N +M^{\kappa}+M^{\theta c}}\max\limits_{w\in \Lambda_2}e^{-L|x-w|_{\alpha}}|G_{\Lambda}(w,y)|\\
    \le & |G_{\Lambda_1}(x,y)|+\max\limits_{w\in \Lambda_2}e^{-(L-O(r^{-3}))|x-w|_{\alpha}}|G_{\Lambda}(w,y)|.
    \end{split}
    \notag
\end{equation}
When $|x-y|>M$, the first term vanishes.

Combining Case 1 and Case 2, we get
\begin{equation}
    |G_{\Lambda}(x,y)|\le \max \{ A^{r^{C_3}}, e^{M^{\kappa}} \}+\frac{1}{2}\max\limits_{w} |G_{\Lambda}(w,y)|.
\end{equation}
Therefore, we obtain 
\begin{equation}
    \max_{x,y} |G_{\Lambda}(x,y)| \le 2\max\{A^{r^{C_3}},e^{M^{\kappa}}\}.
\end{equation}
Note that Case 1  is carried out no more than $2A$ times; otherwise, $2A\cdot \frac{1}{2}A^{r}=N$  provides sufficient off-diagonal decay. 
Thus, for $x,y\in\Lambda$, we have 
\begin{equation}
    |G_{\Lambda}(x,y)|
        \le A e^{-(L-O(r^{-3}))|x-y|_{\alpha}}A^{r^{C_3}+ M^{\kappa}+M^{c}}.
\end{equation}
    Since $2\kappa C_6<C_3 -1$, $2c C_6 <C_3 -1$ and $C_3>10$, we have $r^{C_3}+M^{\kappa}+M^{c}<r^{C_3}+2 (r+1)^{C_3-1}<(r+1)^{C_3}$ for $r\gtrsim\log\log\frac{1}{\varepsilon}$.
    Thus, we conclude
    \begin{equation}
    |G_{\Lambda}(x,y)|
        \le A^{(r+1)^{C_3}} e^{-(L-O(r^{-3}))|x-y|_{\alpha}},
\end{equation}
for $x,y\in \Lambda$.
\end{proof}

\subsection{The matrices $T_{N}^{\sigma}$ and its estimation}

Assume that the Iterative Lemma holds up to  step $r-1$. To control the Green function through a multi-scale analysis, we introduce  a  matrix $T_{r-1}^{\prime \sigma}$, which involves an additional parameter $\sigma$. For simplicity, in this subsection,  we denote $T_{r-1}^{\prime}$ as $T_{r-1}$. We define $T_{r-1}^{\sigma}$ as follows:
\begin{equation}
    T_{r-1}^{\sigma}=D^{\sigma}-\varepsilon S_{r-1},
\end{equation}
where 
\begin{equation}
    D^{\sigma}=\begin{pmatrix}
        k\cdot\lambda'+\sigma-\mu_n&0\\
        0&  -k\cdot\lambda'-\sigma-\mu_{-n}
    \end{pmatrix}
\end{equation}
and $S_{r-1}$ is the matrix obtained in Lemma \ref{di}. Clearly, $T_{r-1}^{\sigma}$ is defined on $\mathop{\cup}\limits_{I\in \Lambda_{r-1}^{(1)}}I$.  It is  easy to see that $S_{r-1}$ is also a Toeplitz operator and is therefore invariant under translation. Moreover, the entries of $T_{r-1}^{\sigma}$ are polynomials in $(\lambda,\lambda',\sigma)$. Denote $D_{\pm,n,k}^{\sigma}=\pm(k\cdot \lambda'+\sigma)-\mu_{\pm n}$.
For $M_0\in\mathbb{Z}_{+}$ and $Q$ an $M_{0}$-interval in $\mathbb{Z}^d$, denote by $T_{l,M_0,Q}^{\sigma}$ the restriction of $T_{l}^{\sigma}$ to $|k|<M_0$ and $n\in Q$. Observe that  we have
\begin{equation}
    |D_{\pm,n,k}^{\sigma}|>|\mu_{\pm n}-|k\cdot \lambda'+\sigma||.
\end{equation}
Note that for $n\notin \{ n_1,...,n_b \}$, we have $\mu_{n}=|n|^{2}$,  and for $n\in \{ n_1,...,n_b \}$, we have $\mu_{n}=|n|^{2}+O(1)$.
Hence, if $|D_{\pm,n,k}^{\sigma}|<C(M_0)$ and $|D_{\pm,n',k'}^{\sigma}|<C(M_0)$, we obtain
\begin{equation}
    ||n|^2-|n'|^2|<C|k-k'|+2C(M_0)+C.
    \label{diagonal}
\end{equation}

We will use the following arithmetical lemma, see \cite{bourgain2005green} for a proof.

\begin{lemma}
    Fix any large number $B$. There is a partition $\{\pi_{\zeta}\}$ of $\mathbb{Z}^d$ satisfying the following properties:
    \begin{equation}
        \text{diam } \pi_{\zeta}<B^{C_0},
    \end{equation}
    and 
    $$|n-n'|+||n|^2-|n'|^2|>B \textup{ if } n\in \pi_{\zeta}, n'\in \pi_{\zeta'}, \zeta\ne\zeta',$$   where $C_0=C(d)$.
    \label{jihe}
\end{lemma}

We first consider the simpler case where $Q$ is an $M_0$-interval in $\mathbb{Z}^{d}$, satisfying 
\begin{equation}
    \min\limits_{n\in Q}|n|>M_0.
\end{equation}
We also impose a Diophantine condition:
\begin{equation}
    |k\cdot \lambda'|>\gamma (1+|k|)^{-C}, |k|<M,
\end{equation}
where $\gamma$ is the same constant as in (\ref{diophantine}), and $C$ is a constant depending on $d$. We refer to this condition as $DC_{M}$. 

Without loss of generality, assume that for $(\lambda,\lambda')\in I$, where $I\in\Lambda_{l}$, $\lambda'$ satisfies $DC_{N_l}$. 

For $l\lesssim \log\log \frac{1}{\varepsilon}$, this is evident. For $l\gtrsim \log\log \frac{1}{\varepsilon}$, suppose that  this condition holds for $l-1$. Taking $M=N_{l-1}$ and $N=N_{l}$ in Lemma \ref{diufantudaxiao},  we again denote by  $\{I_{\zeta}\}$ the disjoint $N_{l}^{-C}$-intervals described in Lemma \ref{diufantudaxiao}. If the $A^{-\alpha (l+2)^{C_{1}}}$-interval $I\in \Lambda_{l}$ intersects with $I_{0}\times I_{\zeta}$ for some $ I_{\zeta}\in \{I_{\zeta}\}$, we delete it. This process results in the loss of a measure on the order $N_{l}^{-C}$. Compared with $[\exp\exp(\log l)^{1/4}]^{-1}$, this loss is negligible.

\begin{lemma}
    There exist constants  $\rho$, $\kappa$,  $c_2$, $C_{7}$ and $0<\theta<1$ such that
    $$
    \frac{2}{3}\rho<c<\theta-2\rho C_0 d,\ c_2=\frac{\rho^2}{2}, \ \kappa=\frac{3}{4}\rho^2,\ C_7>\frac{2}{c_2},\ \frac{2}{c_2}(\max\{\frac{c_2}{2},\kappa,\theta c\}-c)<-3.
    $$
    Assume the Iterative Lemma holds up to  step $r-1$.  
    Consider $T_{l}^{\sigma}$ ($l\le r-1$) defined on $\mathop{\cup}\limits_{I\in \Lambda_{l}^{(1)}} I$. 
    
     Let $M_0$ be a large integer, and let  $Q$ be an $M_0$-interval in $\mathbb{Z}^{d}$ satisfying 
\begin{equation}
    \min\limits_{n\in Q}|n|>M_0.
\end{equation}
 Then,
there exists a system of Lipschitz functions $\sigma_s=\sigma_s (\lambda,\lambda'), s<e^{(\log M_0)^{C_7}}$ defined on $I_0\times I_0$ (depending on $Q$), such that
\begin{equation}
    \left\lVert \sigma_s \right\rVert_{Lip}\le C M_0.
\end{equation}
For $(\lambda,\lambda')\in I$ such that $\lambda'$ satisfies $DC_{M_{0}}$, we have
\begin{equation}
    \left\lVert G_{l,M_0,Q}^{\sigma} \right\rVert <e^{M_0^{\kappa}}
    \label{good1}
\end{equation}
and
\begin{equation}
    |G_{l,M_0,Q}^{\sigma}(x,x')|<e^{-\alpha L_{l+1}^{(0.3)}|x-x'|_{\alpha}},\ \ |x-x'|>M_0^{\theta},
    \label{good2}
\end{equation}
provided that
\begin{equation}
    \min\limits_{s}|\sigma-\sigma_s(\lambda,\lambda')|>e^{-M_0^{c_2}}.
\end{equation}

\label{19.13}

\end{lemma}

\begin{proof}
Recall that we say $T_l$($l\le r-1$) satisfies (\ref{decay1}). We now provide a detailed  proof. For $l\lesssim \log\log \frac{1}{\varepsilon}$, (\ref{decay1}) can be readily verified  by applying (\ref{decay3}) and Lemma \ref{di}. Assume that $T_{l-1}(l\le r-1)$ already satisfies (\ref{decay1}),   we now prove that $T_{l}$ also  satisfies (\ref{decay1}). 

Applying   Lemma \ref{chajv} and Lemma \ref{di}, we obtain 
\begin{equation}
    |(T_{l}-T_{l-1})(x,x')|<(\epsilon_{l+1}^2+\epsilon_{l}^2+\left\lVert\Delta q_{l}\right\rVert_{\alpha,L_{l}}) \delta_l^{C_{\alpha}} e^{-\alpha L_{l+1}^{(0.2)}|x-x'|_{\alpha}}.
\end{equation}
Since $\left\lVert\Delta q_{l}\right\rVert_{\alpha,L_{l}}\le \epsilon_{l-1}^{1-}$ for $(\lambda,\lambda')\in \mathop{\cup}\limits_{I\in \Lambda_{l}^{(1)}} I$, we have
\begin{equation}
    |(T_{l}-T_{l-1})(x,x')|<\epsilon_{l-1}^{1-} e^{-\alpha L_{l+1}^{(0.2)}|x-x'|_{\alpha}}.
    \label{dituiguanxi}
\end{equation}
By applying the principle of induction, we have
\begin{equation}
    |T_{l}(x,x')|<e^{-\alpha L_{l+1}^{(0.2)}|x-x'|_{\alpha}}
\end{equation}
for $(\lambda,\lambda')\in \mathop{\cup}\limits_{I\in \Lambda_{l}^{(1)}} I$ and $l\le r-1$. Furthermore, noting that $\left\lVert \Delta q_l \right\rVert_{\alpha,A^{-(l+2)^{C_1}},L_l}<\epsilon_{l-1}^{1-}$, 
applying   Lemma \ref{chajv} and Lemma \ref{di} again gives us the following estimate:
\begin{equation}
    |\partial(S_{l}-S_{l-1})(x,x')|<\left(\epsilon_{l+1}^{2}+\epsilon_{l}^{2}+A^{\alpha(l+2)^{C_1}}\epsilon_{l-1}^{1-}\right) e^{-\alpha L_{l+1}^{(0.2)}|x-x'|_{\alpha}}.
    \label{dituiguanxidaoshu}
\end{equation}
Using induction again, we have
\begin{equation}
    \left\lVert\partial S_l\right\rVert<C.
\end{equation}

Now we prove Lemma \ref{19.13} by double induction.

First, we prove this lemma for $r=1$.  For  $M_{0}< (\log\frac{1}{\varepsilon})^{C_{10}}$($c_2 C_{10}<1$), taking the diagonal of $T_{1}$ as $\sigma_{s}(\lambda,\lambda')$(with a sign choice of $\pm$), and applying a standard Neumann argument,  we have
\begin{equation}
    \left\lVert G_{1,M_0,Q}^{\sigma} \right\rVert <e^{M_0^{\kappa}}
    \notag
\end{equation}
and
\begin{equation}
    |G_{1,M_0,Q}^{\sigma}(x,x')|<e^{-\alpha L_{2}^{(0.1)}|x-x'|_{\alpha}},\ \ |x-x'|>M_0^{\theta},
    \notag
\end{equation}
provided that
\begin{equation}
    \min\limits_{s}|\sigma-\sigma_s(\lambda,\lambda')|>e^{-M_0^{c_2}}.
    \notag
\end{equation}
Moreover, we have $\lVert \sigma_{s} \rVert_{Lip}\le C M_{0}$.

For convenience, we say the statement $P(l,M,L)$ holds if for any $M$-interval $Q$ satisfying 
\[
\min\limits_{n\in Q} |n|>M,
\]
there exists a system of Lipschitz functions $\sigma_{s}=\sigma_{s}(\lambda,\lambda'), s<e^{(\log M)^{C_{7}}}$ defined on $I_{0}\times I_{0}$ (depending on $Q$), such that $\lVert \sigma_{s} \rVert_{Lip}\le CM$. For $(\lambda,\lambda')\in I\in\Lambda_{l}^{(1)}$ such that $\lambda'$ satisfies $DC_{M}$, we have
\begin{equation}
    \left\lVert G_{l,M,Q}^{\sigma} \right\rVert <e^{M^{\kappa}}
    \notag
\end{equation}
and
\begin{equation}
    |G_{l,M,Q}^{\sigma}(x,x')|<e^{-\alpha L|x-x'|_{\alpha}},\ \ |x-x'|>M^{\theta},
    \notag
\end{equation}
provided that
\begin{equation}
    \min\limits_{s}|\sigma-\sigma_s(\lambda,\lambda')|>e^{-M^{c_2}}.
    \notag
\end{equation}

Let $\tilde{M}$ be an large integer, assume that we have confirmed the statement $P(1,M,L)$.
Let $f(M)=e^{M^{\frac{1}{2}c_{2}}}$ and denote by $f^{-1}$ the inverse of $f$. 
Combing the multiscale reasoning in the proof of  Lemma 19.13 of \cite{bourgain2005green} and Lemma \ref{couplelemma1}, we can confirm  the 
statement $P(1,M,L-(f^{-1}(M))^{\max\{\frac{c_{2}}{2}, \kappa, \theta c\}-c})$ for $\tilde{M}\le M<f(\tilde{M})$. Thus we have the statement $P(1, M, L-(f^{-1}(\tilde{M}))^{\max\{\frac{c_{2}}{2}, \kappa, \theta c\}-c})$ for $\tilde{M}\le M<f(\tilde{M})$. By induction from scale $\tilde{M}=(\log\frac{1}{\varepsilon})^{C_{10}}$, we have the statement $P(1, M, L')$ for all $M$, where 
\begin{equation}
    L'=L_{2}^{(0.1)}-\sum\limits_{j=-1}^{\infty}(f^{j}(\tilde{M}))^{\max\{\frac{c_{2}}{2}, \kappa, \theta c\}-c}>L_{2}^{(0.3)}.
\end{equation}
Thus, we have prove Lemma \ref{19.13} for $r=1$. We call the induction we have made by induction-1.


Assume that Lemma \ref{19.13} holds for $T_{l-1}^{\sigma}$. Using (\ref{dituiguanxi}) and Lemma \ref{equi}, we obtain the statement $P(r,M,L_{l}^{(0.4)})$ for $M<(\frac{5}{4})^{l}$.  We demonstrate  induction-1 from scale $\tilde{M}=(\frac{5}{4})^{l-1}$, then we obtain the statement $P(r,M,L')$ for all $M$, where 
\begin{equation}
\begin{split}
     L'&=L_{l}^{(0.4)}-\sum\limits_{j=-1}^{\infty}(f^{j}(\tilde{M}))^{\max\{\frac{c_{2}}{2}, \kappa, \theta c\}-c}>L_{l}^{(0.4)}-2(f^{-1}(\tilde{M}))^{\max\{\frac{c_{2}}{2}, \kappa, \theta c\}-c}\\
     &> L_{l}^{(0.4)}-C(l-1)^{\frac{2}{c_{2}}(\max\{\frac{c_{2}}{2}, \kappa, \theta c\}-c)}>L_{l}^{(0.4)}-C(l-1)^{-3}>L_{l+1}^{(0.3)},
\end{split}
\end{equation}
as long as $c_2$ is sufficiently small such that
\begin{equation}
    \frac{2}{c_2}(\max\{\frac{c_2}{2},\kappa,\theta c\}-c)<-3.
\end{equation}
This completes the proof.
\end{proof}

Next, we will prove an estimate for $G_{l,M_0,Q}^{\sigma}$, where  $Q=[-10 M_0,10 M_0 ]^d$. To establish this estimate,  some further restrictions in the  $(\lambda,\lambda')$-parameter space will be required. To describe these restrictions, we introduce the definition of sectional measure.

\begin{definition}
    We say that $\mathcal{A}\subset I_0\times I_0\subset \mathbb{R}^{2b}$ has sectional measure at most $\eta$, denoted by $mes_{sec}\mathcal{A}<\eta$, if the following condition  holds:

    Let $\varphi$ be a $C^1$-function defined on $I_{0}$, with $\left\lVert \nabla \varphi \right\rVert_{C^{0}(I_{0})}<10^{-2}$. Then 
    \begin{equation}
        mes\{ \lambda\in I_{0} | (\lambda, \lambda+ \varphi(\lambda)) \in\mathcal{A} \}<\eta.
    \end{equation}
\end{definition}
Before stating the estimate for $G_{l,M_0,Q}^{\sigma}$, we first present a lemma concerning  multi-scale reasoning.

\begin{lemma}
    Let $S$ be a matrix defined on $(\lambda,\lambda')\in \mathop{\cup}\limits_{I\in\Lambda} I$, where $\Lambda$ is a collection of disjoint intervals and $\mathop{\cup}\limits_{I\in\Lambda} I\subset I_{0}\times I_{0}$. The entries of $S$ are polynomials in $(\lambda,\lambda')$, and $\left\lVert \partial S \right\rVert<C$, while $|S(x,x')|<e^{-L|x-x'|_{\alpha}}$ for $x\ne x'$.
    Let $T^{\sigma}=D^{\sigma}-\varepsilon S$.
    Let $M_0,M_1\in \mathbb{Z}_{+}$ be large integers satisfying
    \begin{equation}
        M_0<M_1<\exp\exp(\log M_0)^{\frac{1}{10}},\ \ M_{1}\ge\log\frac{1}{\varepsilon},
    \end{equation}
and  assume that, for any $(\lambda,\lambda')\in I\in\Lambda$,  we have
\begin{equation}
    | k\cdot \lambda' |> \gamma |k|^{-C},  \textup{  for } |k|<M_1,
\end{equation}
where $\gamma$ is the same as before.
Assume further that:
\begin{itemize}
    \item  For $M\le M_0$ and $(\lambda,\lambda')\in I\in\Lambda$, $T_{M}^{\sigma}:=T_{M,Q=[-10M,10M]^{d}}^{\sigma}$ satisfies
\begin{equation}
\left\{
    \begin{split}
        &\left\lVert G_M^{\sigma} \right\rVert<e^{M^{\kappa'}},\\
        &|G_M^{\sigma} (x,x')|<e^{-L|x-x'|_{\alpha}},\ \ |x-x'|>M^{\theta},
    \end{split}\right.
    \label{heyhey}
\end{equation}
for all $\sigma$ outside a set of measure less than $e^{-M^{c_3}}$. 
    \item For $M\le M_1$ and $Q$, an $M$-interval in $\mathbb{Z}^d$ satisfying $\min\limits_{n\in Q} |n|>M$, there exists a system of Lipschitz functions $\sigma_s=\sigma_s (\lambda,\lambda')$ (associated with $Q$), with $s<e^{(\log M)^{C_7}}$, defined on $I_{0}\times I_{0}$, such that
\begin{equation}
    \left\lVert \sigma_{s} \right\rVert_{Lip} \le CM.
\end{equation}
Moreover, if $(\lambda,\lambda') \in I\in\Lambda$,  and $\min\limits_{s} |\sigma-\sigma_s (\lambda,\lambda')|>e^{-M^{c_2}}$, then
\begin{equation}
\left\{
    \begin{split}
        &\left\lVert G_{M,Q}^{\sigma} \right\rVert<e^{M^{\kappa}},\\
        &|G_{M,Q}^{\sigma}(x,x')|< e^{-L|x-x'|_{\alpha}}, \ \  |x-x'|>M^{\theta}.
    \end{split}\right.
    \label{hey}
\end{equation}
\end{itemize}

Then there exists a set $\mathcal{A}$, obtained as a union of intervals of size $[\exp\exp(\log\log M_1)^3]^{-1}$, such that
\begin{equation}
    mes_{sec} \mathcal{A}<[\exp\exp(\log\log M_3)^2]^{-1},
\end{equation}
and for $(\lambda,\lambda')\in \left( \mathop{\cup}\limits_{I\in \Lambda} I \right)\setminus \mathcal{A}$,
\begin{equation}
\left\{
\begin{split}
    &\left\lVert G_{M_1}^{\sigma} \right\rVert<e^{M_1^{\kappa'}},\\
    &|G_{M_1}^{\sigma}(x,x')|<e^{-(L-O(\exp(-(\log\log  M_1)^{4})))|x-x'|_{\alpha}},\ \ |x-x'|>M_1^{\theta},
\end{split}\right.
    \label{heyheyhey}
\end{equation}
    for all $\sigma$ outside a set of measure less than $e^{-M_1^{c_3}}$.
    
    Here, $\kappa,\theta,c_2$ are the  same as in Lemma \ref{19.13}, and  $M_3=M_1^{\frac{\tilde{\kappa}}{10}\cdot d}$, with constants $c, \kappa', \tilde{\kappa}, \theta, c_3$  satisfying 
    \begin{align}
        c=\frac{1}{\alpha}, \quad \quad c>\kappa, \quad \quad \tilde{\kappa}=10^{-2}d^{-1}\kappa',\\
        c_3=\frac{1}{4}\min\{ 10^{-4}c\tilde{\kappa}^{2}d,10^{-3} \kappa\tilde{\kappa}^{2}d \}, \quad\quad  2\kappa'<c<\theta-\frac{\tilde{\kappa}}{10}\cdot d.
    \end{align}

    \label{lemma5.4}
\end{lemma}

The proof framework of Lemma \ref{lemma5.4} is the same as that  in Chapter 19 of \cite{bourgain2005green}, except that the constants need to be precisely selected. For the convenience of the readers, we have placed the proof in the appendix.

Now, we use Lemma \ref{lemma5.4} to establish the estimate for  $T_{l,M,Q=[-10M,10M]^d}^{\sigma}$, which we refer to as the large deviation theorem for $T_{l,M,Q=[-10M,10M]^d}^{\sigma}$.

\begin{lemma}
    Assume that the Iterative Lemma holds up to  step $r-1$.  Consider $T_{l}^{\sigma}$ for $l\le r-1$. There exists $\Lambda_{l}^{(2)}\subset \Lambda_{l}^{(1)}$ such that we have the following statement holds:
    
    Let $\exp(\log\log M)^{3}\le l\le r-1$. Then, for all $(\lambda,\lambda')\in I\in\Lambda_{l}^{(2)}$, $T_{l,M,Q=[-10M,10M]^{d}}^{\sigma}$ satisfies that
    \begin{equation}
    \left\{
        \begin{split}
            &\left\lVert G_{l,M,Q=[-10M,10M]^d}^{\sigma} \right\rVert<e^{M^{\kappa'}},\\
            &|G_{l,M,Q}^{\sigma}(x,x')|<e^{-L_{l+1}^{(0.4)}|x-x'|_{\alpha}}, \ \ |x-x'|>M^{\theta},
        \end{split}\right.
    \end{equation}
for $\sigma$ outside a set of measure less than $e^{-M^{c_3}}$. Moreover,
\begin{equation}
    mes_{sec}\left(  \mathop{\cup}\limits_{I\in \Lambda_{l}^{(1)}} I\setminus \mathop{\cup}\limits_{I\in \Lambda_{l}^{(2)}} I \right)<[\exp\exp(\log l)^{1/4}]^{-1}.
\end{equation}

\label{5.5}

\end{lemma}

\begin{proof}
    We prove this lemma by induction.
    
    When $l\lesssim \log\log\frac{1}{\varepsilon}$, this lemma can be easily  checked using Neumann argument. Now,  suppose that this lemma holds for  $l'\le l $, where  $l\le r-2$. Since 
    $$|(T_{l+1}-T_{l})(x,x')|<\epsilon_{l}^{1-} e^{-\alpha L_{l+2}^{(0.2)}|x-x'|_{\alpha} }$$ 
    and $\epsilon_{l}^{1-}< e^{-M}$ for $\exp(\log\log M)^3\le l$, by Lemma \ref{equi} and the inductive assumption, we know that for $(\lambda,\lambda')\in \Lambda_{l+1}^{(1)}$ and $\exp(\log\log M)^{3}\le l$, we have 
    \begin{equation}
    \left\{
        \begin{split}
            &\left\lVert G_{l+1,M,Q=[-10 M,10M]^d}^{\sigma} \right\rVert<e^{M^{\kappa'}},\\
            &|G_{l+1,M,Q}^{\sigma}(x,x')|<e^{-\alpha L_{l+2}^{(0.4)}|x-x'|_{\alpha}},\ \  |x-x'|>M^{\theta},
        \end{split}\right.
    \end{equation}
for $\sigma$ outside a set of measure less than $e^{-M^{c_3}}$. Thus, we only need to establish this statement at scale $M_1$ satisfying
\begin{equation}
    l<\exp(\log\log M_1)^{3}\le l+1.
\end{equation}
Let $l_0=C \log M_1$  ($C>\frac{2}{\log\frac{4}{3}}$) and consider $T_{l_0}^{\sigma}$. If we manage to establish that for some $(\lambda,\lambda')$,
\begin{equation}
\left\{
        \begin{split}
            &\left\lVert G_{l_0,M_1,Q=[-10 M_1,10M_1]^d}^{\sigma} \right\rVert<e^{M_1^{\kappa'}},\\
            &|G_{l_0,M_1,Q}^{\sigma}(x,x')|<e^{-\alpha L_{l+2}^{(0.4)}|x-x'|_{\alpha}},\ \  |x-x'|>M_1^{\theta},
        \end{split}\right.
\end{equation}
for $\sigma$ outside a set of measure less than $e^{-M_1^{c_3}}$. Since $$|(T_{l+1}-T_{l_0})(x,x')|<\epsilon_{l_0-1}^{1-} e^{-\alpha L_{l+1}^{(0.2)}|x-x'|_{\alpha}}$$
and $\epsilon_{l_0-1}^{1-}< e^{-M_1}$, by Lemma \ref{equi}, we know that for the same $(\lambda,\lambda')$,
\begin{equation}
\left\{
        \begin{split}
            &\left\lVert G_{l+1,M_1,Q=[-10 M_1,10M_1]^d}^{\sigma} \right\rVert<e^{M_1^{\kappa'}},\\
            &|G_{l+1,M_1,Q}^{\sigma}(x,x')|<e^{-\alpha L_{l+2}^{(0.4)}|x-x'|_{\alpha}},\ \  |x-x'|>M_1^{\theta}.
        \end{split}\right.
\end{equation}
for $\sigma$ outside a set of measure less than $e^{-M_1^{c_3}}$.

Let $M_0=\exp(\log\log M_1)^{10}$, thus $\exp(\log\log M_0)^{3}<l_0$. We have for $M<M_0$,
\begin{equation}
\left\{
        \begin{split}
            &\left\lVert G_{l_0,M,Q=[-10 M,10M]^d}^{\sigma} \right\rVert<e^{M^{\kappa'}},\\
            &|G_{l_0,M,Q}^{\sigma}(x,x')|<e^{-\alpha L_{l_0+1}^{(0.4)}|x-x'|_{\alpha}},\ \  |x-x'|>M^{\theta},
        \end{split}\right.
    \end{equation}
for $\sigma$ outside a set of measure$< e^{-M^{c_3}}$ for $(\lambda,\lambda')\in I\in \Lambda_{l_0}^{(2)}$.

Take $\tilde{I}\in \Lambda_{l_0}^{(2)}$ and restrict $\tilde{I}$ in $DC_{M_1}$, then there are at most $M_1^{C}$ intervals $\tilde{I}'\subset \tilde{I}$. Denote the set composed  of all $\tilde{I}'$ by $\Lambda_{l_0}^{(2+)}$, and we know
\begin{equation}
   \#\{\tilde{I}'\}\le M_1^{C} A^{\alpha((l_0+3)^{C_1}-(l_0+3)^{C_2})}.
\end{equation}
Consider $T_{l_0}^{\sigma}$  on $\mathop{\cup}\limits_{\tilde{I}'\in\Lambda_{l_{0}}^{(2+)}}\tilde{I}'$. Lemma \ref{lemma5.4} tells us there is a set $\mathcal{A}$ satisfying 
\begin{equation}
    mes_{sec}\mathcal{A}<[\exp\exp(\log\log M_3)^2]^{-1},
    \label{shanchudedaxiao}
\end{equation}
and $\mathcal{A}$ is composed of intervals of size
\begin{equation}
    [\exp\exp(\log\log M_1)^3]^{-1}\sim e^{-l}.
    \label{sizededaxiao}
\end{equation}
For $(\lambda,\lambda')\in \mathop{\cup}\limits_{\tilde{I}'\in\Lambda_{l_{0}}^{(2+)}}\tilde{I}'\setminus \mathcal{A}$, we have the required statement (since $[\exp(\log\log M_1)^{4}]^{-1}\ll(l+1)^{-3}$). 

Now we define $\Lambda_{l+1}^{(2)}$.  If $I\in \Lambda_{l+1}^{(1)}$ intersects with $\mathcal{A}$, we delete it from $\Lambda_{l+1}^{(1)}$. We collect the remaining $I\in \Lambda_{l+1}^{(1)}$ in a set and denote it by $\Lambda_{l+1}^{(2)}$. 
Let $\Gamma$ be the graph of $\lambda'=\lambda+\varphi(\lambda)$ where $\lVert \nabla \varphi \rVert<10^{-2}$.
By (\ref{shanchudedaxiao}) and (\ref{sizededaxiao}), we have
\begin{align}
    & mes_{b}\left( \mathop{\cup}\limits_{I\in\Lambda_{l+1}^{(1)}} (\Gamma_{}\cap I) \setminus \mathop{\cup}\limits_{I\in\Lambda_{l+1}^{(2)}} (\Gamma_{}\cap I) \right)\notag\\
        <& \frac{e^{-l}+A^{-\alpha((l+3)^{C_1}-(l+3)^{C_2})}}{e^{-l}} [\exp\exp(\log\log M_1^{\frac{\tilde{\kappa}d}{10}})^2]^{-1}\notag\\
        <& \frac{e^{-l}+A^{-\alpha((l+3)^{C_1}-(l+3)^{C_2})}}{e^{-l}} [\exp\exp(\frac{1}{2}\log l_0)^2]^{-1}\notag\\
        <&[\exp\exp (\log l)^{1/4}]^{-1}.
\end{align}

Thus, the proof is completed.
\end{proof}

\subsection{Estimation of $G_{N_{r+1}}$}

Suppose that the Iterative Lemma holds up to step $r$ (the reason we assume the Iterative Lemma holds up to step $r$ here, rather than till $r-1$ as before, is solely for the sake of notation simplicity). In this section, we will generate $\Lambda_{r+1}$ by refining $\Lambda_{r}^{(2)}$ to ensure that $T_{r,N_{r+1}}^{-1}$ satisfies condition $(r.3.a)$. Once this is confirmed, the previous sections indicate that all other conditions at  step $r+1$, except for the measure estimate (\ref{measureestimate}), will hold. Therefore, our goal in this section is to find $\Lambda_{r+1}$ and obtain the measure estimate.

To invert $T_{r,N_{r+1}}$, consider a paving of $[-N_{r+1},N_{r+1}]^{b+d}$ by intervals $J$ of size $M=(\log N_{r+1})^{C_6}$, where $C_6$ is a sufficiently large constant. 
Since $T_{r-1,N_{r}}^{-1}$ satisfies 
\begin{equation}
|G_{r-1, N_{r}}(x,x')|< A^{(r-1)^{C_3}} e^{-\alpha L_{r}^{(1)} |x-x'|_{\alpha}},
\end{equation}
 and 
 $$|(T_{r-1}-T_{r})(x,x')|<\epsilon_{r-1}^{1-} e^{-\alpha L_r^{(0.2)}|x-x'|_{\alpha}}$$
 for $(\lambda,\lambda')\in I\in \Lambda_{r}$, we have
\begin{equation}
|G_{r, N_r}(x,x')|<A^{r^{C_3}} e^{-\alpha L_{r+1}^{(1)} |x-x'|_{\alpha}}.
\end{equation}

By Lemma \ref{nianjie}, we only need to ensure that for $M$-interval $J=([-M,M]^b+ k)\times Q$, where $|k|>\frac{A^{r}}{2}$ or $\min\limits_{n\in Q} |n|>\frac{A^{r}}{2}$,
we have
\begin{equation}
\left\{
    \begin{split}
        &\left\lVert G_{r,J} \right\rVert<e^{M^{\kappa''}},\\
        &|G_{r,J}(x,x')|<e^{-\alpha L_{r+1}^{(0.4)}|x-x'|_{\alpha}},
    \end{split}\right.
\end{equation}
for some $\kappa''$ satisfying the conditions in  Lemma \ref{nianjie}.
 When $\min\limits_{n\in Q} |n|>\frac{A^{r}}{2}, |k|<\frac{A^{r}}{2}$,  we  have
\begin{equation}
    |\pm k\cdot \lambda'-|n|^2|>A^{3r/2}.
\end{equation}
Thus, for those $J=([-M,M]^b+ k)\times Q$, we can invert the matrix by Neumann argument.  

Hence, it suffices to consider $\frac{1}{2} A^{r}\le |k|\le N$. Denote $T=T_{r}$ and
\begin{equation}
    R_{J}T R_{J}=T_{M,Q}^{\sigma=k\cdot \lambda'}
\end{equation}
and we distinguish the cases $\min\limits_{n\in Q}|n|>M$ and $Q=[-10 M,10 M]=:Q_0$. 

\textbf{Case 1}: $\min\limits_{n\in Q} |n|>M$.

Define $r_0=\exp(\log\log M)^{3}<r$ and restrict $(\lambda,\lambda')$ to the graph $\Gamma_{r_0}$.
If $\min\limits_{n\in Q} |n|>M$, by Lemma \ref{19.13}, we know for those fixed $Q$ there are at most $e^{(\log M)^{C_7}}$ Lipschitz functions $\sigma_s$ defined on $I_0\times I_0$ satisfy the statement in Lemma \ref{19.13}.  For fixed $Q$, the number of  lipschitz functions is at most $e^{(\log M)^{C_7}}$, and  if 
$$\min\limits_{s}|k\cdot \lambda'-\sigma_{s}(\lambda,\lambda')|>e^{-M^{c_2}},$$
we have
\begin{equation}
\left\{
    \begin{split}
        &\left\lVert G_{r_0,M,Q}^{\sigma=k\cdot \lambda'} \right\rVert<e^{M^{\kappa}},\\
        &|G_{r_0,M,Q}^{\sigma=k\cdot \lambda'}(x,x')|<e^{-\alpha L_{r+1}^{(0.3)}|x-x'|_{\alpha}},\text{  for  } |x-x'|>M^{\theta}.
    \end{split}\right.
    \label{xiaoggod}
\end{equation}

If we can ensure
\begin{equation}
    \min\limits_{Q} \min\limits_{s} |k\cdot \lambda'-\sigma_{s}(\lambda,\lambda')|>e^{-M^{c_2}}
\end{equation}
for all $|k|>\frac{A^r}{2}$, then (\ref{xiaoggod}) holds for all $|k|>\frac{A^r}{2}$ and $Q$.

Denote
\begin{equation}
    R_1=\{\lambda | \min\limits_{Q} \min\limits_{s} |k\cdot (\lambda+\varepsilon \varphi_{r_0}(\lambda) )-\sigma_{s}(\lambda,\lambda+\varepsilon \varphi_{r_0}(\lambda))|\le e^{-M^{c_2}}, \forall \  |k|>\frac{A^{r}}{2}\}.
\end{equation}
After excluding $R_1$, (\ref{xiaoggod})  holds for $\lambda'=\lambda+\varepsilon\varphi_{r_0}$ for all $|k|>\frac{A^{r}}{2}$ and $Q$.
Since $\left\lVert \sigma_{s} \right\rVert_{Lip}\le C M$, if $c_2 C_6>1$,  we have
\begin{equation}
    mes_{b} R_1 \le N_{r+1}^{b+d} e^{(\log M)^{C_7}} e^{-M^{c_2}}<e^{-\frac{1}{2} M^{c_2}}<A^{-(r+1)^{c_2 C_6}}.
\end{equation}

\textbf{Case 2}: $Q=Q_0:=[-10 M,10M]^{d}$.

Again define $r_0=\exp(\log\log M)^{3}<r$. By Lemma \ref{5.5}, we know that for $(\lambda,\lambda')\in I'\in \Lambda_{r_0}^{(2)}$, we have
\begin{equation}
\left\{
\begin{split}
     &\left\lVert G_{r_0,M,Q_0}^{\sigma} \right\rVert<e^{M^{\kappa'}},\\
     &|G_{r_0,M,Q_0}^{\sigma}(x,x')|<e^{-\alpha L_{r_0+1}^{(0.4)}|x-x'|_{\alpha}},\text{  for  } |x-x'|>M^{\theta},
\end{split}\right.
\label{ogood}
\end{equation}
for $\sigma$ outside a set of measure$<e^{-M^{c_3}}$. Consider 
\begin{equation}
    \mathcal{O}=\{(\lambda,\lambda',\sigma)\in I'\times \mathbb{R}| T_{r_0,M,Q_0} \text{ fails (\ref{ogood})}      \}
\end{equation}
and 
\begin{equation}
    \mathcal{S}=\mathcal{O}\cap((\Gamma_{r_0}'\cap I')\times \mathbb{R}).
    \label{singularset}
\end{equation}
The set  $\mathcal{S}$  is semi-algebraic of degree at most
\begin{equation}
    (M A^{r_0})^C <\exp\exp(\log\log r)^{4}.
\end{equation}
We can restrict $\sigma$ to $[-M^{C},M^{C}]$, since otherwise (\ref{ogood}) holds by Neumann argument. We decompose $[-M^{C},M^{C}]$ into intervals of length $1$ and identify each of them with $[0,1]$. Then $\mathcal{S}$ is divided into $M^{C}$ semi-algebraic subsets $\mathcal{S}'$. 
By Lemma \ref{5.5}, the measure of each $1$-dimensional $(\lambda,\lambda')$-section is less than  $e^{-M^{c_3}}$, so we have
\begin{equation}
    mes_{1+b} \mathcal{S}'<e^{-M^{c_3}}.
    \notag
\end{equation}
Take $C_6$ such that $c_3 C_6>1$, then we have  
$$mes_{b+1} \mathcal{S}'<e^{-M^{c_3}}<A^{-(r+1)^{c_3 C_6}}$$

Our aim is to estimate the measure for $|k|>\frac{1}{2}A^{r}$:
\begin{equation}
    mes_{b}\{ \lambda| (\lambda, \lambda+\varepsilon \varphi_{r_0}', k\cdot(\lambda+\varepsilon \varphi_{r_0}')) \in \mathcal{S}' \}.
\end{equation}

  Since $\log (M A^{r_0})^{C}\ll \log A^{(r+1)^{c_3 C_6}}$, we can apply Lemma \ref{fenjiedingli} to the semi-algebraic set $\mathcal{S}'$ by identifying the algebraic curve $\Gamma_{r_{0}}'\cap I'$ with an interval in $\mathbb{R}^{b}$. Let $\tilde{\varepsilon}$ (from Lemma \ref{fenjiedingli}) be $\tilde{\varepsilon}=A^{-(r+1)/2}$.  Then we obtain the decomposition: 
\begin{equation}
    \mathcal{S}'=\mathcal{S}_1\cup\mathcal{S}_2
\end{equation}
such that:
\begin{itemize}
    \item The set $\mathcal{S}_1$ satisfies
\begin{equation}
    mes_{b} \Pi_{\Gamma_{r_{0}}} \mathcal{S}_1<(M A^{r_0})^{C} \tilde{\varepsilon}<A^{-\frac{r+1}{3}},
\end{equation}
where $\Pi_{\Gamma_{r_{0}}}$ denotes the projection onto $\Gamma_{r_0}$.
\item  For $|k|>\frac{1}{2}A^{r}$, suppose $|k_1|>1$ for example. Fix $\lambda_2,...,\lambda_{d}$, then we can identify 
$$\bm{L}=\{(\lambda,\lambda+\varepsilon \varphi_{r_0}(\lambda),k\cdot (\lambda+\varepsilon \varphi_{r_0}(\lambda)))\}$$
as a hyperplane satisfying condition (\ref{xielvhaodao}). Thus, we have
\begin{equation}
    mes_{1}\{ \lambda_1| (\lambda,\lambda+\varepsilon \varphi_{r_0}(\lambda),k\cdot (\lambda+\varepsilon \varphi_{r_0}(\lambda)))\in \mathcal{S}_2 \}
    <(M A^{r_0})^{C} A^{\frac{1}{2} r}\cdot e^{-\frac{1}{2b} M^{c_3}}.
\end{equation}
\end{itemize}
By Fubini's theorem, we have
\begin{align}
    mes_{b}\{ \lambda| (\lambda,\lambda+\varepsilon \varphi_{r_0}(\lambda),k\cdot (\lambda+\varepsilon \varphi_{r_0}(\lambda)))\in \mathcal{S}_2\}<(M A^{r_0})^{C} A^{\frac{1}{2} r}\cdot e^{-\frac{1}{2b} M^{c_3}}.
\end{align}
Thus, we have
\begin{align}
    &mes_{b}\{ \lambda| (\lambda,\lambda+\varepsilon \varphi_{r_0}(\lambda),k\cdot (\lambda+\varepsilon \varphi_{r_0}(\lambda)))\in \mathcal{S}_2 \forall \ \frac{1}{2} A^{r}<|k|<N_{r+1} \}\notag\\
        <&A^{(r+1)b}\cdot (M A^{r_0})^{C} A^{\frac{1}{2} r}\cdot e^{-\frac{1}{2b} M^{c_3}}\notag\\
        <&A^{-(r+1)}.
\end{align}
        
Therefore, there exists a set $\Gamma'\subset \Gamma_{r_0}'\cap I'$ such that
\begin{equation}
    mes_{b}\Gamma'<M^{C}(A^{-\frac{r+1}{3}}+A^{-(r+1)})<A^{-\frac{r+1}{4}},
\end{equation}
and for all $(\lambda,\lambda')\in (\Gamma_{r_0}\setminus \Gamma')\cap I'$ and $\frac{1}{2} A^{r}<|k|<N_{r+1}$,   (\ref{ogood}) holds.   The number of intervals $I'$ is at most $A^{\alpha (r_0+2)^{C_3}}$, so the total measure of the excluded $\Gamma''$ is at most
\begin{equation}
    A^{\alpha(r_0+2)^{C_3}} (A^{-\frac{1}{4}(r+1)}+A^{-(r+1)})<A^{-\frac{1}{5} (r+1)}.
\end{equation}

Now we construct $\Lambda_{r+1}$.
Since (\ref{xiaoggod}) and (\ref{ogood}) allow $O(e^{-M})$ perturbation, $\epsilon_{r_0-1}^{1-}\ll e^{-M}$, $\left\lVert\varphi_{r}-\varphi_{r_0}\right\rVert_{C^{0}(I_0)}<\epsilon_{r_0-1}^{1-}$ and $|(T_{r}-T_{r_0})(x,x')|<\epsilon_{r_0-1}^{1-} e^{-\alpha L_{r+1}^{(0.2)}|x-x'|_{\alpha}}$, we can obtain a subset $\widetilde{\Gamma}_r\subset\Gamma_{r}$, such that 
\begin{equation}
\left\{
    \begin{split}
        &\left\lVert G_{r,M,Q}^{\sigma=k\cdot \lambda'} \right\rVert<e^{M^{\kappa}},\\
        &|G_{r,M,Q}^{\sigma=k\cdot \lambda'}(x,x')|<e^{-\alpha L_{r+1}^{(0.3)}|x-x'|_{\alpha}},\text{  for  } |x-x'|>M^{\theta},
    \end{split}\right.
    \label{xiaoggodr}
\end{equation}
and 
\begin{equation}
\left\{
\begin{split}
     &\left\lVert G_{r,M,Q_0}^{\sigma} \right\rVert<e^{M^{\kappa'}},\\
     &|G_{r,M,Q_0}^{\sigma}(x,x')|<e^{-\alpha L_{r+1}^{(0.4)}|x-x'|_{\alpha}},\text{  for  } |x-x'|>M^{\theta},
\end{split}\right.
\label{ogoodr}
\end{equation}
hold on 
\begin{equation}
    \mathop{\cup}\limits_{I'\in\Lambda_{r_0}^{(2)}}(I'\cap(\Gamma_r\setminus \widetilde{\Gamma}_{r})).
\end{equation}
Hence,  (\ref{xiaoggodr}) and (\ref{ogoodr})  hold on 
\begin{equation}
    \mathop{\cup}\limits_{I\in\Lambda_{r}^{(2)}}(I\cap(\Gamma_r\setminus \widetilde{\Gamma}_{r})).
\end{equation}
Moreover,  we have 
$$\textup{mes}_{b}\widetilde{\Gamma}_{r}<A^{-(r+1)/5}+A^{-(r+1)}<A^{-(r+1)/6}.$$
We mesh all $I\in \Lambda_{r}^{(2)}$ into intervals of size $A^{-\alpha (r+3)^{C_1}}$. If an interval of size $A^{-\alpha (r+3)^{C_1}}$ intersects $\Gamma_r\setminus \widetilde{\Gamma}_{r}$, we keep it; otherwise, we delete it.  We denote by $\Lambda_{r+1}$ the set of the remaining intervals. By Lemma \ref{qvyvraodong},  $\Lambda_{r+1}$ satisfies the requirement.

Moreover, we have
\begin{equation}
    mes_{b}(\mathop{\cup}\limits_{I\in\Lambda_{r}^{(2)}} (\Gamma_{r}\cap I)\setminus\mathop{\cup}\limits_{I\in\Lambda_{r+1}} (\Gamma_{r}\cap I) )<A^{-r/6},
\end{equation}
which implies
\begin{align}
    &mes_{b}(\mathop{\cup}\limits_{I\in\Lambda_{r}} (\Gamma_{r}\cap I)\setminus\mathop{\cup}\limits_{I\in\Lambda_{r+1}} (\Gamma_{r}\cap  I) )\notag\\
    <& mes_{b}(\mathop{\cup}\limits_{I\in\Lambda_{r}^{(2)}} (\Gamma_{r}\cap I)\setminus\mathop{\cup}\limits_{I\in\Lambda_{r+1}} (\Gamma_{r}\cap I) )
    +mes_{b}(\mathop{\cup}\limits_{I\in\Lambda_{r}} (\Gamma_{r}\cap I)\setminus\mathop{\cup}\limits_{I\in\Lambda_{r}^{(2)}} (\Gamma_{r}\cap I) )\notag \\
    <&A^{-r/6}+ [\exp\exp(\log r)^{1/4}]^{-1}\notag\\
    \lesssim&[\exp\exp(\log r)^{1/4}]^{-1}.
\end{align}
Thus, we have proved the Iterative Lemma.

\section{Proof of Theorem \ref{zhudingli} }

Let $\lambda'=\lambda+\varepsilon \varphi$ be the limit of $\lambda_{r}'=\lambda+\varepsilon \varphi_{r}$. Thus, $\lambda'=\lambda+\varepsilon\varphi$ is a $C^{1}$ function defined on $I_0$, and $\left\lVert \varphi-\varphi_{r} \right\rVert_{C(I_0)}<\epsilon_{r-1}^{1-}$. Denote by $\Gamma$ the graph of $\lambda'=\lambda+\varepsilon\varphi$.  Since $\epsilon_{r-1}^{1-}$ is negligible compared with $A^{-\alpha(r+2)^{C_{1}}}$ and $\lVert\partial_{\lambda}(\lambda+\varepsilon \varphi_{r})\rVert_{C^{0}(I_0)}>\frac{1}{2}$,   we have
\begin{align}
    &\text{mes}_{b}\left( \Gamma \cap\left( \mathop{\cup}\limits_{I'\in \Lambda_{r-1}} I'\setminus\mathop{\cup}\limits_{I\in\Lambda_{r}}I \right) \right)\notag\\
    <&[\exp\exp(\log r)^{1/4}]^{-1}+\epsilon_{r-1}^{1-}\cdot A^{\alpha(r+2)^{C_{1}}}\notag\\
    <&2[\exp\exp(\log r)^{1/4}]^{-1},
\end{align}
for $r \gtrsim \log\log\frac{1}{\varepsilon}$, and
\begin{equation}
    \text{mes}_{b}\left( \Gamma \cap\left( I_0\times I_0\setminus\mathop{\cup}\limits_{I\in\Lambda_{r}}I \right) \right)<(\log\log\frac{1}{\varepsilon})^{-\frac{1}{2}}
\end{equation}
for  $r\lesssim \log\log\frac{1}{\varepsilon}$.

 Let $I^{(\infty)}=\mathop{\cap}\limits_{r=1}^{\infty}(\mathop{\cup}\limits_{I\in\Lambda_{r}}I)$, then $I^{(\infty)}$ is a Cantor-type subset of $I_0\times I_0$, and we have
\begin{align}
    &mes_{b} \left( \Gamma\cap (I_0\times I_0 \setminus I^{(\infty)}) \right)\notag\\
    <& (\exp\exp(\log\log\log\frac{1}{\varepsilon})^{1/3})^{-1}+(\log\log\frac{1}{\varepsilon})^{-\frac{1}{2}}\notag\\
    <&2(\log\log\frac{1}{\varepsilon})^{-\frac{1}{2}}.
\end{align}

Let $q^{(\infty)}(\lambda,\lambda')$ be the limit of $q_r(\lambda,\lambda')$. Then for $(\lambda,\lambda')\in I^{(\infty)}$, we have $q^{(\infty)}(\lambda,\lambda')\in G^{\alpha,\frac{L_0}{2}}(\mathbb{T}^{b+d})$. Moreover,   
for  $(\lambda,\lambda')\in \Gamma\cap I^{(\infty)}$, we have
\begin{equation}
\begin{split}
    &\left\lVert\frac{1}{i}\lambda'\cdot \partial_{\theta}q^{(\infty)}-\mathcal{L} q^{(\infty)}-\varepsilon \frac{\partial H}{\partial \bar{q}}(q^{(\infty)},\bar{q}^{(\infty)})\right\rVert\\
    \le
    &\left\lVert\frac{1}{i}\lambda'\cdot \partial_{\theta}q^{(\infty)}-\mathcal{L} q^{(\infty)}-\varepsilon \frac{\partial H}{\partial \bar{q}}(q^{(\infty)},\bar{q}^{(\infty)})-(\frac{1}{i}\lambda_{r}'\cdot \partial_{\theta}q_{r}-\mathcal{L}q_r-\varepsilon\frac{\partial H}{\partial \bar{q}}(q_r,\bar{q_r}))\right\rVert\\
    & +\left\lVert\frac{1}{i}\lambda_{r}'\cdot \partial_{\theta}q_{r}-\mathcal{L}q_r-\varepsilon\frac{\partial H}{\partial \bar{q}}(q_r,\bar{q_r})\right\rVert\\
    \le& \left\lVert (\lambda'-\lambda_r')\cdot\partial_{\theta}q^{(\infty)} \right\rVert+\left\lVert \lambda_r'\cdot(\partial_{\theta}q^{(\infty)}-\partial_{\theta}q_{r}) \right\rVert+\left\lVert \mathcal{L}q^{(\infty)}-\mathcal{L}q_{r} \right\rVert\\
    &+\left\lVert \varepsilon \frac{\partial H}{\partial \bar{q}}(q^{(\infty)},\bar{q}^{(\infty)})-\varepsilon\frac{\partial H}{\partial \bar{q}}(q_r,\bar{q_r})) \right\rVert+\left\lVert\frac{1}{i}\lambda_{r}'\cdot \partial_{\theta}q_{r}-\mathcal{L}q_r-\varepsilon\frac{\partial H}{\partial \bar{q}}(q_r,\bar{q_r})\right\rVert\\
    <&\epsilon_{r-1}^{1-}.
\end{split}
\end{equation}
Noting that $r$ is arbitrary,  we have 
\begin{equation}
    \left\lVert\frac{1}{i}\lambda'\cdot \partial_{\theta}q^{(\infty)}-\mathcal{L} q^{(\infty)}-\varepsilon \frac{\partial H}{\partial \bar{q}}(q^{(\infty)},\bar{q}^{(\infty)})\right\rVert=0.
\end{equation}
Thus, we have proved Theorem \ref{zhudingli}.

\appendix

\section{A technical proposition about the reality of Fourier coefficients}

\begin{proposition}[\textbf{Reality}]
    Let $f\in C^{1}([-1,1])$ be real-valued, and let $q$ be a function defined on $\mathbb{T}^n$, with $\hat{q}(k)\in \mathbb{R}$ for all $k\in \mathbb{Z}^n$. Suppose $\left\lVert q \right\rVert_{C^{0}(\mathbb{T}^{n})}$ is small enough, then $(f(\left\lvert q\right\rvert^2))^{\land} (k)\in \mathbb{R} $ for all $k\in \mathbb{Z}^n$.
    \label{p1}
\end{proposition}
\begin{proof}
    Since $\hat{q}(k)\in \mathbb{R}$, we know that $(q\bar{q})^{\land}(k)=(\hat{q}*\hat{\bar{q}})(k)\in \mathbb{R}$.
    Let $\chi$ be a smooth bump function such that $0\le \chi\le 1$, $\chi=1$ on $[-\frac{1}{4},\frac{1}{4}]$ and $\chi=0$ outside $[-\frac{1}{2},\frac{1}{2}]$.  If $\left\lVert q \right\rVert_{C^{0}(\mathbb{T}^{n})}$ is small enough, then $(f\cdot \chi) (\left\lvert q\right\rvert^2)=f(\left\lvert q\right\rvert^2)$. So without loss of generality, we can assume $f$ is periodic with period $1$. Since $f$ is real-valued, it is easy to check that the Fourier series of $f$ has the following form:
$$
f(x)=\sum\limits_{j\in\mathbb{Z}}l_j \cos 2\pi jx+m_j \sin2\pi jx,
$$
where $l_j, m_j(j\in\mathbb{Z})$ are real.
Let $f_n(x)=\sum\limits_{\left\lvert j\right\rvert \le n}l_j \cos 2\pi jx+m_j\sin 2\pi jx$. Since $f\in C^{1}([-\frac{1}{2},\frac{1}{2}])$, $f_n$ converges to $f$ uniformly.

For any $k\in \mathbb{Z}^n$, we have 
\begin{equation}
\begin{split}
    \left\lvert (f(\left\lvert q\right\rvert^2))^{\land} (k)-(f_n(\left\lvert q\right\rvert^2))^{\land} (k)\right\rvert &=\left\lvert\int_{\mathbb{T}^n}(f(\left\lvert q\right\rvert^2)-f_n(\left\lvert q\right\rvert^2))e^{-2\pi i k x}dx \right\rvert\\
    &\le \int \left\lvert f(\left\lvert q\right\rvert^2)-f_n(\left\lvert q\right\rvert^2) \right\rvert dx\\
    &\le \mathop{\textup{sup}}\limits_{x\in \mathbb{T}^n}\left\lvert f(\left\lvert q(x)\right\rvert^2)-f_n(\left\lvert q(x)\right\rvert^2) \right\rvert\\
    &\le \mathop{\textup{sup}}\limits_{s\in[-\frac{1}{2},\frac{1}{2}]}\left\lvert f(s)-f_n(s) \right\rvert.
\end{split}
\nonumber
\end{equation}
Thus, $(f_n(\left\lvert q\right\rvert^2))^{\land} (k)$ converges to $(f(\left\lvert q\right\rvert^2))^{\land} (k)$ as $n\to+\infty$ for all $k$. Therefore, it suffices to show that $(f_n(\left\lvert q\right\rvert^2))^{\land} (k)\in \mathbb{R}$ for all $k$. 

To prove $(f_n(\left\lvert q\right\rvert^2))^{\land} (k)\in \mathbb{R}$ for all $k$, it suffices to show that $(\cos2\pi j |q|^{2})^{\land}(k)\in \mathbb{R}$ and $(\sin 2\pi j |q|^{2})^{\land}(k)\in \mathbb{R} $ for all $k$.
Take $\cos2\pi j|q|^{2} $ for example. We have
\begin{equation}
    \cos2\pi j |q|^{2}=\sum\limits_{n\ge0}\frac{(-1)^{n}(2\pi j |q|^{2})^{2n}}{(2n)!}.
\end{equation}
Since $(|q|^{4n})^{\land}(k)\in \mathbb{R}$ for all $n$ and $k$, we have $(\cos2\pi j |q|^{2})^{\land}(k)\in \mathbb{R}$ for all $k$. Thus, the proof is completed.
\end{proof}

\section{Diophantine condition}

\begin{lemma}
    Let $I\subset \mathbb{R}^{b}$ be an interval of size $1$, and let $M, N$ be large integers, and $\gamma$ be a small constant. Furthermore, suppose that $N^{-1}<\gamma$. Then, there exists a family of disjoint intervals $\{I_{\zeta}\}_{\zeta\in J}$ of size $N^{-C}$ (where $C$ is a constant depending on $\tau$ and $b$) such that: 
    \begin{itemize}
        \item[(i)] $\mathop{\cup}\limits_{\zeta\in J} I_{\zeta}\subset I$.
        \item[(ii)] For $\lambda\in I\setminus (\mathop{\cup}_{\zeta\in J} I_{\zeta})$, we have
        \begin{equation}
            |k\cdot \lambda|>\gamma (1+|k|)^{-\tau},\  M<|k|<N,
            \notag
        \end{equation}
        where $\tau>b$.
        \item[(iii)] $\textup{mes } \mathop{\cup}\limits_{\zeta\in J} I_{\zeta} \lesssim \gamma M^{-\tau+1}$.
            
    \end{itemize}
    \label{diufantudaxiao}
\end{lemma}

\begin{proof}
    Partition $I$ into intervals of size $N^{-C}$. Consider an interval $I'$ from this partition.  If there exists $\lambda\in I'$ such that
    \begin{equation}
        |k\cdot \lambda|>2\gamma (1+|k|)^{-\tau},\ \  M<|k|<N,
            \notag
    \end{equation}
     we delete $I'$. Thus, for the deleted $I'$,  we have 
    \begin{equation}
            |k\cdot \lambda'|>\gamma (1+|k|)^{-\tau},\ \  M<|k|<N,
            \notag
        \end{equation}
     for $\lambda'\in I'$.
     The remaining $N^{-C}$-intervals form the collection $\{I_{\zeta}\}_{\zeta\in J}$,  which satisfies (i) and (ii). 
     
     For  $\lambda\in I_{\zeta}$,
     we have 
     \begin{equation}
         |k\cdot \lambda|\le 2\gamma (1+|k|)^{-\tau},
         \notag
     \end{equation}
    for some $M<|k|<N$. 
    Thus, the measure estimate follows easily.
   This completes the proof.\qedhere
\end{proof}

\section{A matrix-valued Cartan-type theorem}

\begin{proposition}
    Let $A(\sigma)$ be a self-adjoint $N\times N$ matrix function of a real parameter $\sigma\in[-\delta,\delta]$, satisfying the following conditions:
    \begin{itemize}
        \item[(i)] $A(\sigma)$ is real analytic in $\sigma$, and there exists a holomorphic extension to a strip
        \begin{equation}
            |Re z|<\delta,\ \   |Im z|<\gamma
            \label{14.2}
        \end{equation}
        satisfying
        \begin{equation}
            \left\lVert A(z) \right\rVert<B_1.
            \label{14.3}
        \end{equation}
        \item[(ii)] For each $\sigma\in[-\delta,\delta]$, there exists a subset $\Lambda\subset[1,N]$ such that
        \begin{equation}
            |\Lambda|<M
            \label{14.4}
        \end{equation}
        and 
        \begin{equation}
            \left\lVert (R_{[1,N]\setminus\Lambda}A(\sigma)R_{[1,N]\setminus\Lambda})^{-1} \right\rVert<B_{2}.
            \label{14.5}
        \end{equation}
        \item[(iii)] 
        \begin{equation}
            mes \left\{\sigma\in[-\delta,\delta]: \left\lVert A(\sigma)^{-1} \right\rVert>B_3  \right\}<10^{-3}\gamma (1+B_1)^{-1}(1+B_2)^{-1}.
            \label{14.6}
        \end{equation}
    \end{itemize}
    Then, letting
    \begin{equation}
        \iota<(1+B_{1}+B_{2})^{-10M},
        \notag
    \end{equation}
    we have 
    \begin{equation}
        mes \left\{ \sigma\in\left[ -\frac{\delta}{2},\frac{\delta}{2}: \left\lVert A(\sigma)^{-1} \right\rVert >\frac{1}{\iota}\right] \right\}<e^{-\frac{c\log \iota^{-1}}{ M\log(M+B_1+B_2+B_3) }}.
        \label{14.7}
    \end{equation}
    \label{Cartanestimate}
\end{proposition}
(see \cite{bourgain2005green} for a proof).

\section{Some facts about semi-algebraic sets}

\begin{definition}
    A set $S\subset \mathbb{R}^{n}$ is called semi-algebraic if it is a finite union of sets defined by a finite number of polynomial equalities and inequalities. More specifically, let $\mathcal{P}=\{ P_1,..., P_{s}\}\subset\mathbb{R}[X_1,...,X_n]$ be a family of $s$ real polynomials of degree bounded by $d$. A (closed) semi-algebraic set $S$ is given by an expression
    \begin{equation}
        S=\mathop{\cup}\limits_{j}\mathop{\cap}\limits_{l\in\mathcal{L}_{j}}\{P_{l}s_{jl}0\}
        \label{biaodashibandaishu}
    \end{equation}
    where $\mathcal{L}_{j}\subset\{1,...,s\}$ and $s_{jl}\in\{\ge, \le, =\}$ are arbitrary. We say that $S$ has degree at most $sd$, and its degree is the minimum $sd$ over all representations (\ref{biaodashibandaishu}) of $S$.
\end{definition}

Recall the following properties:

\begin{lemma}
    The number of connected components of $S$ in (\ref{biaodashibandaishu}) is bounded by $s^{n}(O(d))^{n}<CB^{C}$, where $C$ denotes constants depending only on $n$.
\end{lemma}

\begin{lemma}
    Let $S\subset[0,1]^{d=d_1+d_2}$ be a semi-algebraic set of degree $B$ and $mes_{d} S<\eta$, where 
    \begin{equation}
        \log B\ll\log \eta.
    \end{equation}
    Denote by $(x,y)\in[0,1]^{d_1}\times [0,1]^{d_2}$ the product variable. Fix $\tilde{\varepsilon}>\eta^{1/d}$. There is a decomposition of $S$ as
    \begin{equation}
        S=S_1\cup S_2
        \notag
    \end{equation}
    with the following properties:
    \begin{equation}
        mes_{d_1} (Proj_{x} S_{1})<B^{C}\tilde{\varepsilon}
    \end{equation}
    and $S_2$ has the transversality property:
    \begin{equation}
        mes_{d_2} (S_2 \cap L)<B^{C}\tilde{\varepsilon}^{-1}\eta^{1/d}
    \end{equation}
    whenever $L$ is a $d_2$-dimensional hyperplane in $[0,1]^{d_1+d_2}$ such that
    \begin{equation}
        \max\limits_{1\le j\le d_1} |Proj_{L}(e_j)|<\tilde{\varepsilon}.
        \label{xielvhaodao}
    \end{equation}
    \label{fenjiedingli}
\end{lemma}

See \cite{bourgain2007anderson} for the proof and more properties.

\section{The multiscale reasoning}
\begin{proof}[The multiscale reasoning of Lemma \ref{19.13}]
    Let $I$ be an interval in $\Lambda_{l}^{(1)}$, and suppose that $(\lambda,\lambda')\in I$. Let $M_1=e^{M_0^{\frac{1}{2}c_2}}$ (where $c_2$ will be specified). 
Let $Q\subset \mathbb{Z}^d$ be an $M_1$-interval such that $\min\limits_{n\in Q} |n| >M_1$. Define $Q_0=[-M_0,M_0]^d$, and consider the matrices $T_{M_0,n_0+Q_0}^{\sigma+k \cdot \lambda'}$, where  $|k|<M_1$ and $ n_0\in Q$. Choose $B=M_1^{\rho} (0<\rho<1)$ in Lemma \ref{jihe}, and let $\{\pi_{\zeta}\}$ be the corresponding partition of $\mathbb{Z}^d$. 

We call a site $(\pm,n,k)$  singular if $|D_{\pm,n,k}^{\sigma}|<M_0^{C}$, where $C>2(b+d)+2$.

Fix $\sigma$ and $(\lambda,\lambda')\in I$.
Suppose that $(\pm,n,k)$ and $(\pm, n',k')$ are singular, with $n\in\pi_{\zeta}$, $n\in\pi_{\zeta'}$, $\zeta\ne\zeta'$. It then follows that
\begin{equation}
    |n-n'|+C'|k-k'|>M_1^{\rho}-2M_0^{C},
\end{equation}
where $C'$ depends on the parameter set  $I_0\times I_0$ and the resonant set. Therefore, we have
\begin{equation}
    |n-n'|+|k-k'|>c' M_{1}^{\rho},
\end{equation}
where $c'$ is a constant depending on $b,d,C'$.

Thus, if we denote 
\begin{equation}
    \mathcal{S}=\{ (n,k)\in Q\times [-M_1,M_1]^b \big| (+,n,k) \text{ or } (-,n,k) \text{ is singular}  \},
\end{equation}
the following separation property holds:
\begin{equation}
    \text{dist} (\mathcal{S}\cap (\pi_{\zeta}\times \mathbb{Z}^b),\mathcal{S}\cap (\pi_{\zeta'}\times \mathbb{Z}^b) ) \ge c M_1^{\rho} \textup{ for } \zeta\ne\zeta'.
\end{equation}
Fix $n_0\in Q$. According to the induction hypothesis, $T_{M_0,n_0+Q_0}^{\sigma+k\cdot \lambda'}$ will satisfy (\ref{good1}) and (\ref{good2}), unless 
\begin{equation}
    \min\limits_{s} |\sigma+k\cdot \lambda'-\sigma_s(\lambda,\lambda')|<e^{-M_0^{c_2}}.
\end{equation}
Since $(\lambda,\lambda')\in I$ satisfies $DC_{N_l}$, it follows that for a fixed $s$, the inequality $|\sigma+k\cdot \lambda'-\sigma_s(\lambda,\lambda')|<e^{-M_{0}^{c_2}}$ holds for at most one value of $k$ with $|k|\le M_1$. Thus, for a fixed $\zeta$, there are at most 
\begin{equation}
    M_1^{\rho C_0\cdot d} e^{(\log M_0)^{C_7}}<M_1^{\rho C_0 d+}
\end{equation}
sites $(n,k)\in \pi_{\zeta}\times [-M_1,M_1]$ for which $G_{M_0,n+Q_0}^{k\cdot\lambda'+\sigma}$ fails to satisfy   (\ref{good1}) or (\ref{good2}). We refer to these intervals $(k+[-M_0,M_0]^b)\times (n+Q_0)$ as bad intervals. For these $(n,k)$, it is necessary that 
\begin{equation}
    ((n+Q_0)\times(k+[-M_0,M_0]^b))\cap \mathcal{S}\ne \varnothing.
\end{equation}
Otherwise, we can use the Neumann argument to derive a contradiction. 
A site $(n,k)$ is called good  if   $G_{M_0,n+Q_0}^{k\cdot\lambda'+\sigma}$  satisfies both  (\ref{good1}) and (\ref{good2}).

Let $\Omega=[-M_1,M_1]^b\times Q$.
Partition $\Omega$ into $M_0$-intervals $\{P\}$. If $P$ does not intersect with any  bad interval, we put it in $\Omega_0$. If $P$ and $P'$ both intersect with  bad intervals and the distance between them is less than $M_1^{\rho}$, we put them together, and  denote $\mathcal{P}_{\beta}=\{P, P'\}$. If $P''$ intersects with  bad intervals and $\min\limits_{P\in\mathcal{P}_{\beta}}\textup{dist }(P'', P)<M_{1}^{\rho}$, we put $P''$ in $\mathcal{P}_{\beta}$. In this way, we obtain $\{\mathcal{P}_{\beta}\}$. Denote
\begin{equation}
\Omega_{1,\beta}=\mathop{\cup}\limits_{P\in\mathcal{P}_{\beta}} P.
\end{equation}

The procedures outlined above allow us to get a decomposition of $\Omega$:
\begin{equation}
    \Omega=\Omega_{0}\cup\Omega_1,
\end{equation}
 where $\Omega_1=\mathop{\cup}\limits_{\beta}\Omega_{1,\beta}$ and both $\Omega_{0}$ and $\Omega_{1,\beta}$ are unions of $M_0$-intervals. The set $\Omega_0$ is made up of good sites, and 
\begin{align}
    &\text{diam }\Omega_{1,\beta}<M_{1}^{2\rho C_0 d},\\
     &\text{dist }(\Omega_{1,\beta},\Omega_{1,\beta'})>M_1^{\rho}, \ \ \beta\ne\beta'.
\end{align}

 Since we have 
\begin{equation}
    \left\lVert \partial S_{l}\right\rVert<C,
\end{equation}
the good sites allow an $O(M_1^{-1} e^{-M_0})$ perturbation of $(\lambda,\lambda',\sigma)$. Consequently, the decomposition $\Omega=\Omega_0\cup\Omega_1$ can be applied within an $M_1^{-1} e^{-2 M_0}$-neighborhood of an initial parameter choice $(\lambda,\lambda',\sigma)$.  Let $\widetilde{\Omega}_{1,\beta}$ denote an $M_{1}^{\frac{3}{4}\rho}$-neighborhood of $\Omega_{1,\beta}$. If we can ensure that
\begin{equation}
    \left\lVert G_{\widetilde{\Omega}_{1,\beta}}^{\sigma} \right\rVert<e^{M_1^{\frac{\rho^2}{2}}},
    \label{fanshudaxiao}
\end{equation}
then by Lemma \ref{6.1},  it follows that 
\begin{equation}
\left\{
    \begin{split}
        &\left\lVert G_{\Omega}^{\sigma} \right\rVert<e^{M_1^{\frac{3}{4}\rho^2}},\\
        &|G_{\Omega}^{\sigma}(x,x')|<e^{-\alpha L_{l+1}^{(0.3)}|x-x'|_{\alpha}},\ \  |x-x'|>M_1^{\theta},
    \end{split}\right.
    \label{fanduanyijv}
\end{equation}
as long as $c_2$ is sufficiently small such that
\begin{equation}
    \frac{2}{c_2}(\max\{\frac{c_2}{2},\kappa,\theta c\}-c)<-3.
\end{equation}
If we set $\kappa=\frac{3}{4}\rho^2$, then (\ref{fanduanyijv}) implies (\ref{good1}) and (\ref{good2}). Therefore, we only need to ensure that (\ref{fanshudaxiao}) holds. 

Now, we consider (\ref{fanshudaxiao}). Recall that for $(k,n)\in \Omega$, 
\begin{equation}
    D_{\pm,n,k}^{\sigma}=\pm (k\cdot \lambda' +\sigma)-|n|^2=\pm\sigma -|n|^2+O(M_1)
\end{equation}
with $n\in Q$, and hence $|n|>M_1$. Without loss of generality, we can assume that $\sigma>0$. Therefore, 
$$|-\sigma-|n|^2|>M_1^2,$$ 
and it follows that
\begin{equation}
    \left\lVert (R_{-} T_{\widetilde{\Omega}_{1,\beta}}^{\sigma}R_{-})^{-1} \right\rVert\lesssim M_1^{-2}.
\end{equation}
Furthermore, $G_{\widetilde{\Omega}_{1,\beta}}^{\sigma}$ is controlled by the inverse of the self-adjoint matrix (see \cite{bourgain2002anderson} for example): 
\begin{equation}
    \begin{split}
        &R_{+}T_{\widetilde{\Omega}_{1,\beta}}^{\sigma}R_{+}-\varepsilon^2 R_{+} S R_{-} (R_{-} T_{\widetilde{\Omega}_{1,\beta}}^{\sigma}R_{-})^{-1} R_{-}S R_{+}\\
        &=\sigma+R_{+}T_{\widetilde{\Omega}_{1,\beta}}R_{+}-\varepsilon^2 R_{+} S R_{-} (R_{-} T_{\widetilde{\Omega}_{1,\beta}}^{\sigma}R_{-})^{-1} R_{-}S R_{+},
    \end{split}
    \label{kongz}
\end{equation}
where $S=S_{l}$.
Denote $\mathcal{T}(\lambda,\lambda',\sigma)=R_{+}T_{\widetilde{\Omega}_{1,\beta}}R_{+}-\varepsilon^2 R_{+} S R_{-} (R_{-} T_{\widetilde{\Omega}_{1,\beta}}^{\sigma}R_{-})^{-1} R_{-}S R_{+}$, then 
\begin{equation}
    \left\lVert \partial_{\lambda} \mathcal{T} \right\rVert\lesssim\left\lVert \partial_{\lambda} S_{l} \right\rVert+ \left\lVert (R_{-} T_{\widetilde{\Omega}_{1,\beta}}^{\sigma}R_{-})^{-1}\right\rVert^{2}\left\lVert \partial_{\lambda} S_{l} \right\rVert\lesssim \left\lVert \partial_{\lambda} S_{l} \right\rVert<C,
\end{equation}
\begin{equation}
    \left\lVert \partial_{\lambda'} \mathcal{T} \right\rVert\lesssim M_1+\left\lVert\partial_{\lambda'} S_{l} \right\rVert\lesssim M_1,
\end{equation}
\begin{equation}
    \left\lVert \partial_{\sigma} \mathcal{T} \right\rVert<C\varepsilon^2  \left\lVert (R_{-} T_{\widetilde{\Omega}_{1,\beta}}^{\sigma}R_{-})^{-1} \right\rVert^2\lesssim M_1^{-4}.
\end{equation}
Let $\{E_j(\lambda,\lambda',\sigma)\}$ be  the eigenvalues of $\mathcal{T}$ in ascending order. Thus $\{ \sigma+E_j(\lambda,\lambda',\sigma) \}$ is the eigenvalue of (\ref{kongz}).   The eigenvalues $\{E_j\}$ are continuous functions of the parameters $\lambda, \lambda', \sigma$, and they are piecewisely holomorphic in each one-dimensional parameter component (since the entries of $T_{l}^{\sigma}$ are polynomials in $(\lambda,\lambda',\sigma)$ over I). It follows from the first-order eigenvalue variation that $E_j (\lambda,\lambda',\sigma)$ is Lipschitz continuous, and 
\begin{equation}
    \left\lVert E_j \right\rVert_{Lip(\lambda)}<C,\ \ \left\lVert E_j \right\rVert_{Lip(\lambda')}<M_1,\ \ \left\lVert E_j \right\rVert_{Lip(\sigma)}<M_1^{-4}.
    \label{lipguji}
\end{equation}

By (\ref{lipguji}), the equation
\begin{equation}
    \sigma+E_{j}(\lambda,\lambda',\sigma)=0
\end{equation}
defines a function
\begin{equation}
    \sigma=\sigma_j(\lambda,\lambda').
\end{equation}
Moreover,  we have
\begin{equation}
    |\sigma_j(\lambda,\lambda')-\sigma_{j}(\bar{\lambda},\bar{\lambda}')|\le C|\lambda-\bar{\lambda}|+M_1|\lambda'-\bar{\lambda}'|+M_1^{-4}|\sigma_j(\lambda,\lambda')-\sigma_{j}(\bar{\lambda},\bar{\lambda}')|,
\end{equation}
which implies that $\left\lVert \sigma_j \right\rVert_{Lip(\lambda)}\le C$ and $\left\lVert \sigma_j \right\rVert_{Lip(\lambda')}<M_1$.
Furthermore, 
\begin{equation}
\begin{split}
     |\sigma+E_j(\lambda,\lambda',\sigma)|&=|\sigma-\sigma_j(\lambda,\lambda')+E_j(\lambda,\lambda',\sigma)-E_j(\lambda,\lambda',\sigma_j(\lambda,\lambda'))|\\
     &\le |\sigma-\sigma_j (\lambda,\lambda')|(1+O(M_1^{-4})).
\end{split}
\end{equation}
Consequently, 
\begin{equation}
    \text{dist}(Spec (\ref{kongz}), 0)\sim \min\limits_{j}|\sigma-\sigma_j(\lambda,\lambda')|.
\end{equation}
Hence
\begin{equation}
    \left\lVert G_{\widetilde{\Omega}_{1,\beta}}^{\sigma} \right\rVert\le \left\lVert (\sigma+\mathcal{T})^{-1} \right\rVert\le \max_{j}|\sigma-\sigma_{j}(\lambda,\lambda')|^{-1}.
    \label{zhongji}
\end{equation}
If (\ref{zhongji}) is less than $e^{M_{1}^{\frac{\rho^2}{2}}}$, then (\ref{fanshudaxiao}) holds. Therefore, we set $c_2=\frac{\rho^2}{2}$.

By collecting the functions $\{\sigma_j(\lambda,\lambda')\}$ for all $\beta$  and extending $\sigma_j$ to $I_0\times I_0$ ( see Appendix B in \cite{poschel2009lecture} for Lipschitz extension, for example),  we obtain at most $M_1^{b+d}$ Lipschitz functions 
$\{\sigma_j=\sigma_j(\lambda,\lambda')\}$ such that 
\begin{equation}
    \left\lVert \sigma_j \right\rVert_{Lip}\lesssim M_1.
\end{equation}
The reason for extending $\sigma_j$ is that the above construction is carried out in an an $M_1^{-1} e^{-2 M_0}$-neighborhood of $(\lambda,\lambda',\sigma)$. The number of these parameter neighborhoods is bounded by $M_1^{d+1}(M_1 e^{2 M_0})^{2b+1}$, and thus the total number of $\{\sigma_j \}$   is at most
\begin{equation}
    M_1^{b+d} M_1^{d+1} (M_1 e^{2M_0})^{2b+1}<e^{(4b+4) M_0}<e^{(4b+4)(\log M_1)^{\frac{2}{c_2}}}<e^{(\log M_1)^{C_7}},
\end{equation}
as long as $C_7>\frac{2}{c_2}$.

For $M_{0}>N_{l}$, this result is established via an induction on scale starting from $N_{l}$. As we have proved, the statement holds for $M_{0}\le N_{l}$. Assume the statement holds for $M<M_{0}$ and set $M_{0}=e^{M^{\frac{1}{2}c_{2}}}$. Using the same argument, and noting that the reduction in $L$ is $O(l^{-3})$,  we obtain the case for $M_{0}>N_{l}$.
\end{proof}

\begin{proof}[The proof of Lemma \ref{lemma5.4}]
    We will prove  Lemma \ref{lemma5.4} in two steps. The first step is to establish the first inequality in (\ref{heyheyhey}), and the second step is to establish the second inequality in (\ref{heyheyhey}).

\textbf{Step 1}.

Fix $\sigma$ and let $M_0<M_1<\exp\exp (\log M_0)^{\frac{1}{10}}$.  Let
\begin{equation}
M=\exp(\log\log M_1)^{3}<\exp(\log M_0)^{\frac{3}{10}}.    
\end{equation}
 Pave $[-M_1,M_1]^b\times [-10 M_1,10 M_1]^d$ by translations of $[-M,M]^b\times [-M,M]^d$. We need to control $G_{M,Q}^{\sigma+k\cdot \lambda'}$, where $Q\subset [-10 M_1,10 M_1]^d$ is an  $M$-interval and $|k|<M_1$.
 
 If $\min\limits_{n\in Q}|n|>M$, then by the assumptions, for $G_{M,Q}^{\sigma+k\cdot\lambda'}$ to satisfy (\ref{hey}), it is sufficient to ensure that
\begin{equation}
    \min\limits_{s} |\sigma+k\cdot \lambda'-\sigma_{s}(\lambda,\lambda')|>e^{-M^{c_2}},
\end{equation}
where $\{\sigma_s\}$ depends on $Q$. Collecting $\sigma_s$ over all $M$-intervals $Q\subset[-10 M_1, 10 M_1]^d\setminus [|n|<M]$, the total number of functions $\sigma_{s}$ is at most 
\begin{equation}
    (20 M_1)^d e^{(log M)^{C_7}}<M_1^{d+1}.
\end{equation}
We denote by $\{ \sigma_s \}$ the collected system. We introduce $\mathcal{A}'$ such that for $(\lambda,\lambda')\notin \mathcal{A}'$ and $M^{1+}<|k|\le M_1$, we have
\begin{equation}
    \min\limits_{s,s'}|k\cdot \lambda'-\sigma_{s}(\lambda,\lambda')+\sigma_{s'}(\lambda,\lambda')|>2 e^{-M^{c_2}}.
\end{equation}
If $(\lambda,\lambda') \in I\setminus\mathcal{A}' $, where $I\in \Lambda$, then it follows that if $\min\limits_{n\in Q\cup Q'}|n|>M$ and both $G_{M,Q}^{\sigma+k\cdot \lambda'}$ and $G_{M,Q'}^{\sigma+k'\cdot \lambda'} $  fail to satisfy (\ref{hey}), we must have $|k-k'|<M^{1+}$. Thus,  the values of $k$ corresponding to the bad intervals lie within an $M^{1+}$-neighborhood of a single $\bar{k}\in [-M_1,M_1]^b$. 

If  $G_{M,Q}^{\sigma+k\cdot \lambda'}$ fails to satisfy (\ref{hey}),  then 
\begin{equation}
    \min\limits_{|k'-k|<M,n\in Q} |\pm (k'\cdot \lambda'+\sigma)-|n|^2|<M^{C}
\end{equation}
for some constant $C$ depending only on $b$ and $d$. It follows that
\begin{equation}
    \min\limits_{n\in Q} ||n|^2-|\bar{k}\cdot \lambda+\sigma||<M^{1+}+M+M^{C}<M^{C}.
\end{equation}
Denote $R^2=|\bar{k}\cdot \lambda'+\sigma|$. The above argument implies that  there exists  $\bar{k}\in [-M_1,M_1]^b$ and $R>0$ such that
\begin{equation}
    |k-\bar{k}|+\min\limits_{n\in Q}||n|^2-R^2|<M^{C}.
    \label{jvjiha}
\end{equation}
Unless (\ref{jvjiha}) holds, $G_{M,Q}^{\sigma+k\cdot \lambda'}$ satisfies (\ref{hey}). 

Returning to the definition of $\mathcal{A}'$, there are at most 
\begin{equation}
    M_1^b \cdot (M_1^{d+1})^2=M_1^{b+2d+2}
\end{equation}
conditions involved.  Let $\varphi$ satisfy $\left\lVert \nabla \varphi \right\rVert<10^{-2}$, then 
\begin{equation}
    mes\{ \lambda| |k\cdot(\lambda+\varphi(\lambda))-\sigma_{s}(\lambda,\lambda+\varphi(\lambda))+\sigma_{s'}(\lambda,\lambda+\varphi(\lambda))|<3 e^{-M^{c_2}} \}<e^{-M^{c_2}},
\end{equation}
since 
\begin{equation}
    \begin{split}
        \left|\partial_{\lambda}[k\cdot(\lambda+\varphi(\lambda))-\sigma_{s}(\lambda,\lambda+\varphi(\lambda))+\sigma_{s'}(\lambda,\lambda+\varphi(\lambda))]\right|>M^{1+}-M\gg 1.
    \end{split}
\end{equation}
This implies that
\begin{equation}
   mes_{sec} \mathcal{A}' <M_1^{b+2d+2} e^{-M^{c_2}}<[\exp\exp(\log\log M_1)^2]^{-1}. 
\end{equation}
Clearly, $\mathcal{A}'$ may  be taken to be a union of intervals of size $e^{-M}$.

Let $(\lambda,\lambda')\in I\setminus \mathcal{A}'$, where $I\in \Lambda$. 

Consider  $G_{M^2}^{\sigma+k\cdot \lambda'}$, $|k|<M_1$. 

The set of $\sigma$ satisfying (\ref{heyhey}) with $M$ replaced by  $M^2$ can be described as a semi-algebraic set of $\sigma$ of degree less than $(CM)^{2(b+d)}$. Since $M^2\le M_0$,  (\ref{heyhey})  with $M$ replaced by  $M^2$
essentially holds outside a set of $\sigma$ with measure less than $e^{-M^{2 c_3}}$. In fact, we have a semi-algebraic set $S$ of $\sigma$ with degree less than $M^{C}$ and with measure less than $e^{-M^{2c_3}}$, such that
\begin{equation}
\left\{
    \begin{split}
        &\left\lVert G_{M^2}^{\sigma} \right\rVert<e^{2 M^{2\kappa'}},\\
        &|G_{M^2}^{\sigma} (x,x')|<e^{-L|x-x'|_{\alpha}},\ \ |x-x'|>M^{2\theta}.
    \end{split}\right.
    \label{heyheyhh}
\end{equation}
The set $S$ we choose actually  lies between the set of $\sigma$ for which (\ref{heyhey}) fails and the set of $\sigma$ for which (\ref{heyheyhh}) fails. In this article, when we say that a set has a semi-algebraic description, we always mean this, and we will not explain it again below.

Noting that $|k\cdot \lambda'|>M_1^{-C}$, it follows that $G_{M^2}^{\sigma+k\cdot\lambda'}$ fails to satisfy (\ref{heyheyhh}) for at most $M^C$ values of $k$.

Partition $\Omega=[-M_1,M_1]^b\times[-10 M_1,10 M_1]^{d}$ into $M$-interval. For $|n|>M$, if the $M$-interval centered at $(n,k)$ for some $k$ does not intersect with any bad $M$-interval, we place the interval in $\Omega_0$. For $|n|<M$, if the $M^2$-interval $([-M^{2},M^{2}]+k)\times [-10 M^2,10 M^2]$ does not intersect with any bad $M^2$ interval, we place the interval in $\Omega_0$. All remaining $M$-intervals are placed in $\Omega_1$. Thus, we obtain the decomposition
\begin{equation}
    \Omega=\Omega_0\cup\Omega_1.
\end{equation}
Denote by $\pi_1$ the projection onto the $k$-variables. Then $$\# \pi_{1}(\Omega_1)<M^C+M^C\cdot M^2\le M^C,$$ 
where $C$ depends only on $b$ and $d$.
For the $n$-variables, when $(k,n)\in \Omega_1$, we have
\begin{equation}
    |n|<M^2 \text{ or } ||\tilde{n}|^2-R^2|<M^{C} \text{ for some } |\tilde{n}-n|\lesssim M,
\end{equation}
which implies 
\begin{equation}
    |n|<M^2 \text{ or } ||n|-R|<\frac{M^{C}}{R} +M,
\end{equation}
where $C$ depends only on $b$ and $d$. 

Our purpose is to apply Proposition \ref{Cartanestimate} in  the Appendix with $A(\sigma)=T_{M_1}^{\sigma}$. Thus, we have $$B_1\le M_1^{C}.$$ Take $\Lambda=\Omega_1$, then we have 
\begin{equation}
    \# \Lambda<M^C \cdot R^{d-1+}+M^{C}.
\end{equation}
 Since $\Omega_0$ is made up of good sites, $c>\kappa$ and $c>\kappa'$, we have 
 $$B_2=2 e^{M^{2\kappa'}}+2e^{M^{\kappa}}$$
 by using  the resolvent identity.
 To obtain (\ref{14.6}) in Proposition \ref{Cartanestimate}, we simply apply the assumptions at scale $M'=\exp(\log\log M_1)^{4}<M_0$, together with the resolvent identity. We then have 
\begin{equation}
\begin{split}
    mes\{ \sigma | \left\lVert G_{M_1}^{\sigma} \right\rVert> 2e^{M'^{\kappa}}+2 e^{M'^{\kappa'}}\}
    &<M_1^{b+d} e^{(\log M')^{C_7}}e^{-M'^{c_2}}+M_1^b e^{-M'^{c_3}}\\
    &<e^{-M'^{\min \{c_2/2,c_3/2\}}}\\
    &<10^{-2}\cdot \frac{1}{1+B_1}\cdot\frac{1}{1+B_2}.
    \notag
\end{split}
\end{equation}
We make the following distinction:

\textbf{Case 1}: $R\le 2M_1^{\tilde{\kappa}}$.

\textbf{Case 2}: $R>2M_1^{\tilde{\kappa}}$, where $\tilde{\kappa}$ is a constant to be specified.

\textbf{Case 1}:

In this case, $\#\Lambda< M_1^{\tilde{\kappa}d}$. Proposition \ref{Cartanestimate} then permits us to conclude that
\begin{equation}
\begin{split}
     mes\{ \sigma | \left\lVert G_{M_1}^{\sigma} \right\rVert >e^{M_1^{\kappa'}} \}
     &<e^{-\frac{M_1^{\kappa'}}{M_{1}^{\tilde{\kappa}d}  \log(M_1^{\tilde{\kappa}d}+M_1^{C}+2 e^{M^{\kappa}}+2 e^{M^{\kappa'}}+2e^{M'^{\kappa}}+2e^{M'^{\kappa'}})}}\\
     &<e^{-M_1^{\kappa'-\tilde{\kappa}d}}.
\end{split}
\end{equation}

\textbf{Case 2}:

Apply Proposition \ref{Cartanestimate} to $A(\sigma)=T_{M_1,|n|<M_1^{\tilde{\kappa}}}^{\sigma}$, for which $\#\Lambda<M^{C}$. It follows from Proposition \ref{Cartanestimate} that
\begin{equation}
    mes\{ \sigma | \left\lVert G_{M_1, |n|<M_{1}^{\tilde{\kappa}}} \right\rVert>e^{M_{1}^{\frac{c\tilde{\kappa}}{1000}}} \}<e^{-\frac{M_1^{c\tilde{\kappa}10^{-3}}}{M^C   \log(M_1^{\tilde{\kappa}d}+M_1^{C}+2 e^{M^{\kappa}}+2 e^{M^{\kappa'}}+2e^{M'^{\kappa}}+2e^{M'^{\kappa'}})}}<e^{-M_1^{c\tilde{\kappa}\cdot 10^{-3}}}.
\end{equation}
To control $G_{M_1,|n|>M_1^{\tilde{\kappa}/10}}$, we simply use the assumptions. More precisely, define
$M_2=M_1^{\tilde{\kappa}/100}$ and cover $[|k|<M_1]\times [M_1^{\tilde{\kappa}/10}<|n|<10 M_1]$ by  intervals $P$ of size $M_2$. From the assumptions,  we have
\begin{equation}
\left\{
    \begin{split}
        &\left\lVert G_{P}^{\sigma} \right\rVert<e^{M_2^{\kappa}},\\
        &|G_{P}^{\sigma}(x,x')|<e^{-L|x-x'|_{\alpha}},\ \ |x-x'|>M_2^{\theta},
    \end{split}\right.
    \label{Phao}
\end{equation}
for $\sigma$ outside a set of measure at most
\begin{equation}
    M_1^{b+d}\cdot e^{(\log M_2)^{C_7}}\cdot e^{-M_2^{c_2}}<e^{-M_{1}^{10^{-2}\tilde{\kappa}c_2}}.
\end{equation}
Thus, for $\sigma$ outside a set of measure at most $e^{-M_1^{\tilde{\kappa}\cdot 10^{-3}}}+e^{-M_1^{10^{-2}\tilde{\kappa}c_2}}$, we have
\begin{equation}
    \begin{split}
        &\left\lVert G_{M_1,|n|<M_1^{\tilde{\kappa}}}^{\sigma} \right\rVert<e^{M_1^{10^{-3}c\tilde{\kappa}}}
    \end{split}
\end{equation}
and (\ref{Phao}). By using the resolvent identity and letting  $\tilde{\kappa}=10^{-2}d^{-1}\kappa'$,  we obtain
\begin{equation}
    \left\lVert G_{M_1}^{\sigma} \right\rVert<e^{M_1^{10^{-3}c\tilde{\kappa}}+M_1^{10^{-2}\kappa\tilde{\kappa}}}<e^{M_1^{10^{-2}\kappa'}}.
\end{equation}

Since both Case 1 and Case 2  admit an $e^{-2M}$ perturbation of $\sigma$, it follows that, outside a $\sigma$-set of measure less than
\begin{equation}
    M_{1}^{C} e^{2 M} (e^{-M_1^{c\tilde{\kappa}\cdot 10^{-3}}}+e^{-M_1^{10^{-2}\tilde{\kappa}c_2}}+e^{-M_1^{\kappa'-\tilde{\kappa}d}})<e^{-M_{1}^{\frac{1}{2}\min\{ 10^{-3}c\tilde{\kappa},10^{-2} \kappa\tilde{\kappa} \}}},
\end{equation}
 we have
\begin{equation}
    \left\lVert G_{M_1} \right\rVert<e^{M_1^{\kappa'}}.
\end{equation}
Thus, we have completed Step 1.

\textbf{Step 2}.

Define again $M_2=M_1^{10^{-2}\tilde{\kappa} }$. First, consider $M_2$-intervals $Q\subset [-10 M_1,10 M_1]^d$ such that $\min\limits_{n\in Q}|n|>M_2$. The assumptions imply that  $G_{M_2,Q}^{\sigma+k\cdot \lambda'}$ satisfies (\ref{hey}) for all $k\in [-M_1,M_1]^{b}$ and $\sigma$ outside a set of measure at most
\begin{equation}
    M_1^{b+d} e^{(\log M_2)^{C_7}} e^{-M_2^{c_2}}<e^{-\frac{1}{2} M_2^{c_2}}<e^{-\frac{1}{2}M_{1}^{10^{-2}c_2\tilde{\kappa}}}.
\end{equation}
It remains to consider the lower region $[-M_1,M_1]^{b}\times [-100 M_2,100 M_2]^{d}$.

Let $M=\exp(\log\log M_1)^5<\exp(\log M_0)^{1/2}$. 
Let P be intervals of the form   $[k-M,k+M]^{b}\times [-10 M,10 M]^{d}$, and let $P'$ be intervals of the form $[k-M,k+M]^b\times Q$, where $Q$ is an $M$-interval in $[-100 M_2,100 M_2]^{d}$ such that $\min\limits_{n\in Q}|n|>M$.

Consider a paving of $[-M_1,M_1]^{b}\times [-100 M_2,100 M_2]^{d}$ by intervals $P$ and  $P'$. For $P$-intervals, the assumptions imply $G_{P}^{\sigma}$ satisfies (\ref{heyhey}) except for at most $M^{C}$ values of $k$ (note that $M_1^{-C}>e^{-M^{c_3}}$). For $P'$ intervals, the number of bad intervals is at most 
\begin{equation}
    e^{(\log M)^{C_7}}(100 M_2)^{d}< M_1^{\frac{\tilde{\kappa}}{99}\cdot d}.
\end{equation}
Thus, the total number of bad $P$ and $P'$ in $[-M_1,M_1]^b\times [-100 M_2,100M_2]^d$ is at most $M_1^{\frac{\tilde{\kappa}}{50}\cdot d}$.

Let $M_3'=M_1^{\frac{\tilde{\kappa}}{10}\cdot d}-M_1^{\frac{\tilde{\kappa}}{20}\cdot d}$ and $M_3=M_1^{\frac{\tilde{\kappa}}{10}\cdot d}$.  If the distance in the $k$-variables between two  bad interval is less than $M_1^{\frac{\tilde{\kappa}}{15}\cdot d}$, then we can find a $\bar{k}$ such that both  bad intervals lie in the same $(\bar{k}+[-M_3'',M_3'']^b)\times [-10 M_3',10 M_3']^d$, where $M_{3}''\le M_{3}'$. For two different intervals $(\bar{k}+[-M_3'',M_3'']^b)\times [-10 M_3',10 M_3']^d$ containing bad intervals, the $k$-distance between them is more than $M_1^{\frac{\tilde{\kappa}}{15}\cdot d}$. 

Let $\Omega'=[-M_3,M_3]^b\times [-10 M_3,10 M_3]^d$, and let $\mathcal{A}$ be  the set $\mathcal{A}'$ defined in Step 1, corresponding to $M_3$. Thus, $\mathcal{A}$ is a union of intervals of size $[\exp\exp(\log\log M_3)^3]^{-1}$, such that
\begin{equation}
    mes_{sec} \mathcal{A}<[\exp\exp(\log\log M_3)^2]^{-1}.
\end{equation}
Step 1 ensures that for $(\lambda,\lambda')\in I\setminus \mathcal{A}$, where $I\in \Lambda$, and $|k|<M_1$, 
we have
\begin{equation}
    \left\lVert G_{(k,0)+\Omega'}^{\sigma} \right\rVert<e^{M_3^{\kappa'}}
\end{equation}
except for $\sigma$ in a set of measure less than 
\begin{equation}
    M_1^b\cdot e^{-M_3^{\frac{1}{2}\min\{ 10^{-3}c\tilde{\kappa},10^{-2} \kappa\tilde{\kappa} \}}}< e^{-M_3^{\frac{1}{3}\min\{ 10^{-3}c\tilde{\kappa},10^{-2} \kappa\tilde{\kappa} \}}}.
\end{equation}
Since $2\kappa'<c<\theta-\frac{\tilde{\kappa}}{10}\cdot d $,   we have by Lemma \ref{6.1} that
\begin{equation}
\left\{
    \begin{split}
        &\left\lVert G_{M_1}^{\sigma}\right\rVert<e^{M_3^{\kappa'+ }}<e^{M_1^{\kappa'}},\\
        &|G_{M_1}^{\sigma}(x,x')|<e^{-(L-O(\exp(-(\log\log  M_1)^{4})))|x-x'|_{\alpha}},\ \ |x-x'|>M_1^{\theta},
    \end{split}\right.
\end{equation}
for $\sigma$ outside a set of measure less than 
\begin{align}
        &e^{-M_3^{\frac{1}{3}\min\{ 10^{-3}c\tilde{\kappa},10^{-2} \kappa\tilde{\kappa} \}}}+e^{-\frac{1}{2}M_{1}^{10^{-2}c_2\tilde{\kappa}}}\notag\\
    < &e^{-M_1^{10^{-1}\tilde{\kappa}d\cdot \frac{1}{3}\min\{ 10^{-3}c\tilde{\kappa},10^{-2} \kappa\tilde{\kappa} \} }}+e^{-\frac{1}{2}M_{1}^{10^{-2}c_2\tilde{\kappa}}}\notag\\
    <&e^{-M_1^{ \frac{1}{4}\min\{ 10^{-4}c\tilde{\kappa}^{2}d,10^{-3} \kappa\tilde{\kappa}^{2}d \} }}.
\end{align}

Let $c_3=\frac{1}{4}\min\{ 10^{-4}c\tilde{\kappa}^{2}d,10^{-3} \kappa\tilde{\kappa}^{2}d \} $,   and we have completed the proof.  
\end{proof}

\end{document}